\input ./style/arxiv-general.cfg
\documentclass[aop,MSNbibl,seceqn,dvips]{arximspdf}
\makeatletter
   \@ifpackageloaded{graphicx}{}{\usepackage{graphicx}}
\makeatother
\usepackage{accents}
\usepackage{mathrsfs}

\doi{10.1214/14-AOP933} 
\volume{43}
\issue{5}
\pubyear{2015}
\firstpage{2282}
\lastpage{2331}

\makeatletter

\newcommand{\lvec}[1]{\accentset{\leftarrow}{#1}{}}
\renewcommand{\centerdot}{{\bolds{\cdot}}}
\newcommand{\rrvert}{\vert}
\newcommand{\llvert}{\vert}
\newtheorem{theorem}{Theorem}[section]
\newtheorem{lemma}[theorem]{Lemma}
\newtheorem{proposition}[theorem]{Proposition}
\newtheorem{corollary}[theorem]{Corollary}
\newproclaim{remark}[theorem]{Remark}
\newproclaim{definition}[theorem]{Definition}
\def\xhat{\hat x}
\makeatother

\begin{document}
\begin{frontmatter}

\title{Ratios of partition functions for\break the log-gamma polymer}
\runtitle{Log-gamma polymer}

\begin{aug}
\author[A]{\fnms{Nicos}~\snm{Georgiou}\thanksref{T1}\ead[label=e1]{n.georgiou@sussex.ac.uk}},
\author[B]{\fnms{Firas} \snm{Rassoul-Agha}\thanksref{T1}\ead[label=e2]{firas@math.utah.edu}},
\author[C]{\fnms{Timo}~\snm{Sepp\"al\"ainen}\corref{}\thanksref{T3}\ead[label=e3]{seppalai@math.wisc.edu}}
\and
\author[D]{\fnms{Atilla} \snm{Yilmaz}\thanksref{T4}\ead[label=e4]{atilla.yilmaz@boun.edu.tr}}
\runauthor{Georgiou, Rassoul-Agha, Sepp\"al\"ainen and Yilmaz}
\affiliation{University of Sussex,
University of Utah,
University of Wisconsin--Madison and Bo\u gazi\c ci University}
\address[A]{N. Georgiou\\
Department of Mathematics\\
University of Sussex\\
Falmer Campus\\
Brighton BN1 9QH\\
United Kingdom\\
\printead{e1}}
\address[B]{F. Rassoul-Agha\\
Mathematics Department\\
University of Utah\\
155S 1400E\\
Salt Lake City, Utah 84112\hspace*{8pt}\\
USA\\
\printead{e2}}
\address[C]{T. Sepp\"al\"ainen\\
Mathematics Department\\
University of Wisconsin--Madison\\
Van Vleck Hall\\
480 Lincoln Dr.\\
Madison, Wisconsin 53706-1388\\
USA\\
\printead{e3}} 
\address[D]{A. Yilmaz\\
Department of Mathematics\\
Bo\u gazi\c ci University\\
34342 Bebek, Istanbul\\
Turkey\\
\printead{e4}}
\end{aug}
\thankstext{T1}{Supported in part by NSF Grant
DMS-07-47758.}
\thankstext{T3}{Supported in part by NSF
Grants DMS-10-03651 and DMS-13-06777 and by the Wisconsin Alumni
Research Foundation.}
\thankstext{T4}{Supported in part by European Union FP7 Grant
PCIG11-GA-2012-322078.}

\received{\smonth{3} \syear{2013}}
\revised{\smonth{1} \syear{2014}}

%
\begin{abstract}
We introduce a random walk in random environment associated to an
underlying directed polymer model in $1+1$ dimensions. This walk is the
positive temperature
counterpart of the competition interface of percolation and arises as the
limit of quenched polymer measures. We prove this limit for the exactly solvable
log-gamma polymer, as a consequence of almost sure limits of ratios of
partition functions.
These limits of ratios give the Busemann functions of the log-gamma
polymer, and furnish
centered cocycles that solve a variational formula for the limiting
free energy.
Limits of ratios of point-to-point and point-to-line partition
functions manifest a
duality between tilt and velocity that comes from quenched large
deviations under
polymer measures. In the log-gamma case, we identify a family of
ergodic invariant distributions for the random walk in random environment.
\end{abstract}

%
\end{frontmatter}

\section{Introduction}\label{sec1}
In directed polymer models the definition of weak disorder is that
normalized point-to-line partition functions converge to a strictly
positive random variable. In strong disorder these normalized partition
functions converge to zero. Weak disorder takes place only in
dimensions $3+1$ and higher and under high enough temperature, while
lower dimensions are in strong disorder throughout the temperature
range; see \cite
{carm-hu-02,come-shig-yosh-03,come-shig-yosh-04,come-varg-06,come-yosh-aop-06,denholl-polymer,laco-10}
for reviews and some key results.

We work in $1+1$ dimensions with the explicitly solvable log-gamma polymer.
We show that ratios of both point-to-point partition functions
and tilted point-to-line
partition functions converge almost surely to gamma-distributed limits.
Out of this basic fact we derive several consequences.
\begin{longlist}[(ii)]
\item[(i)] Limits of ratios of partition functions give us limits of quenched
polymer measures, both point-to-point and point-to-line,
as the path length tends to infinity.
The limit processes can be
regarded as infinitely long polymers. Technically they are
random walks in correlated random environments (RWRE).
When we average over the environment, this RWRE
has fluctuation exponent $2/3$, in accordance with $1+1$ dimensional
Kardar--Parisi--Zhang universality.
This
polymer RWRE is also a
positive temperature counterpart of a competition interface
in a percolation model. (This terminology comes from the idea
that percolation models are zero-temperature polymers. Remark
\ref{hrmk2} below explains.) For the RWRE we identify a family of
stationary and ergodic distributions for the environment as seen
from the particle. The averaged stationary RWRE is a standard random
walk.

\item[(ii)] Logarithms of the limiting point-to-point
ratios give us an analogue of Busemann functions
in the positive temperature setting. Busemann functions have
emerged as a central object in the study of geodesics and
invariant distributions of percolation models and related interacting
particle systems \cite{bakh-cato-khan,cato-pime-12,damr-hans,hoff-05,newm-icm-95}. Our paper introduces this notion in the positive
temperature setting.
We show how Busemann functions solve a variational problem that
characterizes the limiting free energy density of the log-gamma
polymer.
\end{longlist}

A theme that appears more than once is a
familiar large deviations duality between
the asymptotic velocity of the path under polymer distributions
and a tilt introduced
into the partition function and probability distribution.
In this duality the mapping
from velocity to tilt is given by the expectation of
the Busemann function. In particular, this duality determines how
limits of
ratios of point-to-point
and tilted point-to-line
partition functions match up with each other.

A word of explanation about our focus on the log-gamma polymer.
The ultimate goal is of course to find results valid for a wide class
of polymer models. We could formulate at least some of our results
more generally. But the statements would be complicated and need
hypotheses that we can presently
verify only for the log-gamma model anyway. For general polymers,
just as for general percolation models, we cannot currently prove even mild
regularity properties
for the limiting free energy. Thus we
chose to focus exclusively on the log-gamma model (except for the
general discussions in Sections~\ref{seccif} and
\ref{secbuse}).

We expect that much of this picture can
eventually be verified for general $1+1$ dimensional directed polymers.
Our hope is that this paper would inspire such further work.
For example, it is clear that the solution of the variational formula
for the
free energy in terms of Busemann functions works completely
generally, once a sufficiently strong existence statement for
Busemann functions is proved. Busemann functions with
tractable distributions are an essential feature of the exact solvability
of the log-gamma polymer. They can be used to construct a \mbox{shift-}invariant
version of the polymer model, which was earlier used for
deriving
fluctuation exponents and large deviation rate functions \cite
{geor-sepp,sepp-12-aop}.

The log-gamma polymer was introduced in \cite{sepp-12-aop} and
subsequently linked
with integrable
systems and interesting combinatorics \cite
{boro-corw-mcd,corw-ocon-sepp-zygo,ocon-sepp-zygo}.
The log-gamma polymer is a canonical model in the Kardar--Parisi--Zhang
universality class, in the same vein as
the asymmetric simple exclusion process, the
corner growth model with geometric or exponential weights and the
semidiscrete polymer of O'Connell--Yor \cite{corw-rev,joha,oconn-yor-01,quas-icm,spoh-12,trac-wido-02}.
These exactly solvable models
are believed to be representative of
what should be true more generally.


\subsection*{Organization of the paper}
The paper is
essentially self-contained. One exception is that
in Section~\ref{secbuse} we cite variational formulas for
the free energy
from
\mbox{\cite{rass-sepp-p2p,rass-sepp-yilm}}. Here is an outline of the paper:
\begin{longlist}[Section~9.]
\item[Section~\ref{seccif}.] Introduction of the polymer RWRE in a general
context as the positive temperature counterpart of the competition
interface of
last-passage \mbox{percolation}.

\item[Section~\ref{seclg}.] Introduction of the log-gamma polymer. The
shift-invariant log-gamma polymer is formalized in the definition of a
gamma system of weights.

\item[Section~\ref{secratio}.] Limits of ratios of point-to-point partition
functions for the log-gamma polymer.

\item[Section~\ref{secbuse}.] Busemann functions are constructed
from limits of ratios of point-to-point partition functions and used
to solve a variational formula for the limiting free energy.
Duality between tilt and velocity.

\item[Section~\ref{secp2l}.] Limits of ratios of tilted point-to-line
partition functions for the log-gamma polymer. Duality between tilt and
velocity appears again.

\item[Section~\ref{Q-sec}.] Limits of ratios of partition functions yield
convergence of polymer measures to the polymer RWRE. The limit RWRE
has fluctuations of size $n^{2/3}$ under the averaged measure.

\item[Section~\ref{secrwre}.] A stationary, ergodic distribution for the
log-gamma polymer RWRE.

\item[Section~\ref{secapp}.] Several auxiliary results, including a large deviation
bound for the log-gamma polymer and a general ergodic theorem for cocycles.
\end{longlist}

\subsection*{Notation and conventions}
$\mathbb{N}=\{1,2,3,\ldots\}$ and $\mathbb{Z}_+=\{0,1,2,\ldots\}$.
For $n\in\mathbb{N} $, $[n]=\{1,2,\ldots,n\}$. $x\vee y=\max\{x,y\}
$ and
$x\wedge y=\min\{x,y\}$.
On $\mathbb{R} ^2$ the $\ell^1$ norm is $\llvert  x\rrvert _1=\llvert
x_1\rrvert +\llvert  x_2\rrvert $,
the inner product is $x\cdot y=x_1y_1+x_2y_2$,
and inequalities are coordinatewise: $(x_1,x_2)\le(y_1,y_2)$
if $x_r\le y_r$ for $r\in\{1,2\}$.
Standard basis vectors are $e_1=(1,0)$ and $e_2=(0,1)$.
Our random walks live in $\mathbb{Z} _+^2$ and admissible paths
$x_\centerdot=(x_k)_{k=0}^n$ have steps $z_k=x_k-x_{k-1}\in\mathcal
{R}=\{
e_1,e_2\}$.
Points of $\mathbb{Z} _+^2$ are written as $u, v, x, y$ but also as
$(m,n)$ or $(i,j)$.
Weights indexed by a single point do not have the parentheses:
if $x=(i,j)$, then $\eta_x=\eta_{i,j}$.
For $u\le v$ in $\mathbb{Z} _+^2$, $\Pi_{u,v}$ is the set of
admissible paths
from $x_0=u$
to $x_{\llvert  v-u\rrvert _1}=v$.
Limit velocities of these walks lie in the simplex
$\mathcal{U}=\{(u,1-u)\dvtx  u\in[0,1]\}$, whose (relative) interior is denoted
by $\operatorname{int}\mathcal{U}$.
Shift maps $T_v$ act on suitably indexed configurations
$w=(w_x)$ by $(T_vw)_x=w_{v+x}$.
$\mathbb{E} $ and $\mathbb{P} $ refer to the random weights (the
environment), and otherwise
$E^\mu$ denotes expectation under probability\vspace*{1pt} measure $\mu$.
The usual gamma function for $\rho>0$ is
$\Gamma(\rho)=\int_0^\infty x^{\rho-1}e^{-x} \,dx$, and the
$\operatorname{Gamma}(\rho)$
distribution on $\mathbb{R} _+$ is $\Gamma(\rho)^{-1} x^{\rho
-1}e^{-x} \,dx$.
$\Psi_0=\Gamma'/\Gamma$ and $\Psi_1=\Psi_0'$
are the digamma and trigamma
functions.

The reader should be warned that several different partition functions
appear in this paper. They\vspace*{2pt} are all denoted by $Z$ and sometimes
with additional notation such as $\check Z$. It should
be clear from the context which $Z$ is meant.
Each $Z$ is a sum of weights $W(x_\centerdot)$
of paths $x_\centerdot$ from a collection of nearest-neighbor lattice paths.
Associated to each $Z$ is a polymer probability measure $Q$
on paths, $Q\{x_\centerdot\}=Z^{-1}W(x_\centerdot)$.

\section{The polymer random walk in random environment}\label{seccif}

In this section we introduce a random walk in random environment
(RWRE) associated to an underlying directed polymer model.
This walk appears when we
look for the positive temperature counterparts of the
notions of geodesics and
competition interface that appear in last-passage percolation.
Percolation and polymers are discussed in this section in terms
of real weights, without specifying probability distributions.

\subsection{Geodesics and competition interface in last-passage percolation}
We give a quick definition of last-passage percolation, also known as
the zero-temperature polymer.
Let $\{\omega_x\dvtx  x\in\mathbb{Z} _+^2\}$ be a collection of
real-valued weights.
For $u\le v$ in $\mathbb{Z} _+^2$ let $\Pi_{u,v}$ denote\vspace*{1pt} the set of admissible
lattice paths $x_\centerdot=(x_i)_{0\le i\le n}$ with $n=\llvert
v-u\rrvert _1$
that satisfy $x_0=u$,
$x_i-x_{i-1}\in\{e_1,e_2\}$, $x_n=v$.
The last-passage times are defined by
\[
G_{u,v}=\max_{x_\centerdot\in\Pi_{u,v}} \sum
_{i=1}^{\llvert
v-u\rrvert _1} \omega_{x_i}, \qquad u\le v
\mbox{ in } \mathbb{Z} _+^2.
\]
%
A finite path $(x_i)_{0\le i\le n}$ in $\Pi_{u,v}$
is a \emph{geodesic}
between $u$ and $v$ if it is the
maximizing path that realizes $G_{u,v}$, namely,
$G_{u,v}= \sum_{i=1}^{n} \omega_{x_i} $.
Every subpath of a geodesic is also a geodesic.
Let us assume that no two paths of any length have equal sum
of weights so that maximizing paths are unique. This would almost surely
be the case, for example, if the weights are i.i.d. with a
continuous distribution.

It is convenient to construct the geodesic from $u$ to $v$ backward,
utilizing the iteration
\[
G_{u,x}= G_{u,x-e_1}\vee G_{u, x-e_2} +
\omega_x.
\]
%
Start the construction with
$x_n=v$. Suppose the segment $(x_{k},x_{k+1},\ldots, x_n)$ of the geodesic
has been
constructed. If $x_{k}>u$ coordinatewise, set
%
%
\begin{equation}
x_{k-1} =\cases{ x_k-e_1, &\quad if
$G_{u,x_k-e_1}>G_{u,x_k-e_2}$,
\vspace*{3pt}\cr
x_k-e_2, &\quad if $G_{u,x_k-e_1}<G_{u,x_k-e_2}$.} \label{geod}
\end{equation}
If $x_k\cdot e_r=u\cdot e_r$ for either $r=1$ or $r=2$, then
define the remaining segment as $(x_0,\ldots, x_k)=(u+ie_{3-r})_{0\le
i\le k}$.

For a fixed initial point $u\in\mathbb{Z} _+^2$, the \emph{geodesic spanning
tree} $\mathcal T_u$ of the lattice $u+\mathbb{Z} _+^2$ is the union
of all the
geodesics from $u$ to $v$, $v\in u+\mathbb{Z} _+^2$.

The \emph{competition interface} $\varphi=(\varphi_k)_{k\in\mathbb
{Z} _+}$ is
a lattice path on $\mathbb{Z} _+^2$
defined as a function of $\{G_{0,v}\}_{v\in\mathbb{Z} _+^2}$.
It starts at $\varphi_0=0$ and then chooses, at each step, the minimal
$G$-value,
%
%
\begin{equation}
\varphi_{k+1}= \cases{ \varphi_k+e_1, &\quad
if $G_{0,\varphi_k+e_1}<G_{0,\varphi_k+e_2}$,
\vspace*{3pt}\cr
\varphi_k+e_2,
&\quad if $G_{0,\varphi_k+e_1}>G_{0,\varphi_k+e_2}$.} \label{cif}
\end{equation}

The\vspace*{1pt} relationship between $\mathcal T_0$ and $\varphi$ is that $\varphi$
separates the two subtrees $\mathcal T_{0,e_1}, \mathcal T_{0,e_2}$ of
$\mathcal T_0$
rooted at $e_1$ and $e_2$. Since every
$\mathbb{Z} _+^2$ lattice path from $0$ has to go through either $e_1$
or $e_2$,
$\mathcal T_0=\{0\}\cup\mathcal T_{0,e_1}\cup\mathcal T_{0,e_2}$ as a
disjoint union.
For each $n\in\mathbb{Z} _+$, $\varphi_n$ is the unique point such
that $\llvert  \varphi_n\rrvert _1=n$
and for $r\in\{1,2\}$, $\{\varphi_n+ke_r\dvtx  k\in\mathbb{N} \}
\subseteq\mathcal T_{0,e_r}$.
Note that we cannot say which tree contains $\varphi_n$, unless
we know that $\varphi_n-\varphi_{n-1}=e_r$ in which case $\varphi
_n\in\mathcal T_{0,e_r}$.
If we shift $\varphi$ by $(1/2, 1/2)$, then it threads exactly between
the two
trees (Figure~\ref{cif-fig}).

The term competition interface comes from the interpretation that
$\mathcal T_{0,e_1}$ and $\mathcal T_{0,e_2}$ are two competing
clusters or infections
on the lattice \cite{ferr-mart-pime-09,ferr-pime-05}. The model can be~defined dynamically. The clusters
at time $t\in\mathbb{R} _+$ are $\mathcal T_{0,e_r}(t)=\{ v\in
\mathcal T_{0,e_r}\dvtx
G_{0,v}\le t\}$.

%
%
\begin{figure}[t]

\includegraphics{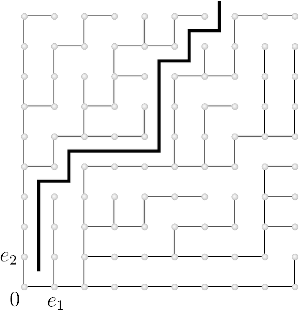}

\caption{The competition interface shifted by $(1/2, 1/2)$
(solid line) separating
the subtrees of $\mathcal T_0$ rooted at $e_1$ and $e_2$.}\label{cif-fig}
\end{figure}


\subsection{Geodesics and competition interface for a positive temperature polymer}
Let $\{V_x\}_{x\in\mathbb{Z} _+^2}$ be positive weights.
Define point-to-point
polymer partition functions
for $u\le v$ in $\mathbb{Z} _+^2$ by
%
%
\begin{equation}
Z_{u,v}=\sum_{x_\centerdot\in\Pi_{u,v}} \prod
_{i=1}^{\llvert
v-u\rrvert _1} V^{-1}_{x_i}
\label{hZ}
\end{equation}
and the polymer measure on the set of paths $\Pi_{u,v}$ by
%
%
\begin{equation}
Q_{u,v}\{x_\centerdot\} =\frac{1}{Z_{u,v}} \prod
_{i=1}^{\llvert
v-u\rrvert _1} V^{-1}_{x_i}, \qquad
x_\centerdot\in\Pi_{u,v}. \label{hQ}
\end{equation}
Our convention is to use reciprocals $V^{-1}_x$ of the weights
in the definitions. The reason is that this way
the weights in the log-gamma polymer are gamma distributed
and features of the beta-gamma algebra arise naturally.

%
%
\begin{remark}\label{hrmk2}
A conventional way of defining polymer partition
functions is
\[
Z^\beta_{u,v}=\sum_{x_\centerdot\in\Pi_{u,v}}
e^{\beta\sum_{i=1}^{\llvert  v-u\rrvert _1} \omega_{x_i}}
\]
%
with an inverse temperature parameter $0<\beta<\infty$. In the
zero-temperature
limit $\beta^{-1}\log Z^\beta_{u,v} \to G_{u,v}$ as $\beta\to
\infty$, and
the polymer measure $Q^\beta_{u,v}$ concentrates on the geodesic(s)
from $u$ to $v$. This is the sense in which last-passage percolation
is the zero-temperature polymer.
See Remark~\ref{grmk2} below for this point for the log-gamma polymer.
\end{remark}

We implement noisy versions of rules (\ref{geod}) and (\ref{cif})
to define positive temperature
counterparts of geodesics and the competition interface.

Fix a base point $u\in\mathbb{Z} _+^2$ and
define a backward Markov transition kernel $\lvec{\pi}^u$ on
the lattice $u+\mathbb{Z} _+^2$ by
$\lvec{\pi}^u(u,u)=1$, and
%
%
\begin{equation}
\qquad \lvec{\pi}^u(x,x-e_r)= \frac{V^{-1}_x Z_{u,
x-e_r}}{Z_{u, x}} =
\frac{ Z_{u, x-e_r}}{Z_{u, x-e_1}+Z_{u, x-e_2}} \qquad\mbox{for $r\in\{1,2\}$}, \label{hbackpi}
\end{equation}
if\vspace*{2pt} both $x$ and $x-e_r$ lie in $u+\mathbb{Z} _+^2$.
The middle formula above gives the correct values on the boundaries
of $u+\mathbb{Z} _+^2$
where there is only one admissible backward step,
$\lvec{\pi}^u(u+ie_r, u+(i-1)e_r)=1$ for $i\ge1$ and $r\in
\{1,2\}$.


For a path $x_\centerdot\in\Pi_{u,v}$ comparison of (\ref{hQ}) and
(\ref{hbackpi}) shows
\[
Q_{u,v}\{x_\centerdot\}=\prod_{i=1}^{\llvert  v-u\rrvert _1}
\lvec{\pi}^u(x_i,x_{i-1}).
\]
So the quenched polymer distribution $Q_{u,v}$ is the distribution of the
backward Markov chain with initial state $v$,
transition $\lvec{\pi}^u$, and absorption at $u$.
The distributions $Q_{u,v}$ are the noisy counterparts of
geodesics. The nesting property of geodesics manifests itself
through conditioning. Let $u<z<w<v$ in $\mathbb{Z} _+^2$.
Let $A_{z,w}$ be the set of paths in $\Pi_{u,v}$ that
go through the points $z$ and $w$. Given $y_\centerdot\in\Pi_{z,w}$,
let $B_{y_\centerdot}$ be the set of paths in $\Pi_{u,v}$
that traverse the path $y_\centerdot$ (i.e., contain $y_\centerdot$
as a subpath). Then
\[
Q_{u,v}(B_{y_\centerdot} \vert A_{z,w}) = Q_{z,w}
\{y_\centerdot\}.
\]
%

Define
the random geodesic spanning tree $\mathcal T_u$ rooted at $u$
by choosing, for each $x\in(u+\mathbb{Z} _+^2)\setminus\{u\}$,
a parent
%
%
\begin{equation}
\gamma(x)=\cases{ x-e_1, &\quad with probability $\lvec{
\pi}^u(x,x-e_1)$,
\vspace*{3pt}\cr
x-e_2, &\quad with
probability $\lvec{\pi}^u(x,x-e_2)$.}
\label{hparent}
\end{equation}

Now that we have the positive temperature counterparts of geodesics, we can
find the positive temperature counterpart of the competition interface
by reference to the tree $\mathcal T_0$ rooted at $0$.
Let $\mathcal T_{0,e_r}$ be the subtree rooted at $e_r$, so that
$\mathcal T_0=\{0\} \cup\mathcal T_{0,e_1}\cup\mathcal T_{0,e_2}$ as
a disjoint union.
The lemma below shows that there is
a well-defined path $X_\centerdot$ that separates the trees $\mathcal T
_{0,e_1}$ and $\mathcal T_{0,e_2}$,
and evolves as a Markov chain in the environment defined by the partition
functions. In other words, this random walk in a random environment
(RWRE) is
the positive temperature analogue of
the competition interface. The picture for $X_\centerdot$
is the same as for $\varphi$ in Figure~\ref{cif-fig}.

%
%
\begin{lemma}\label{hlm3}
\textup{(a)} Given the choices made in (\ref{hparent}),
there is a unique lattice path $(X_n)_{n\in\mathbb{Z} _+}$ with these
properties:
$X_0=0$, $X_n-X_{n-1}\in\{e_1,e_2\}$,
and for each $n$ and $r\in\{1,2\}$, $\{X_n+ke_r\dvtx  k\in\mathbb{N} \}
\subseteq
\mathcal T_{0,e_r}$.

\textup{(b)}
$X_n$ is a Markov chain with transition matrix
%
%
\begin{equation}
\pi_{x,x+e_r}= \frac
{Z^{-1}_{0,x+e_r}}{Z^{-1}_{0,x+e_1}+Z^{-1}_{0,x+e_2}}, \qquad x\in\mathbb{Z} _+^2,
r\in\{1,2\}. \label{hrwre}
\end{equation}
\end{lemma}

\begin{pf} (a) To prove the existence of the path, start with $X_0=0$,
and iterate the
following move: if $\gamma(X_n+e_1+e_2)=X_n+e_r$, set $X_{n+1}=X_n+e_{3-r}$.

(b) Given the path $(X_k)_{k=0}^n$ with $X_n=x$, we choose
$X_{n+1}=x+e_1$ if $\gamma(x+e_1+e_2)=x+e_2$ which happens with
probability
\[
\frac{Z_{0,x+e_2}}{Z_{0,x+e_1}+Z_{0,x+e_2}} = \frac
{Z^{-1}_{0,x+e_1}}{Z^{-1}_{0,x+e_1}+Z^{-1}_{0,x+e_2}}
\]
and similarly for $X_{n+1}=x+e_2$ with the complementary probability.
\end{pf}

A genuine RWRE transition probability satisfies
$\pi_{x,y}(\omega)=\pi_{0,y-x}(T_x\omega)$ for shift mappings
$(T_x)_{x\in\mathbb{Z} _+^2}$ acting on the environments $\omega$.
We augment the space of weights to achieve this.
We need to be precise about the
sets of sites on which various classes of weights are defined.

%
%
\begin{definition}
\label{hga-def}
A collection of positive real weights
\[
(\xi,\eta,\zeta, \check\xi)= \bigl\{ \xi_x, \eta_{x-e_2},
\zeta_{x-e_1}, \check\xi_{x-e_1-e_2}\dvtx  x\in \mathbb{N}^2
\bigr\}
\]
satisfies \emph{north--east} (\emph{NE}) \emph{induction} if
these equations hold for each $x\in\mathbb{N}^2$
%
%
\begin{eqnarray}
\eta_x&=&\xi_x\frac{\eta_{x-e_2}}{\eta_{x-e_2}+\zeta_{x-e_1}}, \qquad
\zeta_x=\xi_x\frac{\zeta_{x-e_1}}{\eta_{x-e_2}+\zeta_{x-e_1}}\quad\mbox{and}
\label{hNE1}
\\
\quad\check\xi_{x-e_1-e_2}&=&\eta_{x-e_2}+
\zeta _{x-e_1}. \label{hNE2}
\end{eqnarray}
\end{definition}

North--east induction simply keeps track of ratios of
partition functions. Take as given the subcollection of weights
$\{ \xi_{i,j}, \eta_{i,0}, \zeta_{0,j}\dvtx  i,j\in\mathbb{N}\}$ on
$\mathbb{Z} _+^2$,
and construct the polymer partition
functions
%
%
\begin{equation}
\qquad Z_{0,v}=\sum_{x_\centerdot\in\Pi_{0,v}} \prod
_{i=1}^{\llvert
v\rrvert _1} V^{-1}_{x_i} \qquad
\mbox{with } V_{i,j}=\cases{ \xi_{i,j}, &\quad$(i,j)\in
\mathbb{N} ^2$,
\vspace*{3pt}\cr
\eta_{i,0}, &\quad$i\in\mathbb{N} $,
$j=0$,
\vspace*{3pt}\cr
\zeta_{0,j}, &\quad$i=0, j\in\mathbb{N} $.} \label{hZ1}
\end{equation}
Then define
%
%
\begin{equation}
\eta_x =\frac{Z_{0,x-e_1}}{Z_{0, x}} \quad\mbox{and}\quad
\zeta_x= \frac{Z_{0,x-e_2}}{Z_{0, x}} \qquad\mbox{for $x\in\mathbb{N}
^2$.} 
\label{htau3}
\end{equation}
Now\vspace*{2pt} the subsystem $(\xi, \eta, \zeta)$ satisfies (\ref{hNE1}), as
can be verified by induction.
To get the full system $(\xi, \eta, \zeta, \check\xi)$ just define
$\check\xi_{x}=\eta_{x+e_1}+\zeta_{x+e_2}$ for $x\in\mathbb{Z} _+^2$.

Note that (\ref{htau3}) is valid also on the boundaries, by the
definition (\ref{hZ1}) of $Z_{0,ke_r}$ for $k\in\mathbb{N} $.
The reason for the distinct notation $\{\eta_{i,0}, \zeta_{0,j}\}$
for boundary weights in~(\ref{hZ1}) is that these are also ratios of partition
functions, just as $\eta_x$ and $\zeta_x$ in (\ref{htau3}). We shall
find that in the interesting log-gamma models, the\vspace*{2pt} boundary weights
$\{\eta_{i,0}, \zeta_{0,j}\}$ are different from the bulk weights
$\{\xi_{i,j}\}$. The role of the $\check\xi$ weights is not clear
yet, but they will become central
in the log-gamma context.

Define the space of environments
%
%
\begin{eqnarray}\label{hspace}
\Omega_{\mathrm{NE}} 
&=&\bigl\{\omega=(\xi,\eta,\zeta, \check\xi)\in
\mathbb{R} _+^{\mathbb{N} ^2+(\mathbb{N} \times\mathbb{Z}
_+)+(\mathbb{Z} _+\times\mathbb{N} )+\mathbb{Z} _+^2}\dvtx
\nonumber\\[-8pt]\\[-8pt]
&&\hspace*{51pt}\mbox{$(\xi,\eta,\zeta, \check\xi)$ satisfies NE induction} \bigr\}.\nonumber
\end{eqnarray}
Translations act via
$T_z\omega=(\xi_{z+\mathbb{N} ^2}, \eta_{z+\mathbb{N} \times
\mathbb{Z} _+},\zeta_{z+\mathbb{Z}
_+\times\mathbb{N} }, \check\xi_{z+\mathbb{Z} _+^2})$ for $z\in
\mathbb{Z} _+^2$,
where we introduced notation $\xi_{z+\mathbb{N} ^2}=\{\xi_{z+x}\}
_{x\in\mathbb{N}
^2}$, and
similarly for the other configurations.

%
%
\begin{definition} The \emph{polymer random walk in random environment}
is a RWRE with environment space $\Omega_{\mathrm{NE}}$
and transition probability
%
%
\begin{equation}
\qquad \pi_{x, x+e_1}(\omega)=\frac{\eta_{x+e_1}}{\eta
_{x+e_1}+\zeta_{x+e_2}} \quad\mbox{and}\quad
\pi_{x, x+e_2}(\omega)=\frac{\zeta_{x+e_2}}{\eta_{x+e_1}+\zeta_{x+e_2}}. \label{hrwre1}
\end{equation}
\end{definition}

This definition is the same as (\ref{hrwre}) with partition functions
(\ref{hZ1}). The quenched path probabilities
$P^\omega$ of this RWRE
started at $x_0=0$ are defined by
%
%
\begin{equation}
P^\omega(X_0=0, X_1=x_1, \ldots,
X_n=x_n) = \prod_{k=1}^n
\pi_{x_{k-1},
x_k}(\omega). \label{hrwre2}
\end{equation}
Distributions of $X_n$ are again related to polymer distributions
and partition functions. Define
\[
\check Z_{0,v} = \sum_{x_\centerdot\in\Pi_{0,v}} \prod
_{i=0}^{\llvert  v\rrvert _1-1} \check\xi^{-1}_{x_i},
\qquad v\in\mathbb{Z} _+^2.
\]
%
In contrast with (\ref{hZ1}), this time the weight at the origin is
included but
the weight at $v$ excluded. From (\ref{hNE2}) and (\ref{hrwre2}) we
derive two
formulas. First, the distribution of $X_n$ is a ratio of partition functions
\[
P^\omega(X_n=x)= \frac{\check Z_{0,x}}{Z_{0,x}} \qquad\mbox{for $x\in
\mathbb{Z} _+^2$ such that $\llvert x\rrvert _1=n$.}
\]
%
Then if the walk is conditioned to go through a point, the distribution
of the path segment is the polymer probability in $\check\xi$ weights:
for $x_\centerdot=(x_k)_{k=0}^n\in\Pi_{0,x_n}$,
\begin{eqnarray*}
&& P^\omega(X_0=0, X_1=x_1, \ldots,
X_n=x_n \vert X_n=x_n)
\\
&&\qquad = \frac{1}{\check Z_{0,x_n}} \prod
_{i=0}^{n-1} \check\xi^{-1}_{x_i}
= \check Q_{0,x_n}\{x_\centerdot\}.
\end{eqnarray*}
%

\section{The log-gamma polymer}\label{seclg}
This section gives a quick definition of the log-gamma polymer
and its Burke property.
Let $0<\lambda<\rho<\infty$.

%
%
\begin{definition}\label{ga-def}
A collection
$(\xi, \eta, \zeta, \check\xi)=
\{ \xi_x, \eta_{x-e_2}, \zeta_{x-e_1}, \check\xi_{x-e_1-e_2}\dvtx  x\in
\mathbb{N}^2\}$
of positive real random variables
is a \emph{gamma system} of weights with parameters $(\lambda, \rho)$ if
the following three properties hold:
\begin{longlist}[(a)]
\item[(a)] NE induction (Definition~\ref{hga-def})
holds:
for each $x\in\mathbb{N}^2$, almost surely,
%
%
\begin{eqnarray}\label{NE}
\eta_x&=&\xi_x\frac{\eta_{x-e_2}}{\eta_{x-e_2}+\zeta_{x-e_1}}, \qquad
\zeta_x=\xi_x\frac{\zeta_{x-e_1}}{\eta_{x-e_2}+\zeta_{x-e_1}}\quad\mbox{and}
\nonumber\\[-8pt]\\[-8pt]
\check\xi_{x-e_1-e_2}&=&\eta_{x-e_2}+
\zeta _{x-e_1}.\nonumber 
\end{eqnarray}

\item[(b)]
The marginal distributions of the variables are
%
%
\begin{equation} \label{ga-s}
\eta_{x}\sim \operatorname{Gamma}(\lambda),\qquad \zeta_{x}\sim
\operatorname{Gamma}(\rho-\lambda)\quad\mbox{and}\quad \xi_{x}, \check\xi_{x} \sim
\operatorname{Gamma}(\rho).\hspace*{-30pt}
\end{equation}

\item[(c)]
The variables
$\{\xi_{i,j}, \eta_{i,0},\zeta_{0,j}\dvtx  i,j\in\mathbb{N}\}$ are mutually
independent.
\end{longlist}

A triple
$(\xi, \eta, \zeta)$ is a gamma system with
parameters $(\lambda, \rho)$ if conditions (a)--(c) are
satisfied without the conditions on $\check\xi$.
\end{definition}

Note that the $\xi$-weights are defined only in the bulk $\mathbb
{N}^2$, while
the $\check\xi$-weights are defined also on the boundaries and at the
origin.
Variables $\eta_x$ and $\zeta_x$ can be thought of as weights on
the edges, while $\xi_x$ and $\check\xi_x$ are weights on the vertices.
Edge weights will also be denoted by
%
%
\begin{equation}
\tau_{x-e_1,x} =\eta_x \quad\mbox{and}\quad
\tau_{x-e_2,x} =\zeta_x. \label{tau}
\end{equation}

A natural way to think about equations (\ref{NE}) for a fixed
$x$ is as a mapping of triples (Figure~\ref{figne}),
%
%
\begin{equation}
(\xi_x, \eta_{x-e_2}, \zeta_{x-e_1}) \mapsto (
\eta_{x}, \zeta_{x}, \check\xi_{x-e_1-e_2}).
\label{gmap}
\end{equation}

%
%
\begin{figure}[b]

\includegraphics{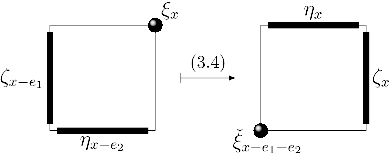}

\caption{Mapping (\protect\ref{gmap}) that involves
variables on a single lattice square. The picture illustrates how
southwest corners are flipped into northeast corners in an inductive
proof of the Burke property.}\label{figne}
\end{figure}

This mapping has the property that if $(\xi_x, \eta_{x-e_2}, \zeta_{x-e_1})$
are independent with marginals (\ref{ga-s}), then the same is true for
$(\eta_{x}, \zeta_{x}, \check\xi_{x-e_1-e_2})$, as can be checked,
for example,
via Laplace transforms.
Consequently a gamma system $(\xi, \eta, \zeta, \check\xi)$ can be
constructed
by repeated application of equations (\ref{NE}) to independent gamma
variables given in (c).

An equivalent way to define the gamma system
is to first construct the following polymer partition
functions from the weights given in (c): for $0\le u<v$ in $\mathbb{Z} _+^2$,
%
%
\begin{equation}
Z_{u,v}=\sum_{x_\centerdot\in\Pi_{u,v}} \prod
_{i=1}^{\llvert
v-u\rrvert _1} V^{-1}_{x_i} \qquad
\mbox{with } V_{i,j}=\cases{ \xi_{i,j}, &\quad$(i,j)\in
\mathbb{N} ^2$,
\vspace*{3pt}\cr
\eta_{i,0}, &\quad$i\in\mathbb{N} $,
$j=0$,
\vspace*{3pt}\cr
\zeta_{0,j}, &\quad$i=0$, $j\in\mathbb{N} $.} \label{gZ1}
\end{equation}
Then, for $x\in\mathbb{Z} _+^2$ and $r\in\{1,2\}$ such that
$x-e_r\in\mathbb{Z}
_+^2$, define
\[
\tau_{x-e_r,x}= \frac{Z_{0,x-e_r}}{Z_{0, x}}.
\]
%
The\vspace*{2pt} weights $\eta_x$ and $\zeta_x$ are then defined via
(\ref{tau}). Now we have a gamma system $(\xi, \eta, \zeta)$,
which can be
augmented to a gamma system $(\xi, \eta, \zeta, \check\xi)$ since
$\check\xi$ is a function of $(\eta, \zeta)$.

Mapping (\ref{gmap}) furnishes the induction step
in the proof of the \emph{Burke property} of the log-gamma polymer
(\cite{sepp-12-aop}, Theorem~3.3): for any down-right path on $\mathbb
{Z} _+^2$,
the $\tau$-variables on the path, the $\xi$ variables strictly to the
northeast
of the path, and the $\check\xi$ variables strictly to the southwest
of the path are all mutually independent with marginal distributions
(\ref{ga-s}). The induction proof begins with the path that consists
of the
$e_1$- and $e_2$-axes. Southwest corners of the path can be flipped
into northeast corners by an application of (\ref{gmap}), as
illustrated in
Figure~\ref{figne}.

As an application of the Burke property, consider the down-right path
consisting of the north and east boundaries of the rectangle
$\{0,\ldots,m\}\times\{0,\ldots,n\}$. Then the Burke property gives us
this statement:
%
%
\begin{eqnarray}\label{gcheckxi}
& \mbox{variables $\{\eta_{i,n}, \zeta_{m,j}, \check\xi
_{i-1,j-1}\dvtx  1\le i\le m, 1\le j\le n\}$}&
\nonumber\\[-8pt]\\[-8pt]
&\mbox{are mutually independent with marginals (\ref{ga-s}).}&\nonumber
\end{eqnarray}
%

%
%
\begin{remark}
\label{grmk2}
Let us revisit the zero temperature limit (Remark~\ref{hrmk2}).
The log-gamma polymer does not have an explicit $\beta$ parameter,
but $\rho$ represents temperature. Replace $\rho$ by $\varepsilon
\rho$ in the
definitions above, so that $\xi_x \sim \operatorname{Gamma}(\varepsilon\rho$).
Then as $\varepsilon\searrow0$, $-\varepsilon\log\xi
_x\Rightarrow\omega_x$, a rate $\rho$ exponential
weight. For
$u<v$ in $\mathbb{N} $,
\begin{eqnarray*}
&& \varepsilon\log Z_{u,v} = \varepsilon\log\sum
_{x_{\centerdot}\in\Pi_{u,v}} \exp \Biggl\{ - \varepsilon^{-1} \sum
_{i=1}^{\llvert  v-u\rrvert _1}\varepsilon\log
\xi_{x_i} \Biggr\}
\\
&&\qquad \Rightarrow\quad \max_{x_{\centerdot}\in\Pi_{u,v}} \sum_{i=1}^{\llvert
v-u\rrvert _1}
\omega_x = G_{u,v} \qquad\mbox{as $\varepsilon
\searrow0$}.
\end{eqnarray*}
In other words, we have convergence in distribution to last-passage
percolation with exponential weights.
\end{remark}

An important function of the polymer path is the exit point or exit
time $t_{\mathrm{exit}}$ of the path from the boundary:
$t_{\mathrm{exit}}={t_{e_1}}\vee{t_{e_2}}$,
%
%
\begin{equation}
{t_{e_1}}=\max\bigl\{ k\ge0\dvtx  \mbox{$x_i=(i,0)$ for $0\le
i\le k$} \bigr\} \label{gxexit}
\end{equation}
and
%
%
\begin{equation}
{t_{e_2}}=\max\bigl\{ \ell\ge0\dvtx  \mbox{$x_j=(0,j)$ for $0
\le j\le\ell$} \bigr\}. \label{gyexit}
\end{equation}
Note that for each path ${t_{e_1}}\wedge{t_{e_2}}=0$.
Partition functions (\ref{gZ1}) based at $0$ can be equivalently
written as
%
%
\begin{equation}
Z_{0,v}=\sum_{x_\centerdot\in\Pi_{0,v}} \Biggl( \prod
_{i=1}^{t_{\mathrm{exit}}}\tau_{x_{i-1}, x
_{i}}^{-1}
\Biggr) \Biggl( \prod_{j=t_{\mathrm{exit}}+1}^{\llvert  v\rrvert _1}
\xi_{
x_j}^{-1} \Biggr), \qquad v\in\mathbb{Z}
_+^2. \label{gZ}
\end{equation}

In a $(\lambda,\rho)$ gamma system we have the means $\mathbb{E}
(\log\eta
_{i,0})=\Psi_0(\lambda)$
and
%
%
\begin{equation}
\mathbb{E} (\log Z_{0,(m,n)}) = -m\Psi_0(\lambda)-n
\Psi_0(\rho -\lambda). \label{ElogZ}
\end{equation}
The second one comes from
%
%
\begin{equation}
\log Z_{0,(m,n)} =-\sum_{i=1}^m
\log\eta_{i,0} - \sum_{j=1}^n
\log\zeta_{m,j}, \label{app4}
\end{equation}
a sum of two correlated sums of i.i.d. random variables.
Above $\Psi_0=\Gamma'/\Gamma$ is the digamma function. It is
strictly increasing on $(0,\infty)$,
with $\Psi_0(0+)=-\infty$ and $\Psi_0(\infty)=\infty$. Its derivative
is the trigamma
function $\Psi_1=\Psi_0'$ that is convex, strictly decreasing, with
$\Psi_1(0+)=\infty$ and $\Psi_1(\infty)=0$.

The asymptotic directions (or velocities) of admissible paths
in $\mathbb{Z} _+^2$ lie in the simplex
$\mathcal{U}=\{\mathbf{u}=(u,1-u)\dvtx  u\in[0,1]\}$.
Fundamental for the behavior of the log-gamma polymer
is a 1--1 correspondence between
velocities $\mathbf{u}\in\mathcal{U}$ and parameters $\lambda\in
[0,\rho]$,
for a fixed $\rho$. The \emph{characteristic direction} for $(\lambda,\rho)$ is
%
%
\begin{equation}
\mathbf{u}_{\lambda,\rho}= \biggl(\frac{\Psi_1(\rho
-\lambda)}{\Psi_1(\lambda)+\Psi_1(\rho-\lambda)}, \frac{\Psi_1(\lambda)}{\Psi_1(\lambda)+\Psi_1(\rho-\lambda
)}
\biggr)\in\mathcal{U}. \label{gchar}
\end{equation}
Conversely,
for $\mathbf{u}=(u,1-u)$, the unique parameter $\theta(u)=\theta
(\mathbf{u}
)\in[0,\rho]$
for which
$\mathbf{u}$ is the characteristic direction is defined by
$\theta(0)=0$, $\theta(1)=\rho$ and
%
%
\begin{equation}
-u\Psi_1\bigl(\theta(u)\bigr)+(1-u)\Psi_1\bigl(\rho-
\theta(u)\bigr)=0 \qquad\mbox{for $u\in(0,1)$}. \label{thetax}
\end{equation}
Function $\theta(u)$ is a strictly increasing bijective
mapping between $u\in[0,1]$ and $\theta\in[0,\rho]$.

The function $\theta(\mathbf{u})$ will appear throughout the paper.
Let us point out that if $(m,n)=c\mathbf{u}$, then the right-hand side
of (\ref{ElogZ}) is minimized by $\lambda=\theta(\mathbf{u})$.
As we shall see, this identifies the limiting free energy for the
log-gamma polymer with i.i.d. $\operatorname{Gamma}(\rho)$
weights.
Notationally, $\lambda, \alpha, \nu$ denote generic parameters
in $[0,\rho]$, while $\theta$ is reserved for the function defined
above.

\section{Limits of ratios of point-to-point partition functions}\label{secratio}
Fix $0<\rho<\infty$.
Let i.i.d. $\operatorname{Gamma}(\rho)$
weights $w=\{w_x\dvtx  x\in\mathbb{Z} _+^2\}$ be given. Define
partition functions
%
%
\begin{equation}
Z_{u,v}=\sum_{x_{ \centerdot}\in\Pi_{u,v}} \prod
_{i=0}^{\llvert  v-u\rrvert _1-1} w_{x_i}^{-1},
\qquad0\le u\le v \mbox{ in } \mathbb{Z} _+^2. \label{Z8}
\end{equation}
Note that the weight at $u$ is included and $v$ excluded,
in contrast with definitions (\ref{hZ}) and (\ref{gZ1}). This is
for convenience, to have clean limit statements below.

Suppose a lattice point $(m,n)\in\mathbb{N} ^2$ tends to infinity in
the first quadrant
so that it has an asymptotic direction in the interior of the quadrant.
Let $\lambda\in(0,\rho)$ be the unique value such that the following
assumption holds:
%
%
\begin{equation}
m\wedge n\to\infty\quad\mbox{and}\quad 
\frac{m}n
\to\frac{\Psi_1(\rho-\lambda)}{\Psi_1(\lambda)}. \label{char5}
\end{equation}
When (\ref{char5}) holds we say that $(m,n)\to\infty$ in the
characteristic direction
of $(\lambda,\rho)$.

The central theorem of this paper constructs gamma systems out of
i.i.d. weights by taking limits of ratios of point-to-point partition
functions.

%
%
\begin{theorem}\label{Zthm1}
On\vspace*{1pt} the probability space of the i.i.d. $\operatorname{Gamma}(\rho)$
weights $w=\{w_x\dvtx  x\in\mathbb{Z} _+^2\}$, there exist random variables
$\{\xi^\lambda_{x}, \eta^\lambda_{x-e_2},\zeta^\lambda_{x-e_1}\dvtx
\lambda\in(0,\rho), x\in\mathbb{N} ^2\}$ with the following properties:
\begin{longlist}[(iii)]
\item[(i)] For each $\lambda\in(0,\rho)$, $(\xi^\lambda, \eta
^\lambda,\zeta^\lambda, w)$
is a gamma system with parameters $(\lambda, \rho)$. Furthermore,\vspace*{1pt}
if on the same probability space there are additional random variables
$(\tilde\xi, \tilde\eta,\tilde\zeta)=\{\tilde\xi_x, \tilde\eta
_{x-e_2},\tilde\zeta_{x-e_1}\dvtx  x\in\mathbb{N} ^2\}$
such\vspace*{1pt} that $(\tilde\xi, \tilde\eta,\tilde\zeta, w)$
is a gamma system
with parameters
$(\nu,\rho)$, then
$ (\tilde\xi, \tilde\eta,\tilde\zeta)= (\xi^\nu,
\eta^\nu,\zeta^\nu
)$ a.s.

\item[(ii)] Suppose a sequence
$(m,n)\to\infty$ in the characteristic direction
of $(\lambda,\rho)$, as defined in (\ref{char5}). Then, for all
$x\in\mathbb{N} \times\mathbb{Z} _+$ and $y\in\mathbb{Z} _+\times
\mathbb{N} $, these almost sure
limits hold:
%
%
\begin{equation}
\eta^\lambda_x=\lim_{(m,n)\to\infty}
\frac{Z_{x,(m,n)}}{Z_{x-e_1,(m,n)}} \quad\mbox{and}\quad \zeta^\lambda_y=\lim
_{(m,n)\to\infty}\frac{Z_{y,(m,n)}}{Z_{y-e_2,(m,n)}} \label{Z82}
\end{equation}
and, furthermore, for all $1\le p<\infty$,
%
\begin{eqnarray}
\label{Z821} \lim_{(m,n)\to\infty}\mathbb{E} \bigl[ \bigl\llvert \log
Z_{x,(m,n)} -\log Z_{x-e_1,(m,n)} - \log\eta^\lambda_x
\bigr\rrvert ^p \bigr] &=&0\quad\mbox{and}
\nonumber\\[-8pt]\\[-8pt]
\lim_{(m,n)\to\infty}\mathbb{E} \bigl[ \bigl\llvert \log
{Z_{y,(m,n)}}-\log{Z_{y-e_2,(m,n)}} - \log\zeta^\lambda_y
\bigr\rrvert ^p \bigr] &=&0.\nonumber
\end{eqnarray}

\item[(iii)] The weights are continuous in $\lambda$, and the edge
weights are monotone in $\lambda$: for each $x$ for which the weights
are defined, almost surely,
%
%
\begin{equation}
\eta^{\lambda_1}_x \le\eta^{\lambda_2}_x \quad
\mbox{and}\quad \zeta^{\lambda_1}_x \ge\zeta^{\lambda_2}_x
\qquad\mbox{for $\lambda_1\le\lambda_2$} \label{Z9}
\end{equation}
and
%
%
\begin{equation}
\eta^{\lambda}_x \to\eta^{\nu}_x,\qquad \zeta
^{\lambda}_x\to\zeta^{\nu}_x,\qquad
\xi^{\lambda}_x\to\xi^{\nu}_x \qquad
\mbox{as $\lambda\to\nu$.} \label{Z11}
\end{equation}
\end{longlist}
\end{theorem}


The rest of this section proves Theorem~\ref{Zthm1}.
The reader not interested in the (rather technical) proof can proceed
to the next
section where these limits are applied to solve a variational problem
for the limiting free energy.

The\vspace*{1pt} proof relies on the following lemma for gamma systems.
Let $(\xi,\eta,\zeta,\check\xi)$ be an $(\alpha,\rho)$-system
according to Definition~\ref{ga-def}. Using the $\check\xi$ weights,
define partition functions
%
%
\begin{equation}
\check Z_{u,v}=\sum_{x_{ \centerdot}\in\Pi_{u,v}} \prod
_{i=0}^{\llvert  v-u\rrvert _1-1} ( \check\xi_{x_i}
)^{-1}, \qquad0\le u\le v \mbox{ in } \mathbb{Z}
_+^2, \label{pZ81}
\end{equation}
and for $x\in\mathbb{N} \times\mathbb{Z} _+$ and $y\in\mathbb{Z}
_+\times\mathbb{N} $ edge
ratio weights
%
%
\begin{equation}
\label{pZ21} \check I_{x,(m,n)}=\frac{\check Z_{x,(m,n)}}{\check Z_{x-e_1,(m,n)}} \quad\mbox{and}\quad
\check J_{y,(m,n)}=\frac{\check Z_{y,(m,n)}}{\check Z_{y-e_2,(m,n)}}.
\end{equation}

%
%
\begin{lemma}\label{Zas-lm3}
Let $0<\lambda< \alpha< \tilde\lambda<\rho$.
Consider two sequences $(m_i,n_i)\to\infty$ and
$(\tilde m_j,\tilde n_j)\to\infty$ in $\mathbb{N} ^2$ such that
\[
\frac{m_i}{n_i} \to\frac{\Psi_1(\rho-\lambda)}{\Psi_1(\lambda)} \quad\mbox{and}\quad \frac{\tilde m_j}{\tilde n_j}
\to\frac{\Psi_1(\rho
-\tilde\lambda
)}{\Psi_1(\tilde\lambda)}.
\]
%
Then for $x\in\mathbb{N} \times\mathbb{Z} _+$ and $y\in\mathbb{Z}
_+\times\mathbb{N} $,
%
%
\begin{equation}
\mathop{\operatorname{\overline{\lim}}}_{i\to\infty} \check I_{x,(m_i,n_i)}
\le\eta_x \le\mathop{\operatorname{\underline{\lim}}}_{j\to\infty}
\check I_{x,(\tilde m_j,\tilde n_j)} \qquad\mbox{a.s.} \label{Zas13}
\end{equation}
and
\[
\mathop{\operatorname{\overline{\lim}}}_{j\to\infty} \check J_{y,(\tilde m_j,\tilde n_j)}
\le \zeta_y \le\mathop{\operatorname{\underline{
\lim}}}_{i\to\infty} \check J_{y,(m_i,n_i)} \qquad\mbox{a.s.}
\]
%
\end{lemma}

\begin{pf}
For notational simplicity
we drop the $i,j$ indices from $(m,n)$ and $(\tilde m, \tilde n)$.
We relate ratios (\ref{pZ21}) to ratios of partition functions with
boundaries.
Let $Z^\mathrm{NE}_{(k,\ell),(m,n)}$ denote a partition function that uses
$\eta$ and $\zeta$ weights on the north and east boundaries of the rectangle
$\{k,\ldots,m\}\times\{\ell,\ldots,n\}$ and $\check\xi$ weights in
the bulk:
\begin{eqnarray*}
Z^\mathrm{NE}_{(k,n),(m,n)} &=& \prod
_{s=k+1}^{m} \frac{1}{\eta_{s,n}},
\\
Z^\mathrm{NE}_{(m,\ell),(m,n)} &=& \prod_{t=\ell+1}^{n}
\frac{1}{\zeta
_{m,t}},
\end{eqnarray*}
and for $0\le k<m$ and $0\le\ell<n$
%
%
\begin{eqnarray}
\label{Z95} Z^\mathrm{NE}_{(k,\ell),(m,n)} &=& \sum
_{i=k}^{m-1} \check Z_{(k,\ell
),(i,n-1)}
\frac{1}{\check\xi_{i,n-1}}\prod_{s=i+1}^{m}
\frac{1}{\eta_{s,n}}
\nonumber\\[-8pt]\\[-8pt]
&&{} + \sum_{j=\ell}^{n-1} \check
Z_{(k,\ell),(m-1,j)} \frac{1}{\check\xi_{m-1,j}}\prod_{t=j+1}^{n}
\frac{1}{\zeta_{m,t}}.\nonumber
\end{eqnarray}
In the last formula $Z^\mathrm{NE}_{(k,\ell),(m,n)}$ is decomposed
according to the entry points $(i,n)$ and $(m,j)$ of the paths
on the north and east boundaries. If the entry is at $(i,n)$, the
first boundary variable encountered is $\eta_{i+1,n}$ associated to
the edge $\{(i,n), (i+1,n)\}$.
The last bulk weight $\check\xi_{i,n-1}$ has to be inserted
explicitly into the formula because $\check Z_{(k,\ell),(i,n-1)} $
does not
include the weight at $(i,n-1)$, by its definition~(\ref{pZ81}).

The corresponding ratio weights on edges are
%
%
\begin{equation}
I_{(k,\ell),(m,n)}=\frac{Z^\mathrm{NE}_{(k,\ell),(m,n)} }{Z^\mathrm
{NE}_{(k-1,\ell
),(m,n)} } \quad\mbox{and}\quad J_{(k,\ell),(m,n)}=
\frac{Z^\mathrm{NE}_{(k,\ell),(m,n)} }{Z^\mathrm
{NE}_{(k,\ell
-1),(m,n)} }. \label{Z9501}
\end{equation}
Due to the reversibility of the shift-invariant setting,
these ratio weights are the same
as the original ratio weights, and thereby do not depend on
$(m,n)$. This is the content of the next lemma.

%
%
\begin{lemma}\label{lm-Z9}
For $0\le k\le m$ and $0\le\ell\le n$
such that the weights below
are defined,
%
%
\begin{equation}
\eta_{k,\ell}= I_{(k,\ell),(m,n)} \quad\mbox {and}\quad
\zeta_{k,\ell}=J_{(k,\ell),(m,n)}. \label{Z9502}
\end{equation}
\end{lemma}

\begin{pf}
When $\ell=n$ in the \mbox{$\eta$-}identity or $k=m$ in the \mbox{$\zeta
$-}identity, the
claims follow from the definitions.
Here is the induction step for $\eta_{k,\ell}$,
assuming the identities have been verified
for the edges $\{(k-1,\ell+1),(k,\ell+1)\}$ and
\mbox{$\{(k,\ell),(k,\ell+1)\}$}, closest to the north and east of the edge
$\{(k-1,\ell),(k,\ell)\}$:
\begin{eqnarray*}
I_{(k,\ell),(m,n)}&=& \frac{\check\xi_{k-1,\ell} Z^\mathrm{NE}_{(k,\ell),(m,n) }}{
Z^\mathrm{NE}_{(k,\ell),(m,n)} + Z^\mathrm{NE}_{(k-1,\ell
+1),(m,n)}}
\\
&=& \check\xi_{k-1,\ell} \biggl( 1 + \frac{Z^\mathrm{NE}_{(k-1,\ell+1),(m,n)} }{Z^\mathrm{NE}_{(k,\ell
+1),(m,n)} } \cdot
\frac{Z^\mathrm{NE}_{(k,\ell+1),(m,n)} }{Z^\mathrm{NE}_{(k,\ell
),(m,n)} } \biggr)^{-1}
\\
&=& \check\xi_{k-1,\ell} \biggl( 1 + \frac{\zeta_{k,\ell+1} }{\eta_{k,\ell+1} }
\biggr)^{-1}
= (\eta_{k,\ell} + \zeta_{k-1,\ell+1}
) \biggl( 1 + \frac{\zeta_{k-1,\ell+1} }{\eta_{k,\ell} } \biggr)^{-1}
\\
&=& \eta_{k,\ell}.
\end{eqnarray*}
The third equality is the induction step. The fourth equality uses
(\ref{NE}) twice.
\end{pf}

We need one more variant of ratio weights, namely the types where the last
step of the path is restricted to either $e_1$ or $e_2$.
Relative to any fixed rectangle $\{k,\ldots,m\}\times\{\ell,\ldots,n\}$,
define the distances of the entrance points of the polymer path
$x_\centerdot\in\Pi_{(k,\ell),(m,n)}$ on the north and east boundaries
to the corner $(m,n)$,
%
%
\begin{equation}
{t^*_{e_1}}=\max\bigl\{ r\ge0\dvtx  \mbox{$x_{m-k+n-\ell
-i}=(m-i,n)$ for
$0\le i\le r$} \bigr\} \label{dualxix}
\end{equation}
and
%
%
\begin{equation}
{t^*_{e_2}}=\max\bigl\{ r\ge0\dvtx  \mbox{$x_{m-k+n-\ell
-j}=(m,n-j)$ for
$0\le j\le r$} \bigr\}. \label{dualxiy}
\end{equation}
For a subset $A$ of paths, write $Z(A)$ for the partition function of
paths restricted
to $A$ (in other words, for the unnormalized polymer measure).
Then define, for $r\in\{1,2\}$,
%
%
\begin{eqnarray}
\label{Z9507} I^{e_r}_{(k,\ell),(m,n)}&=&\frac{ Z^\mathrm{NE}_{(k,\ell
),(m,n)}({t^*_{e_r}}>0)
}{ Z^\mathrm{NE}_{(k-1,\ell),(m,n)}({t^*_{e_r}}>0)}
\quad\mbox{and}
\nonumber\\[-8pt]\\[-8pt]
J^{e_r}_{(k,\ell),(m,n)}&=&\frac{ Z^\mathrm
{NE}_{(k,\ell
),(m,n)}({t^*_{e_r}}>0) }{ Z^\mathrm{NE}_{(k,\ell
-1),(m,n)}({t^*_{e_r}}>0) }.\nonumber
\end{eqnarray}

We are ready to prove Lemma~\ref{Zas-lm3}. We go through the proof of
(\ref{Zas13}),
the case for $\check J$ being the same.
Applying Lemma~\ref{Zcomp-lm1} from the \hyperref[secapp]{Appendix} (to a reversed
rectangle) gives
%
%
\begin{eqnarray}
\label{Z95082} \eta_{k,\ell} + \bigl( I^{e_1}_{(k,\ell),(m+1,n+1)} -
\eta_{k,\ell
} \bigr) & \le&\check I_{(k,\ell),(m,n)}
\nonumber\\[-8pt]\\[-8pt]
& \le& \eta_{k,\ell} + \bigl( I^{e_2}_{(k,\ell),(m+1,n+1)} -
\eta_{k,\ell} \bigr).\nonumber
\end{eqnarray}

Taking (\ref{Z9502}) into consideration, the task is
%
%
\begin{eqnarray}
\label{Z95086}
&& \mathop{\operatorname{\overline{\lim}}}_{(m,n)\to\infty} \bigl\{
I^{e_2}_{(k,\ell),(m+1,n+1)} - I_{(k,\ell),(m+1,n+1)} \bigr\}
\nonumber\\[-8pt]\\[-8pt]
&&\qquad \le0 \le\mathop{\operatorname{\underline{\lim}}}_{(\tilde m,\tilde n)\to
\infty} \bigl\{
I^{e_1}_{(k,\ell
),(\tilde m+1,\tilde n+1)} - I_{(k,\ell),(\tilde m+1,\tilde n+1)} \bigr\}.\nonumber
\end{eqnarray}

We do the first limit for $e_2$. The second is similar. Introduce the parameter
%
%
\begin{equation}
\label{pN} N=\frac{m+n}{\Psi_1(\rho-\lambda)+
\Psi_1(\lambda)} \to\infty
\end{equation}
%
with the property that ${(m,n)}/N \to( \Psi_1(\rho-\lambda), \Psi
_1(\lambda))$.
The first inequality of~(\ref{Z95086})
follows from showing that $\forall\varepsilon>0$ $\exists a>0$ such that
%
%
\begin{equation}
\mathbb{P} \bigl\{ I^{e_2}_{(k,\ell),(m+1,n+1)} \ge I_{(k,\ell),(m+1,n+1)} +
\varepsilon \bigr\} \le2e^{-aN}. \label{Z950835}
\end{equation}

Introduce the quenched path measure $Q^\mathrm{NE}_{(k,\ell),(m,n)}$
that corresponds
to the partition function in (\ref{Z95}):
\begin{eqnarray*}
I^{e_2}_{(k,\ell),(m+1,n+1)} &=&\frac{ Z^\mathrm{NE}_{(k,\ell),(m+1,n+1)}({t^*_{e_2}}>0) }{
 Q^\mathrm{NE}_{(k-1,\ell),(m+1,n+1)}({t^*_{e_2}}>0) \cdot Z^\mathrm
{NE}_{(k-1,\ell
),(m+1,n+1)}}
\\
&\le&  \frac{ I_{(k,\ell),(m+1,n+1)} }{
 Q^\mathrm{NE}_{(k-1,\ell),(m+1,n+1)}({t^*_{e_2}}>0) }.
\end{eqnarray*}
For small enough $\varepsilon_N$
the probability in (\ref{Z950835}) is bounded above by the sum
%
%
\begin{equation}
\label{pZ13} \qquad\mathbb{P} \biggl\{ I_{(k,\ell),(m+1,n+1)} \ge\frac{\varepsilon
}{2\varepsilon_N} \biggr
\} + \mathbb{P} \bigl\{ Q^\mathrm{NE}_{(k-1,\ell
),(m+1,n+1)}\bigl({t^*_{e_1}}>0
\bigr) \ge \varepsilon_N \bigr\}.
\end{equation}
Note that the event in the $Q^\mathrm{NE}$-probability was replaced by
its complement.
A~sequence $0<\varepsilon_N\searrow0$
will be chosen below.

By (\ref{Z9502}) $I_{(k,\ell),(m+1,n+1)}$ has $\operatorname{Gamma}(\alpha$)
distribution, and so the first probability in (\ref{pZ13})
is bounded by $e^{-c{\varepsilon}/{\varepsilon_N}}$.\vspace*{1pt}

We show that the $Q^\mathrm{NE}$-probability in (\ref{pZ13}) is actually
a large deviation by replacing $(m,n)$ with a direction that is
characteristic for $(\alpha,\rho)$. The next lemma contains the
idea for replacing $(m,n)$.

%
%
\begin{lemma} \label{pQlm} Let $(\bar m, \bar n)$ satisfy
$\bar m>m$ and $\ell<\bar n<n$. Then
\[
Q^\mathrm{NE}_{(k,\ell),(m,n)}\bigl({t^*_{e_1}}>0\bigr) =
Q^\mathrm{NE}_{(k,\ell),(\bar m,\bar n)}\bigl({t^*_{e_1}}>\bar m-m\bigr).
\]
%
\end{lemma}

\begin{pf} A path in $\Pi_{(k,\ell),(m,n)}$ that satisfies ${t^*_{e_1}}>0$
must use one of the edges $\{(i,\bar n-1),(i,\bar n)\}$, $k\le i\le m-1$.
Otherwise it hits the east boundary first and ${t^*_{e_1}}=0$.
Decomposing according to this choice of edge and using
definition (\ref{Z95}),
\begin{eqnarray*}
Q^\mathrm{NE}_{(k,\ell),(m,n)}\bigl({t^*_{e_1}}>0\bigr) &=&\sum
_{i=k}^{m-1} \check Z_{(k,\ell),(i,\bar n-1)}
\frac{1}{\check\xi_{i,\bar n-1} } \cdot\frac{Z^\mathrm{NE}_{(i,\bar n),(m,n)}}{Z^\mathrm{NE}_{(k,\ell
),(m,n)}}.
\end{eqnarray*}
By Lemma~\ref{lm-Z9} the last ratio does not depend on $(m,n)$,
and $(m,n)$ can be replaced by
$(\bar m,\bar n)$. This moves the northeast corner in definition
(\ref{Z95}) to $(\bar m,\bar n)$, as well as the reference point
of ${t^*_{e_1}}$ in (\ref{dualxix}). Since the sum still
runs up to $m-1$, it now represents
paths in $\Pi_{(k,\ell),(\bar m,\bar n)}$ that hit the north boundary
to the left of $(m, \bar n)$.
This proves Lemma~\ref{pQlm}.
\end{pf}

Take
%
%
\begin{equation}
\label{pbar-mn} (\bar m, \bar n) = \bigl(\bigl\lfloor{N \Psi_1(\rho-
\alpha)}\bigr\rfloor+k-1, \bigl\lfloor{N \Psi_1(\alpha)}\bigr\rfloor+
\ell \bigr),
\end{equation}
essentially the characteristic direction for
$(\alpha,\rho)$.
Since $\lambda<\alpha$ and $\Psi_1$ is strictly decreasing,
there exists $\gamma>0$ such that for large enough $N$,
$\bar m\ge m+1+N\gamma$ and $\bar n\le n-N\gamma$. Put $\varepsilon
_N=e^{-\delta_1\gamma N}$ for a small enough $\delta_1>0$. Then
for large enough~$N$,
%
%
\begin{eqnarray}
\label{p8}
&& \mathbb{P} \bigl\{ Q^\mathrm{NE}_{(k-1,\ell
),(m+1,n+1)}
\bigl({t^*_{e_1}}>0\bigr) \ge \varepsilon_N \bigr\}
\nonumber\\[-8pt]\\[-8pt]
&&\qquad \le \mathbb{P} \bigl\{ Q^\mathrm{NE}_{(k-1,\ell),(\bar m, \bar
n)}
\bigl({t^*_{e_1}}>N\gamma\bigr) \ge e^{-\delta_1\gamma N} \bigr\} \le
e^{-c_1\gamma N}.\nonumber
\end{eqnarray}
The last inequality came from Lemma~\ref{ldp-lm1} in
the \hyperref[secapp]{Appendix}
where we can take $\kappa_N=1$ and $\delta\le\gamma$.
Both probabilities
in (\ref{pZ13}) have been shown to decay
exponentially in $N$, and consequently (\ref{Z950835}) holds.
This completes the proof of Lemma~\ref{Zas-lm3}.
\end{pf}

Turning to the proof of Theorem~\ref{Zthm1},
we begin by showing the a.s. convergence in (\ref{Z82})
for a fixed sequence and fixed $\lambda$. Later, when we finish the
proof of Theorem~\ref{Zthm1}, the comparisons of Lemma~\ref{Zas-lm3}
allow us to extend the limit to all sequences with an asymptotic direction.
Define ratio variables by
%
%
\begin{equation}
\eta_{x, (m,n)}=\frac{Z_{x,(m,n)}}{Z_{x-e_1,(m,n)}} \quad\mbox{and}\quad
\zeta_{y,(m,n)}= \frac{Z_{y,(m,n)}}{Z_{y-e_2,(m,n)}} \label{p9}
\end{equation}
for $x\in\mathbb{N} \times\mathbb{Z} _+$ and $y\in\mathbb{Z}
_+\times\mathbb{N} $.

%
%
\begin{proposition}\label{Zas-pr1}
Fix $0<\lambda<\rho$ and fix
a sequence $(m,n)\to\infty$
as in (\ref{char5}). Then for all $x\in\mathbb{N} \times\mathbb{Z}
_+$ and $y\in\mathbb{Z}
_+\times\mathbb{N} $
the almost sure limits
%
%
\begin{equation}
\eta_x=\lim_{(m,n)\to\infty} \eta_{x, (m,n)} \quad
\mbox{and}\quad \zeta_y=\lim_{(m,n)\to\infty}
\eta_{y, (m,n)} \label{Zas22}
\end{equation}
exist and have distributions $\eta_x \sim\operatorname
{Gamma}(\lambda)$ and
$\zeta_y \sim\operatorname{Gamma}(\rho-\lambda)$.
\end{proposition}

\begin{pf}
We treat the case of the $\eta$ variables, the case for $\zeta$ being
identical.
For a while, until otherwise indicated, we are considering
a fixed sequence of lattice points that satisfies $(m,n)\to\infty$
as in (\ref{char5}). To avoid extra notation we refrain from
indexing the lattice points, as in $(m_k,n_k)$. Later we can improve
the result so that the limit only depends on $\lambda$ and not
on the particular sequence
$(m,n)\to\infty$.

We show that
for $0<s<\infty$ the
distribution functions
%
%
\begin{eqnarray}\label{pG}
G^*(s)&=&\mathbb{P} \Bigl\{ \mathop{\operatorname{\overline{\lim }}}_{(m,n)\to\infty}
\eta _{x,(m,n)} \le s \Bigr\} \quad\mbox{and}
\nonumber\\[-8pt]\\[-8pt]
G_*(s)&=&\mathbb{P} \Bigl\{ \mathop{
\operatorname{\underline{\lim }}}_{(m,n)\to\infty} \eta _{x,(m,n)} \le s \Bigr\}\nonumber
\end{eqnarray}
satisfy $G^*(s)=G_*(s)= F_\lambda(s) $
where
\[
F_\lambda(s) =\Gamma(\lambda)^{-1}\int_0^s
t^{\lambda-1}e^{-t} \,dt
\]
is the c.d.f. of the $\operatorname{Gamma}(\lambda)$ distribution. Since
$\mathop{\operatorname{\underline{\lim}}}\eta_{x,(m,n)} \le
\mathop{\operatorname{\overline{\lim}}}\eta_{x,(m,n)}$, this suffices
for the conclusion. Working with the distributions allows us to use
any particular construction of the processes.

Let $\{U_{i,j}\}$ be i.i.d. $\operatorname{Uniform}(0,1)$ random variables.
For $i,j\in\mathbb{N} $ and $\alpha\in(0,\rho)$ define
%
%
\begin{equation}
\bar\eta^\alpha_{i,0}=F_\alpha^{-1}(U_{i,0})
\quad\mbox{and}\quad \bar\zeta^\alpha_{0,j} =
F_{\rho-\alpha}^{-1}(U_{0,j}). \label
{Z9-coup1}
\end{equation}
This gives\vspace*{1pt} coupled weights $\bar\eta^\alpha_{i,0} \sim
\operatorname{Gamma}(\alpha)$
on the south boundary and
$\bar\zeta^\alpha_{0,j} \sim \operatorname{Gamma}(\rho-\alpha)$ on the west boundary
of the positive quadrant.
For the bulk weights take an i.i.d. collection
$\{\sigma_x\}_{x\in\mathbb{N} ^2}$ of $\operatorname{Gamma}(\rho)$ weights
independent of
$\{U_{i,j}\}$.

As mentioned after Definition~\ref{ga-def}, the mutually independent initial
weights $\{\sigma_{i,j}, \bar\eta^\alpha_{i,0},\bar\zeta^\alpha
_{0,j}\dvtx  i,j\in\mathbb{N}\}$
can\vspace*{1pt} be extended to the full
gamma $(\alpha, \rho)$
system $(\sigma, \bar\eta^\alpha, \bar\zeta^\alpha, \check
\sigma
^{[\alpha]})$.
The construction preserves monotonicity of the edge weights,
so that
\[
\bar\eta^\alpha_{i,j}\le\bar\eta^\nu_{i,j}\quad
\mbox{and}\quad \bar \zeta^\alpha_{i,j}\ge\bar\zeta^\nu_{i,j}
\qquad\mbox{for $\alpha\le\nu$.}
\]
%
Superscript $[\alpha]$ reminds us that even though
the variables $\{\check\sigma^{[\alpha]}_{i,j}\}_{i,j\ge0}$ are
i.i.d. $\operatorname{Gamma}(\rho)$ for each $\alpha\in(0,\rho)$, they were
computed from $\alpha$-boundary
conditions. Define partition functions
%
%
\begin{equation}
\check Z^{[\alpha]}_{u,v}=\sum_{x_{ \centerdot}\in
\Pi_{u,v}}
\prod_{i=0}^{\llvert  v-u\rrvert _1-1} \bigl( \check \sigma
^{[\alpha]}_{x_i} \bigr)^{-1}, \qquad0\le u\le v
\mbox{ in } \mathbb{Z} ^2, \label{Z81}
\end{equation}
and edge
ratio weights
%
%
\begin{equation}
\label{Z21} \check I^{{[\alpha]}}_{x,(m,n)}=\frac{\check Z^{{[\alpha
]}}_{x,(m,n)}}{\check Z^{{[\alpha]}}_{x-e_1,(m,n)}} \quad
\mbox{and}\quad \check J^{{[\alpha]}}_{y,(m,n)}=\frac{\check Z^{{[\alpha
]}}_{y,(m,n)}}{\check Z^{{[\alpha]}}_{y-e_2,(m,n)}}.
\end{equation}
For each $\alpha\in(0,\rho)$, we have equality in distribution of
processes
%
%
\begin{equation}\label{pd-eq}
\bigl\{ \check I^{{[\alpha]}}_{(i+1,j), (m,n)}, \check J^{{[\alpha
]}}_{(i,j+1),(m,n)}, \check\sigma^{[\alpha]}_{i,j} \bigr\}
\stackrel{d}= \{ \eta_{(i+1,j), (m,n)}, \zeta_{(i,j+1),(m,n)}, w _{i,j}
\}.\hspace*{-30pt} 
\end{equation}
These processes are indexed by $\{(i,j),(m,n)\in\mathbb{Z} _+^2\dvtx
(m,n)\ge
(i+1,j+1)\}$.
The equality in distribution comes from identical constructions applied
to i.i.d. $\operatorname{Gamma}(\rho)$ weights: on the left to $\check\sigma
^{[\alpha]}$, on
the right to $w$. Now in (\ref{pG}) we can use
any process $\{ \check I^{{[\alpha]}}_{x,(m,n)} \}$.\vspace*{1pt}

For\vspace*{2pt} any $0<\alpha_1< \lambda< \alpha_2<\rho$, applying
Lemma~\ref{Zas-lm3} to two gamma systems $(\sigma, \bar\eta
^{\alpha
_1}, \bar\zeta^{\alpha_1}, \check\sigma^{[{\alpha_1}]})$ and
$(\sigma, \bar\eta^{\alpha_2}, \bar\zeta^{\alpha_2}, \check
\sigma
^{[{\alpha_2}]})$
gives
%
%
\begin{equation}
\qquad \mathop{\operatorname{\underline{\lim}}}_{(m,n)\to\infty} \check
I^{ [\alpha_1]}_{(k,\ell
),(m,n)} \ge \bar\eta^{\alpha_1}_{k,\ell}\quad
\mbox{and}\quad \mathop{\operatorname{\overline{\lim}}}_{(m,n)\to\infty} \check
I^{ [\alpha_2]}_{(k,\ell
),(m,n)} \le\bar\eta^{\alpha_2}_{k,\ell}\qquad
\mbox{a.s.} \label{pZas13}
\end{equation}
By the equality in distribution (\ref{pd-eq}),
\[
G_*(s)=\mathbb{P} \Bigl\{ \mathop{\operatorname{\underline{\lim
}}}_{(m,n)\to\infty} \check I^{ [\alpha
_1]}_{(k,\ell),(m,n)} \le s \Bigr\} \le
\mathbb{P} \bigl\{ \bar\eta^{\alpha_1}_{k,\ell} \le s\bigr\} =
F_{\alpha
_1}(s)\searrow F_\lambda(s)
\]
as $\alpha_1\nearrow\lambda$,
and
\[
G^*(s)=\mathbb{P} \Bigl\{ \mathop{\operatorname{\overline{\lim
}}}_{(m,n)\to\infty} \check I^{ [\alpha
_2]}_{(k,\ell),(m,n)} \le s \Bigr\} \ge
F_{\alpha_2}(s)\nearrow F_\lambda(s)\qquad\mbox{as $\alpha
_2\searrow\lambda$}.
\]
This gives
$F_\lambda(s)\le G^*(s)\le G_*(s)\le F_\lambda(s)$ and
completes the proof of Proposition~\ref{Zas-pr1}.
\end{pf}

Proposition~\ref{Zas-pr1} gave
the a.s. convergence of ratios along a fixed sequence and
for a given $\lambda\in(0,\rho)$.
Next we construct
a system
of weights $(\xi,\eta,\zeta,w)$
from the limits (\ref{Zas22}) by defining
\[
\xi_x=\eta_x+\zeta_x\qquad\mbox{for
$x\in\mathbb{N} ^2$.}
\]
%

%
%
\begin{proposition}
\label{pPr-ga}
The collection $(\xi,\eta,\zeta,w)$ is a gamma system 
with
parameters $(\lambda, \rho)$, that is, it satisfies Definition~\ref{ga-def}.
\end{proposition}

\begin{pf}
Equations (\ref{NE}) follow from the limits (\ref{Zas22}) and
\[
w_x= \frac{Z_{x+e_1,(m,n)}+Z_{x+e_2,(m,n)}}{Z_{x,(m,n)}}.
\]

By the equality in distribution in (\ref{pd-eq}), it also follows that
the limits in (\ref{pZas13}) exist,
%
%
\begin{equation}
\check I^{[\alpha]}_{k,\ell}= \lim_{(m,n)\to\infty} \check
I^{
[\alpha]}_{(k,\ell),(m,n)} \qquad\mbox{a.s.} \label{Zas131}
\end{equation}
%

Let $0<\alpha_1< \lambda< \alpha_2<\rho$. Utilizing (\ref{pd-eq}),
(\ref{pZas13}) and (\ref{Zas131}),
\begin{eqnarray*}
(\eta, w)&\stackrel{d}=&\bigl(\check I^{[\alpha_1]}, \check\sigma
^{[\alpha
_1]}\bigr) \ge \bigl(\bar\eta^{\alpha_1}, \check
\sigma^{[\alpha_1]}\bigr) \mathop{\longrightarrow}\limits
_{\alpha_1\nearrow\lambda} \bigl(\bar
\eta^{\lambda}, \check\sigma^{[\lambda]}\bigr)
\end{eqnarray*}
and
\begin{eqnarray*}
 (\eta, w)&\stackrel{d}=&\bigl(\check I^{[\alpha_2]}, \check
\sigma^{[\alpha_2]}\bigr) \le \bigl(\bar\eta^{\alpha_2}, \check
\sigma^{[\alpha_2]}\bigr) \mathop{\longrightarrow}\limits
_{\alpha_2\searrow\lambda} \bigl(\bar
\eta^{\lambda}, \check\sigma^{[\lambda]}\bigr).
\end{eqnarray*}
The inequalities and the convergence are a.s. and coordinatewise.
The convergence follows from the continuity of definitions
(\ref{Z9-coup1}) in $\alpha$ and the continuity in equations
(\ref{NE}) that inductively define the
$(\bar\eta^{\alpha}, \bar\zeta^{\alpha}, \check\sigma^{[\alpha
]}) $ weights.
The consequence is that
%
%
\begin{equation}
(\eta, w)\stackrel{d}=\bigl(\bar\eta^{\lambda}, \check \sigma
^{[\lambda]}\bigr). \label{pd-eq3}
\end{equation}

Equations
$\zeta_x=w_{x-e_2}-\eta_{x-e_2+e_1}$ and
$\xi_x=\eta_x+\zeta_x$ map $(\eta, w)$ to the full system
$(\xi,\eta,\zeta,w)$. The same mapping applied to the
right-hand side of (\ref{pd-eq3}) recreates the system
$(\sigma, \bar\eta^{\lambda}, \bar\zeta^{\lambda}, \check
\sigma
^{[\lambda]})$,
which we know to be a $(\lambda,\rho)$ gamma system by its construction
below (\ref{Z9-coup1}).
\end{pf}

\begin{pf*}{Proof of Theorem~\ref{Zthm1}}
Fix a countable dense
subset $D$ of $(0,\rho)$ and \mbox{$\forall\lambda\in D$} a sequence
$(m,n)\to\infty$ that satisfies (\ref{char5}).
By Propositions
\ref{Zas-pr1} and~\ref{pPr-ga},
we can use limits (\ref{Z82}) along these particular sequences to
define, almost surely, $(\lambda, \rho)$ gamma systems
$(\xi^\lambda, \eta^\lambda,\zeta^\lambda,w)$ for $\lambda
\in D$.
Monotonicity (\ref{Z9}) is satisfied a.s. for $\lambda_1,\lambda
_2\in D$
by Lemma~\ref{Zas-lm3}. [The point is that the $\check Z$ partition functions
in~(\ref{pZ81}) are the same for all systems $(\xi^\lambda, \eta
^\lambda,\zeta^\lambda,w)$.]

Monotonicity and known gamma distributions give also the limits in
(\ref{Z11})
when $\lambda\to\nu$ in $D$. For example, suppose $\lambda\nearrow
\nu$ in $D$. Then\vspace*{1pt} $\lim_{\lambda\nearrow\nu} \eta^\lambda_x\le
\eta^\nu_x$,
but both are $\operatorname{Gamma}(\nu$) distributed and hence coincide a.s. The limit
$\xi^{\lambda}_x\to\xi^{\nu}_x$ comes from the limits of $\eta$
and $\zeta$ and
$\xi^{\lambda}_x=\eta^{\lambda}_x + \zeta^{\lambda}_x$.

Extend the weights to all $\lambda\in(0,\rho)$ by defining
%
%
\begin{equation}
\eta^\lambda_x=\inf\bigl\{ \eta^\nu_x\dvtx
\nu\in D\cap (\lambda,\rho)\bigr\} =\sup\bigl\{ \eta^\alpha_x\dvtx
\alpha\in D\cap(0,\lambda)\bigr\} \label
{Zas622}
\end{equation}
with the obvious counterpart for $\zeta^\lambda_x$ and then
$\xi^{\lambda}_x=\eta^{\lambda}_x + \zeta^{\lambda}_x$.
The inf and the sup in (\ref{Zas622}) must agree a.s. because (i)
the sup is not above
the inf on account of the monotonicity for $\lambda\in D$, and (ii)
they are both
$\operatorname{Gamma}(\lambda)$ distributed. By the same reasoning, for $\lambda\in
D$ definition
(\ref{Zas622}) gives a.s. back the same value $\eta^\lambda_x$ as
originally constructed.

To check that the new system $(\xi^\lambda, \eta^\lambda,\zeta
^\lambda,w)$ is a $(\lambda,\rho)$ gamma system, fix a sequence
$D\ni\alpha_i\nearrow\lambda$, and observe that equations (\ref
{NE}) are preserved
by limits, and the correct distributions come also through the limit.
Extending
properties (iii) utilizes monotonicity again.
Limits (\ref{Z82}) of ratios for arbitrary sequences, including for
$\lambda\notin D$, come from the comparisons of
Lemma~\ref{Zas-lm3} with the sequences fixed in the beginning of this proof.

The uniqueness in part (i) follows from Lemma~\ref{Zas-lm3} because
the limits
(\ref{Z82}) imply that $\eta^{\alpha_1}_x\le\tilde\eta_x\le
\eta
^{\alpha_2}_x$
for all $\alpha_1<\nu<\alpha_2$.

As the last item we
prove the $L^p$ convergence (\ref{Z821}). Let $\eta_{x, (m,n)}$ and
$\zeta_{y, (m,n)}$ be as in (\ref{p9}). It suffices to show that for
each $p\in[1,\infty)$,
there exists a finite constant $C(p)$ such that
%
%
\begin{equation}
\label{Lp3} \mathbb{E} \bigl[ \llvert \log \eta_{x,
(m,n)} \rrvert
^p \bigr] \le C(p) \qquad\mbox{for all $(m,n)$ in the sequence.}
\end{equation}
The argument for $\zeta_{y, (m,n)}$ is analogous, or comes by transposition.
The proof splits into separate bounds for plus and minus parts. The
plus part is quick.
\[
\frac{Z_{x,(m,n)}}{Z_{x-e_1,(m,n)}}=\frac{Z_{x,(m,n)}}{w
_{x-e_1}^{-1} ( Z_{x,(m,n)}+Z_{x-e_1+e_2,(m,n)})} \le w_{x-e_1}
\]
from which, for all $x$, $(m,n)$ and $1\le p<\infty$,
\[
\mathbb{E} \bigl[ \bigl( \log^+ \eta_{x, (m,n)}
\bigr)^p \bigr] \le C(p) <\infty.
\]

For the minus part, pick $\alpha\in(0,\lambda)$ and $\varepsilon
>0$. In the next derivation use distributional equality (\ref
{pd-eq}), bring in the ratio variables (\ref{Z9501}) and (\ref
{Z9507}) with north--east boundaries with parameter $\alpha$, and
finally use the Schwarz inequality and $\check I^{{[\alpha]}}_{x,
(m,n)}\ge I^{e_1}_{x,(m+1,n+1)}$ from (\ref{Z95082}):
%
%
\begin{eqnarray}
\qquad&& \mathbb{E} \bigl[ \bigl( \log^- \eta_{x, (m,n)}\bigr)^p
\bigr]\nonumber
\\
&&\qquad = 
\mathbb{E} \bigl[ \bigl( \log^- \check
I^{{[\alpha]}}_{x, (m,n)}\bigr)^p \bigr]
\nonumber
\\
&&\qquad = \mathbb{E} \bigl[ \bigl( \log^- \check I^{{[\alpha]}}_{x,
(m,n)}
\bigr)^p, I^{e_1}_{x,(m+1,n+1)} \le(1-\varepsilon)
I_{x,(m+1,n+1)} \bigr]\nonumber
\\
&&\quad\qquad{}
+ \mathbb{E} \biggl[ \biggl( \log\frac{1}{ \check I^{{[\alpha]}}_{x,
(m,n)}}
\biggr)^p, \check I^{{[\alpha]}}_{x, (m,n)}\le1,
\nonumber
\\
&&\hspace*{60pt} I^{e_1}_{x,(m+1,n+1)} > (1-\varepsilon)
I_{x,(m+1,n+1)} \biggr]
\nonumber
\\
%
&&\qquad \le \bigl\{ \mathbb{E} \bigl[ \bigl( \log^- \check
I^{{[\alpha]}}_{x,
(m,n)}\bigr)^{2p} \bigr] \bigr
\}^{1/2}
\bigl\{\mathbb{P} \bigl( I^{e_1}_{x,(m+1,n+1)}
\le(1-\varepsilon) I_{x,(m+1,n+1)} \bigr) \bigr\}^{1/2}
\label{Lp8}
\\
&&\quad\qquad{} + \mathbb{E} \biggl[ \biggl\llvert \log\frac{1}{ I_{x,(m+1,n+1)} }\biggr
\rrvert ^p \biggr] +\log\frac{1}{1-\varepsilon}. \label{Lp9}
\end{eqnarray}
By Lemma~\ref{lm-Z9} $I_{x,(m+1,n+1)}$ is a $\operatorname{Gamma}(\alpha$) variable,
and consequently line (\ref{Lp9}) is a constant, independent of $x$
and $(m,n)$.

It remains to show that line (\ref{Lp8}) is bounded by a constant.
From
\begin{eqnarray*}
I^{e_1}_{x,(m+1,n+1)} &=&\frac{ Z^\mathrm
{NE}_{x,(m+1,n+1)}({t^*_{e_1}}>0) }{
 Z^\mathrm{NE}_{x-e_1,(m+1,n+1)}({t^*_{e_1}}>0)}
\\
&\ge&\frac{ Z^\mathrm{NE}_{x,(m+1,n+1)} Q^\mathrm
{NE}_{x,(m+1,n+1)}({t^*_{e_1}}>0)
}{ Z^\mathrm{NE}_{x-e_1,(m+1,n+1)} }
\\
&=& I_{x,(m+1,n+1)} Q^\mathrm{NE}_{x,(m+1,n+1)}\bigl({t^*_{e_1}}>0
\bigr)
\end{eqnarray*}
and a switch to complements, we deduce that the probability on line
(\ref{Lp8}) is bounded by
%
%
\begin{equation}
\label{Lp10} \mathbb{P} \bigl\{ Q^\mathrm {NE}_{x,(m+1,n+1)}
\bigl({t^*_{e_2}} >0\bigr) \ge\varepsilon \bigr\}.
\end{equation}
This probability can be shown to be bounded by $e^{-aN}$ for a constant
$a>0$ exactly as was done for probability (\ref{p8}), where $N$ is
defined by (\ref{pN}) and is proportional to both $m$ and $n$. This
time $\alpha<\lambda$, and so the characteristic direction $(\bar m,
\bar n) $ for $(\alpha,\rho)$ defined as in (\ref{pbar-mn})
satisfies $\bar m< m-N\gamma$ and $\bar n> n+N\gamma$ for some
$\gamma>0$. Qualitatively speaking this means that in (\ref{Lp10})
the direction $(m,n)$ proceeds too fast along the $e_1$-direction,
compared with the characteristic direction, and thereby renders the
event ${t^*_{e_2}}>0$ a deviation.

Last we need to control the moment on line (\ref{Lp8}). Let $x=(k,\ell
)\in\mathbb{Z} _+^2$. From the definition of the ratio weights in
(\ref{Z21}) and the partition functions in (\ref{Z81}), with superscript
$[\alpha]$ dropped to simplify notation,
\begin{eqnarray*}
\frac{1}{ \check I^{{[\alpha]}}_{x, (m,n)}} &=& \frac{\check
Z_{x-e_1,(m,n)}}{\check Z_{x,(m,n)}} = \sum_{b=0}^{n-\ell}
\Biggl( \prod_{j=0}^b \check\sigma
_{x-e_1+je_2}^{-1} \Biggr) \frac{\check Z_{x+be_2,(m,n)}}{\check
Z_{x,(m,n)}}
\\
&\le& \sum_{b=0}^{n-\ell} \check
\sigma_{x-e_1+be_2}^{-1} \prod_{j=0}^{b-1}
\frac{ \check\sigma_{x+je_2}}{ \check\sigma
_{x-e_1+je_2}} \le2n\cdot e^{\max_{0\le b\le n} S_b} \cdot\max_{0\le b\le n}
\check\sigma_{x-e_1+be_2}^{-1},
\end{eqnarray*}
where $S_t=\sum_{j=0}^{t-1}(\log{ \check\sigma_{x+je_2}}-\log{
\check\sigma_{x-e_1+je_2}})$
is a sum of mean-zero i.i.d. variables with all moments.
Consequently
\begin{eqnarray*}
\mathbb{E} \biggl\llvert \log^+\frac{1}{ \check I^{{[\alpha]}}_{x,
(m,n)}}\biggr\rrvert ^{2p}
&\le& C\log n + \mathbb{E} \Bigl[ \max_{0\le b\le n}
\llvert S_b\rrvert ^{2p} \Bigr] + \mathbb{E} \Bigl[ \max
_{0\le b\le n} \llvert \log\check \sigma_{be_2}\rrvert
^{2p} \Bigr]
\\
&\le& Cn^{2p} \le CN^{2p}.
\end{eqnarray*}

Combining the two last paragraphs shows that
\[
\mbox{line (\ref{Lp8})} \le CN^{2p} e^{-aN}\le C(p).
\]
Combining all the bounds verifies (\ref{Lp3}) and thereby the $L^p$
convergence in (\ref{Z821}).
\end{pf*}

\section{Busemann functions and a variational characterization of the free energy}
\label{secbuse}
In this section we turn the
limits of ratios of point-to-point partition functions
into 
Busemann functions, and use these to solve a variational formula
for the limiting free energy.
The parts from this section needed for the sequel are definition
(\ref{tilt4}) of the velocity $\mathbf{u}(h)$ associated to a tilt $h$,
and the
large deviation bound~(\ref{vldp3}). The latter is needed for the
proofs in Section~\ref{secp2l}.

We consider briefly general i.i.d. weights $w=(w_x)_{x\in\mathbb{Z} _+^2}$
on a probability
space $(\Omega, \mathfrak{S}, \mathbb{P} )$ assumed to satisfy
%
%
\begin{equation}
\exists\varepsilon>0\dvt\qquad\mathbb{E} \bigl( \llvert \log w_0
\rrvert ^{2+\varepsilon} \bigr)<\infty. \label{vmom}
\end{equation}
Later we specialize
back to $w_0 \sim \operatorname{Gamma}(\rho)$.
It is convenient to use exponential
Boltzmann--Gibbs factors.
Let $p(e_1)=p(e_2)=1/2$ be the kernel of the background random walk $X_n$
with expectation $E$ and initial point $X_0=0$. Define the potential
$g(w)=-\log w_0+\log2$.
In this notation the point-to-point partition function (\ref{Z8}) is
\[
Z_{0,v}=E \bigl[ e^{\sum_{k=0}^{n-1} g(T_{X_k}\omega)}, X_n=v \bigr], \qquad n=
\llvert v\rrvert _1.
\]
%
Introduce a tilted point-to-line partition function
%
%
\begin{equation}
\qquad Z^h_{0,(N)}=E \bigl[ e^{\sum_{k=0}^{N-1}
g(T_{X_k}\omega
) + h\cdot X_N} \bigr], \qquad
\mbox{$h=(h_1,h_2)\in\mathbb{R} ^2$ and $N\in
\mathbb{N} $.} \label{vZ2}
\end{equation}

The set of limit velocities for admissible walks
in $\mathbb{Z} _+^2$ is $\mathcal{U}=\{(u,1-u)\dvtx  0\le u\le1\}$, with
relative interior
$\operatorname{int}\mathcal{U}=\{(u,1-u)\dvtx  0< u< 1\}$.
For each $\mathbf{u}=(u,1-u)\in\mathcal{U}$, let $\hat x_n(\mathbf
{u})=(\lfloor{nu}\rfloor,
n-\lfloor{nu}\rfloor)$.
Define limiting point-to-point free energies
\[
\Lambda_{p2p}(\mathbf{u}) = \lim_{n\to\infty}
n^{-1}\log Z_{0,
\xhat_n(\mathbf{u})}, \qquad\mathbf{u}\in\mathcal{U},
\]
%
and tilted point-to-line free energies
\[
\Lambda_{p2\ell}(h) = \lim_{N\to\infty} N^{-1}\log
Z^h_{0, (N)}, \qquad h=(h_1,h_2)\in
\mathbb{R} ^2.
\]
%
Under
assumption (\ref{vmom}) these limits exist $\mathbb{P} $-a.s.,
$\Lambda_{p2p}$ is continuous
and concave in $\mathbf{u}$ and $\Lambda_{p2\ell}$ is continuous and
convex in $h$
\cite{rass-sepp-p2p}.

We recall two variational formulas, valid for i.i.d. weights under assumption~(\ref{vmom}). First, a convex duality between the
free energies (\cite{rass-sepp-p2p}, Remark~4.2, also proved below in
(\ref{vI*}))
%
%
\begin{equation}
\Lambda_{p2p}(\mathbf{u})=\inf_{h\in\mathbb{R} ^2}\bigl\{\Lambda
_{p2\ell}(h)-\mathbf{u}\cdot h\bigr\}. \label{var4}
\end{equation}
%
Let ${\mathscr C}_0$ denote the class of \emph{centered cocycles}
$F\dvtx \Omega\times
\{e_1,e_2\}\to\mathbb{R} $
that satisfy $F\in L^1$, $\mathbb{E} F(w,z)=0$ for $z\in\{e_1,e_2\}$, and
a cocycle property
$F(w,e_1)+F(T_{e_1}w,e_2)=F(w,e_2)+F(T_{e_2}w,e_1)$ $\mathbb{P} $-a.s.
Then we have the variational formula (\cite{rass-sepp-yilm}, Theorem~2.3),
%
%
\begin{equation}
\Lambda_{p2\ell}(h)= \inf_{F\in{\mathscr C}_0} \mathbb{P} \mbox{-}
\mathop{\operatorname{ess}\operatorname{sup}} _w \log \sum
_{z\in\{e_1,e_2\}}p(z) e^{g(w)+h\cdot z+ F(w,z)}. \label
{var1}
\end{equation}
%

We solve (\ref{var4}) and
(\ref{var1}) for the log-gamma model. The next corollary turns the
limits of Theorem~\ref{Zthm1}
into Busemann functions, and states the properties needed for the
development that follows.
Recall the function $\theta(\mathbf{u})\in[0,\rho]$ of (\ref{thetax}),
the unique parameter such that $\mathbf{u}$ is the characteristic
direction for
$(\theta(\mathbf{u}), \rho)$.

%
%
\begin{corollary}[(Corollary of Theorem~\ref{Zthm1})]\label{buse-cor}
Assume $\{w_x\}$ are i.i.d.  $\operatorname{Gamma}(\rho)$.
\begin{longlist}[(a)]
\item[(a)] For each velocity $\mathbf{u}\in\operatorname{int}\mathcal{U}$
and for
each $x, v\in\mathbb{Z} _+^2$,
the $\mathbb{P} $-almost sure limit
%
%
\begin{equation}
B^\mathbf{u}(w,x)=\lim_{n\to\infty} (\log Z_{0, \hat x_n(\mathbf{u})+v}
- \log Z_{x, \hat x_n(\mathbf{u})+v} ) \label{buse8}
\end{equation}
exists and is independent of $v$.

\item[(b)] The sequences $\{ B^\mathbf{u}(T_{ie_1}w, e_1)\dvtx  i\in\mathbb
{Z} _+\}$
and $\{ B^\mathbf{u}(T_{je_2}w, e_2)\dvtx  j\in\mathbb{Z} _+\}$ are
i.i.d. with $e^{-B^\mathbf{u}(w, e_1)} \sim \operatorname{Gamma}(\theta
(\mathbf{u}
))$ and
$e^{-B^\mathbf{u}(w, e_2)} \sim \operatorname{Gamma}(\rho-\theta(\mathbf{u}))$.
\end{longlist}
\end{corollary}

We call $B^\mathbf{u}$ a Busemann function, by analogy with the Busemann
functions of last-passage percolation which are limits
of differences $G_{0, \hat x_n(\mathbf{u})+v}
- G_{x, \hat x_n(\mathbf{u})+v}$.
Of course we are merely re-expressing limits (\ref{Z82}) in the form
\[
e^{-B^\mathbf{u}(T_xw, e_1)}= \lim_{n\to\infty} \frac{Z_{x+e_1, \hat x_n(\mathbf{u})+v}}{Z_{x,
\hat
x_n(\mathbf{u})+v}} =
\eta^{\theta(\mathbf{u})}_{x+e_1}.
\]
The admission of the perturbation $v$ in (\ref{buse8}) gives the
cocycle property,
%
%
\begin{equation}
B^\mathbf{u}(w, x)+B^\mathbf{u}(T_xw,
y)=B^\mathbf{u}(w, x+y). \label{vcocycle}
\end{equation}
As a function of $\mathbf{u}\in\operatorname{int}\mathcal{U}$,
define the tilt vector
%
%
\begin{eqnarray}
\label{hxi} h(\mathbf{u}) &=& \bigl( h_1(\mathbf{u}),h_2(
\mathbf{u}) \bigr)=-\sum_{i=1}^2
\mathbb{E} \bigl[B^\mathbf{u}(w, e_i)\bigr]
e_i
\nonumber\\[-8pt]\\[-8pt]
&=& \bigl(\Psi_0\bigl(\theta(\mathbf{u})\bigr),
\Psi_0\bigl(\rho-\theta(\mathbf {u})\bigr) \bigr).\nonumber
\end{eqnarray}
Note that $h(\mathbf{u})$ is not well defined for $\mathbf{u}$ on the axes.
$\theta(\mathbf{u})$ converges to $0$ (to $\rho$) as $\mathbf{u}$
approaches the $y$-axis ($x$-axis).
Then one of the coordinates of $h(\mathbf{u})$ approaches~$-\infty$.
The function
%
%
\begin{equation}
\mathbf{u}=(u,1-u) \mapsto h_1(\mathbf {u})-h_2(
\mathbf{u})= \Psi_0\bigl(\theta(\mathbf{u})\bigr) -
\Psi_0\bigl(\rho-\theta(\mathbf{u})\bigr) \label
{vh-h}
\end{equation}
is a continuous, strictly increasing function from
$u\in(0,1)$ onto $(-\infty,\infty)$. An inverse function to (\ref{hxi}),
$\mathbb{R} ^2\ni h\mapsto\mathbf{u}(h)\in\operatorname
{int}\mathcal{U}$, is given by
%
%
\begin{eqnarray}
\label{tilt4} &\mbox{$\mathbf{u}=\mathbf{u}(h)$ uniquely characterized by the
equation}&
\nonumber\\[-8pt]\\[-8pt]
& h_1-h_2=\Psi_0\bigl(\theta(\mathbf{u})
\bigr)-\Psi_0\bigl(\rho-\theta(\mathbf {u})\bigr).&\nonumber
\end{eqnarray}
Note that $\mathbf{u}(h)$ is constant when $h$ ranges along a 45
degree diagonal.
If $h=0$ there is no tilt, $\mathbf{u}(0)=(1/2,1/2)$, and
$\theta(\mathbf{u}(0))=\rho/2$.

From these ingredients we solve (\ref{var4}).

%
%
\begin{theorem}\label{vp2p-thm}
Let $\mathbf{u}=(u,1-u)\in\operatorname{int}\mathcal{U}$.
Tilt $h(\mathbf{u})$ kills the point-to-line free energy:
$\Lambda_{p2\ell}(h(\mathbf{u}))=0$ $\forall u\in\operatorname
{int}\mathcal{U}$. Furthermore, $h(\mathbf{u}
)$ minimizes
in (\ref{var4}) and so
%
%
\begin{equation}
\Lambda_{p2p}(\mathbf{u})=-\mathbf{u}\cdot h(\mathbf{u}) = -u
\Psi_0\bigl(\theta(\mathbf{u})\bigr) -(1-u) \Psi_0
\bigl(\rho-\theta(\mathbf{u})\bigr). \label{vp2p}
\end{equation}
\end{theorem}

Define the centered cocycle
%
%
\begin{equation}
F^\mathbf{u}(w, z)= -B^\mathbf{u}(w, z) - h(\mathbf{u} )\cdot z, \qquad z\in\{e_1,e_2\}. \label{Fxi}
\end{equation}
%

%
%
\begin{theorem}\label{tilt-thm3}
Given $h=(h_1,h_2)\in\mathbb{R} ^2$, the equation
\[
h_1(\mathbf{u})-h_2(\mathbf{u}) = h_1-h_2
\]
%
determines a unique $\mathbf{u}\in\operatorname{int}\mathcal{U}$.
Then $F^\mathbf{u}\in{\mathscr C}_0$ is
a minimizer in (\ref{var1}). The right-hand side of (\ref{var1})
is constant in $w$, so the essential supremum
can be dropped: $\mathbb{P} $-a.s.,
%
%
\begin{eqnarray}
\label{tilt68} \Lambda_{p2\ell}(h) &=& \log\sum
_{z\in\{e_1,e_2\}}p(z) e^{g(w)+h\cdot z+ F^\mathbf{u}(w,z)} = -h_2(
\mathbf{u})+h_2
\nonumber\\[-8pt]\\[-8pt]
&=& -\Psi_0\bigl(\rho-\theta(\mathbf{u})\bigr)+h_2.\nonumber
\end{eqnarray}
\end{theorem}

%
%
\begin{remark}
Theorem~\ref{vp2p-thm} is the third calculation of the explicit value of
$\Lambda_{p2p}(\mathbf{u})$. This result was first derived in \cite
{sepp-12-aop} together
with fluctuation bounds. The simplest proof is in \cite{geor-sepp} where
the minimization of the limit of the right-hand side of (\ref{ElogZ})
is done with convex analysis. The value (\ref{tilt68}) of the tilted
point-to-line free energy has not been computed before.
\end{remark}

%
%
\begin{remark}[(Large deviations)]\label{vldp-rmk}
Let us observe how the duality between tilt~$h$ and velocity
$\mathbf{u}$ in (\ref{var4}) 
is a standard large deviation duality. The tilted quenched path measure
is
%
%
\begin{equation}
Q^h_{0,(N)}\{x_\centerdot\} = \frac{1}{Z^h_{0,(N)}}
e^{\sum_{k=0}^{N-1} g(T_{x_k}\omega) + h\cdot X_N} P\{x_\centerdot\}. \label{vQ}
\end{equation}
The quenched
large deviation rate function for
the velocity is ($\mathbb{P} $-a.s.)
\begin{eqnarray*}
I_h(\mathbf{v})&=&- \lim_{\delta\searrow0} \mathop{\overline{
\mathop{\operatorname{\underline{\lim }}}}}\limits
_{N\to\infty} N^{-1}\log
Q^h_{0,(N)}\bigl\{ \bigl\llvert N^{-1}X_N-
\mathbf{v}\bigr\rrvert \le\delta\bigr\}
\\
&=&
\Lambda_{p2\ell}(h) - h\cdot\mathbf{v}-
\Lambda_{p2p}(\mathbf{v}). 
\end{eqnarray*}
%
The last equality uses the continuity of $\Lambda_{p2p}$ and Lemma 2.9
in \cite
{rass-sepp-p2p}.
The limiting logarithmic moment generating function is
\begin{eqnarray*}
\Lambda_{Q,h}(a) &=&\lim_{N\to\infty} N^{-1}\log
E^{Q^h_{0,(N)}} \bigl[ e^{a\cdot X_N} \bigr] = \Lambda_{p2\ell}(h+a)-
\Lambda_{p2\ell}(h) \qquad\mbox{$\mathbb {P} $-a.s.}
\end{eqnarray*}
By Varadhan's theorem these are convex duals of each other:
%
%
\begin{equation}
I_h(\mathbf{v}) = \sup_{a\in\mathbb{R} ^2} \bigl\{ a\cdot
\mathbf{v}- \Lambda_{Q,h}(a) \bigr\} \label{vI*}
\end{equation}
which is the same as (\ref{var4}). For the next section we need
the minimizer of $I_h$. By~(\ref{thetax}), (\ref{vp2p}) and
calculus, $I_h$ is uniquely minimized by $\mathbf{u}(h)$ defined by
(\ref{tilt4}). Consequently the walk converges exponentially
fast: for $\delta>0$,
%
%
\begin{equation}
\mathop{\operatorname{\overline{\lim}}}_{N\to\infty} N^{-1}\log
Q^h_{0,(N)}\bigl\{ \bigl\llvert N^{-1}X_N-
\mathbf{u}(h)\bigr\rrvert \ge\delta\bigr\} <0 \qquad\mbox {$\mathbb{P} $-a.s.}
\label{vldp3}
\end{equation}
Function $\Lambda_{p2p}$ extends naturally to all of $\mathbb{R}
_+^2$ by homogeneity:
$\Lambda_{p2p}(c\mathbf{u})=c\Lambda_{p2p}(\mathbf{u})$. Part of
the duality setting is that
the mean of the Busemann function gives the gradient
$
\nabla_\mathbf{u} \Lambda_{p2p}(\mathbf{u})=-h(\mathbf{u})$. %
\end{remark}

The remainder of this section proves the theorems.

\begin{pf*}{Proof of Theorem~\ref{vp2p-thm}}
That $F^\mathbf{u}$ is a centered cocycle is clear by (\ref{vcocycle}).
Let
\[
f^\mathbf{u}(w, x)=\sum_{i=0}^{m-1}
F^\mathbf{u}(T_{x_i}w, x_{i+1}-x_i) =
-B^\mathbf{u}(w, x) - h(\mathbf{u})\cdot x
\]
be the path integral of $F$. The admissible path $\{x_i\}_{i=0}^m $ above
satisfies $x_0=0$ and $x_m=x$, and the cocycle property implies that
$f^\mathbf{u}$
depends on the path only through the endpoint $x$.
Corollary~\ref{buse-cor}(b) verifies exactly the sufficient condition
(\ref{Lmom}) for~(\ref{assL}), for the function $F^\mathbf{u}$ itself.
From Theorem~\ref{erg-thm1}
in the \hyperref[secapp]{Appendix},
\[
\max_{x\in\mathbb{Z} _+^d\dvtx  \llvert  x\rrvert _1=n} \frac{\llvert
f^\mathbf{u}(w, x)\rrvert }n \to0 \qquad\mbox{a.s.}
\]
%
This ergodic property slips $ f^\mathbf{u}(w, X_n)$ into the exponent
in the free energy limit, and shows that tilt $h(\mathbf{u})$ kills the
point-to-line free energy,
%
%
\begin{eqnarray}
\label{tilt5} \Lambda_{p2\ell}\bigl(h(\mathbf{u})\bigr) &=& \lim
_{n\to\infty} n^{-1}\log E\bigl[ e^{\sum_{k=0}^{n-1} g(T_{X_k}w
) + h(\mathbf{u})\cdot X_n}\bigr]\nonumber
\\
&=&\lim_{n\to\infty} n^{-1}\log E\bigl[ e^{\sum_{k=0}^{n-1} g(T_{X_k}w)
+ h(\mathbf{u})\cdot X_n + f^\mathbf{u}(w, X_n)}
\bigr]
\nonumber\\[-8pt]\\[-8pt]
&=&\lim_{n\to\infty} n^{-1}\log E\bigl[ e^{\sum_{k=0}^{n-1} (
g(T_{X_k}w)
+ h(\mathbf{u})\cdot(X_{k+1}-X_k) + F^\mathbf{u}(T_{X_k}w, X_{k+1}-X_k))
}\bigr]\nonumber
\\
&=&0.\nonumber
\end{eqnarray}
The third equality uses the definition of $f^\mathbf{u}$ as the path
integral of $F^\mathbf{u}$.
The last equality comes from
%
%
\begin{eqnarray}
\label{tilt3}
&& \sum_{z\in\{e_1,e_2\}} p(z) e^{g(w)+ h(\mathbf{u})\cdot z+
F^\mathbf{u}(w, z) }\nonumber
\\
&&\qquad =
\sum_{z\in\{e_1,e_2\}} p(z) e^{g(w)- B^\mathbf{u}(w, z) }
\\
&&\qquad =\lim_{n\to\infty} \frac{ \sum_{z\in\{e_1,e_2\}} p(z) e^{g(w
)} Z_{z, \hat x_n(\mathbf{u})} }{Z_{0, \hat x_n(\mathbf{u})} } =\lim
_{n\to\infty} \frac{ Z_{0, \hat x_n(\mathbf{u})} }{Z_{0, \hat
x_n(\mathbf{u})} } =1.\nonumber
\end{eqnarray}

Fix $\mathbf{u}\in\mathcal{U}$.
Since $\llvert  X_n\rrvert _1=n$,
the expression on the right-hand
side of (\ref{var4}) satisfies
\[
\Lambda_{p2\ell}(h) -\mathbf{u}\cdot h = \Lambda_{p2\ell
}(h_1-h_2,0)
-\mathbf{u}\cdot(h_1-h_2,0)
\]
and so, as a function of $h$, is constant along
45 degree diagonals. So the minimization needs one $h$ point from each
diagonal, which is what
parameterization $h(\mathbf{v})$ of~(\ref{hxi})
achieves by virtue of the bijection (\ref{vh-h}).
The upshot is that
\begin{eqnarray*}
\Lambda_{p2p}(\mathbf{u}) &=&\inf_{\mathbf{v}\in\operatorname{int}\mathcal{U}} \bigl\{
\Lambda _{p2\ell}\bigl(h(\mathbf{v})\bigr)-h(\mathbf{v})\cdot \mathbf{u}
\bigr\}
\\
&=&\inf_{\mathbf{v}\in\operatorname{int}\mathcal{U}} \bigl\{ -h(\mathbf {v})\cdot\mathbf{u}\bigr\} =
-h(\mathbf{u} )\cdot\mathbf{u}.
\end{eqnarray*}
%
The last step is calculus: from explicit formula (\ref{hxi}),
$h(\mathbf{v})\cdot\mathbf{u}$ is uniquely maximized at $\mathbf
{v}=\mathbf{u}$.
This completes the proof of Theorem~\ref{vp2p-thm}.
\end{pf*}

\begin{pf*}{Proof of Theorem~\ref{tilt-thm3}}
Since $\llvert  X_n\rrvert _1=n$ and by (\ref{tilt5}),
\begin{eqnarray*}
\Lambda_{p2\ell}(h) &=&\lim_{n\to\infty} n^{-1}\log
E\bigl[ e^{\sum_{k=0}^{n-1} g(T_{X_k}w
) + h\cdot X_n}\bigr]
\\
&=&\lim_{n\to\infty} n^{-1}\log E\bigl[ e^{\sum_{k=0}^{n-1} g(T_{X_k}w
) + h(\mathbf{u})\cdot X_n}
\bigr] -h_2(\mathbf{u})+h_2 =-h_2(
\mathbf{u})+h_2.
\end{eqnarray*}
%

On the other hand, 
by (\ref{tilt3}),
\begin{eqnarray*}
\log\sum_z p(z) e^{g(w)+ h\cdot z+ F^\mathbf{u}(w, z) } &=& \log\sum
_z p(z) e^{g(w)+ h(\mathbf{u})\cdot z+ F^\mathbf{u}(w, z)
} -h_2(
\mathbf{u})+h_2
\\
&=& -h_2(\mathbf{u})+h_2.
\end{eqnarray*}\upqed
\end{pf*}

\section{Limits of ratios of point-to-line partition functions}\label{secp2l}
Armed with the limits of Theorem~\ref{Zthm1} and the large deviation bound
of Remark~\ref{vldp-rmk},
we prove convergence of ratios of tilted point-to-line partition
functions. With
the tilt parameter $h=(h_1,h_2)\in\mathbb{R} ^2$ and $Z_{u,v}$
defined as
in (\ref{Z8}), let
\[
Z^{h}_{x,(N)}=\sum_{v\in x+\mathbb{Z} _+^2\dvtx  \llvert  v\rrvert _1=N}
e^{h\cdot(v-x)} Z_{x,v} \qquad\mbox{for $N\in\mathbb{N} $ and $\llvert
x\rrvert _1\le N$. }
\]
This is the same as (\ref{vZ2}) with a general initial point $x$.
Recall definition (\ref{tilt4}) that associates a velocity $\mathbf{u}
(h)=(u(h), 1-u(h))$ to a tilt $h$, and definition (\ref{thetax}) that
associates a parameter $\theta(\mathbf{v})$
to a velocity $\mathbf{v}$.

%
%
\begin{theorem}\label{tilt-thm}
Fix $0 < \rho<\infty$, and let i.i.d. $\operatorname{Gamma}(\rho)$ weights
$\{w_x\}_{x\in\mathbb{Z} _+^2}$ be given. For $\lambda\in(0,\rho)$,
let $(\xi^\lambda, \eta^\lambda, \zeta^\lambda, w)$ be the
gamma system
constructed in Theorem~\ref{Zthm1}.
Then for $h=(h_1,h_2)\in\mathbb{R} ^2$, $x\in\mathbb{N} \times
\mathbb{Z} _+$, $y\in\mathbb{Z}
_+\times\mathbb{N} $, $\mathbb{P} $-a.s.,
%
%
\begin{eqnarray}
\label{tilt51} \lim_{N\to\infty} \frac{Z^{h}_{x,(N)}}{ e^{-h_1}Z^{h}_{x-e_1,(N)}} &=&
\eta^{\theta(\mathbf{u}(h))}_x
\quad\mbox{and}
\nonumber\\[-8pt]\\[-8pt]
\lim_{N\to\infty} \frac{Z^{h}_{x,(N)}}{e^{-h_2}Z^{h}_{x-e_2,(N)}} &=&
\zeta^{\theta(\mathbf{u}(h))}_x.\nonumber
\end{eqnarray}
\end{theorem}

In other words, the limit of ratios of point-to-line partition functions
tilted by $h$ is equal to the limit of ratios of point-to-point
partition functions
in the direction~$\mathbf{u}(h)$
\begin{eqnarray*}
\lim_{N\to\infty} \frac{Z^{h}_{x,(N)}}{e^{-h_1}Z^{h}_{x-e_1,(N)}} &=&\lim_{(m,n)\to\infty}
\frac{Z_{x,(m,n)}}{Z_{x-e_1,(m,n)}}
\end{eqnarray*}
and
\begin{eqnarray*}
\lim_{N\to\infty} \frac{Z^{h}_{x,(N)}}{e^{-h_2}Z^{h}_{x-e_2,(N)}} &=&\lim
_{(m,n)\to\infty}\frac{Z_{x,(m,n)}}{Z_{x-e_2,(m,n)}},
\end{eqnarray*}
%
provided $ m/n\to u(h)/(1-u(h)) $. We see the duality between tilt and velocity
from Remark~\ref{vldp-rmk}
again. We do not presently have a proof of $L^p$ convergence as we did
for the point-to-point case in (\ref{Z821}).

In Section~\ref{Q-sec} the limits of ratios from Theorems~\ref{Zthm1}
and~\ref{tilt-thm}
give convergence
of polymer measures to random walk in a correlated random environment.
The remainder of this section proves Theorem~\ref{tilt-thm}.

\begin{pf*}{Proof of Theorem~\ref{tilt-thm}}
We prove (\ref{tilt51}) for the horizontal ratios (first limit).
Begin with a lower bound, and let $\delta_0>0$.
\begin{eqnarray*}
&& \frac{Z^{h}_{x,(N)}}{e^{-h_1}Z^{h}_{x-e_1,(N)}}
\\
&&\qquad =\sum_{v\dvtx  \llvert  v\rrvert _1=N} \frac{e^{h\cdot
(v-x)}Z_{x-e_1,v}}{e^{-h_1}Z^{h}_{x-e_1,(N)}}
\cdot\frac{Z_{x,v}}{Z_{x-e_1,v}}
\\
&&\qquad  =\sum_{v\dvtx  \llvert  v\rrvert _1=N} Q^h_{x-e_1,(N)} \{
X_{N-\llvert
x\rrvert _1+1}=v\} \frac{Z_{x,v}}{Z_{x-e_1,v}}
\\
&&\qquad \ge\sum_{m\dvtx  \llvert  m-N u(h)\rrvert < N\delta_0} Q^h_{x-e_1,(N)}
\bigl\{ X_{N-\llvert  x\rrvert _1+1}=(m,N-m)\bigr\} \frac{Z_{x,(m,N-m)}}{Z_{x-e_1,(m,N-m)}}.
\end{eqnarray*}
Above we introduced a
tilted quenched point-to-line polymer measure
%
%
\begin{equation}
Q^h_{y,(N)}\{x_\centerdot\} = \frac{1}{Z^{h}_{y,(N)}}
e^{h\cdot(x_{N-\llvert  y\rrvert _1}-y)} \prod_{i=0}^{N-\llvert  y\rrvert _1-1}
w^{-1}_{x_i} 
\label{tQ}
\end{equation}
for paths $x_\centerdot$ from $x_0=y$ to the line $\llvert  x_{N-\llvert  y\rrvert _1}\rrvert _1=N$.

Apply construction (\ref{Z95}) to the gamma system
$(\xi^\lambda, \eta^\lambda, \zeta^\lambda, w)$ to define partition
functions $Z^\lambda$ and associated polymer measures $Q^\lambda$
with northern boundary weights $\{\eta^\lambda_{i, N-m+1}\}_{1\le
i\le m+1}$
and eastern boundary weights $\{\zeta^\lambda_{m+1, j}\}_{1\le j\le N-m+1}$.
Recall the dual exit points (\ref{dualxix})--(\ref{dualxiy}).
By an application of Lemma~\ref{Zcomp-lm1} (to the
reversed rectangle),
\begin{eqnarray*}
\frac{Z_{x,(m,N-m)}}{Z_{x-e_1,(m,N-m)}} &\ge&\frac{Z^\lambda_{x,(m+1,N-m+1)}({t^*_{e_1}}>0)}{
Z^\lambda_{x-e_1,(m+1,N-m+1)}({t^*_{e_1}}>0)}
\\
&\ge& Q^\lambda_{x,(m+1,N-m+1)}\bigl\{{t^*_{e_1}}>0\bigr\}
\frac{Z^\lambda
_{x,(m+1,N-m+1)}}{Z^\lambda_{x-e_1,(m+1,N-m+1)}}
\\
&=& Q^\lambda_{x,(m+1,N-m+1)}\bigl\{{t^*_{e_1}}>0\bigr\}
\eta^\lambda_x.
\end{eqnarray*}
The last equality came from Lemma~\ref{lm-Z9}.
Note the notational distinction: $Q^h_{y,(N)}$ is the tilted point-to-line
polymer measure, while $Q^\lambda_{x,y}$ is the point-to-point
polymer measure with boundary parameter $\lambda$.

We have the lower bound
%
%
\begin{eqnarray}
\label{Zas66} \frac{Z^{h}_{x,(N)}}{e^{-h_1}Z^{h}_{x-e_1,(N)}}&\ge&\sum_{m\dvtx  \llvert
m-N u(h)\rrvert < N\delta_0}
Q^h_{x-e_1,(N)} \bigl\{ X_{N-\llvert  x\rrvert _1+1}=(m,N-m)\bigr\}\hspace*{-30pt}
\nonumber\\[-8pt]\\[-8pt]
&&\hspace*{73pt}{}\times Q^\lambda_{x,(m+1,N-m+1)}
\bigl\{{t^*_{e_1}}>0\bigr\} \eta^\lambda_x.\nonumber
\end{eqnarray}

Let $0<\lambda<\theta(\mathbf{u}(h))$.
Define parameter $M\nearrow\infty$ by $ N(1-u(h))=\break M\Psi_1(\theta
(\mathbf{u}(h)))$.
Let
$ (\bar m, \bar n) = x+  (\lfloor{M \Psi_1(\rho-\lambda
)}\rfloor,
\lfloor{M \Psi_1(\lambda)}\rfloor ) $,
a velocity essentially characteristic for
$(\lambda,\rho)$.
As $m$ varies in the sum on the right-hand side of~(\ref{Zas66}), let
$(m_1,n_1)=(m+1,N-m+1)$. Since $\Psi_1$ is strictly
decreasing, if we fix $\delta_0>0$
small enough, there exists $\varepsilon_0>0$ such that, for large
enough $N$,
\begin{eqnarray*}
\bar n-n_1&\ge& M\Psi_1(\lambda) - M\Psi_1
\bigl(\theta\bigl(\mathbf{u}(h)\bigr)\bigr) - N\delta_0-2 \ge M
\varepsilon_0
\end{eqnarray*}
and
\begin{eqnarray*}
m_1-\bar m&\ge& N u(h)-N\delta_0+1 -x- M
\Psi_1 (\rho-\lambda) \ge M\varepsilon_0.
\end{eqnarray*}
On the second line above we also use definition (\ref{thetax}) of $u(h)$.

Following the idea of Lemma~\ref{pQlm} and (\ref{p8}),
\begin{eqnarray*}
&&\mathbb{P} \bigl[ Q^\lambda_{x,(m+1,N-m+1)}\bigl\{{t^*_{e_2}}>0
\bigr\} > e^{-\delta
_1\varepsilon_0 M} \bigr]
\\
&&\qquad \le \mathbb{P} \bigl[ Q^{\lambda}_{x,(\bar m, \bar n)} \bigl
\{{t^*_{e_2}}> M\varepsilon _0 \bigr\} > e^{-\delta_1\varepsilon_0 M}
\bigr] \le e^{-c_1\varepsilon_0 M}.
\end{eqnarray*}
Since there are $O(N)$ $m$-values, Borel--Cantelli and
(\ref{Zas66}) give, for large
enough $n$,
\begin{eqnarray*}
&& \frac{Z^{h}_{x,(N)}}{e^{-h_1}Z^{h}_{x-e_1,(N)}} \ge\eta^\lambda_x \bigl(1-
e^{-\delta_1\varepsilon_0 M} \bigr) Q^h_{x-e_1,(N)} \bigl\{ \bigl\llvert
X_{N-\llvert  x\rrvert _1+1}-N\mathbf {u}(h)\bigr\rrvert < N\delta_0\bigr\}.
\end{eqnarray*}
%
By the quenched LDP (\ref{vldp3}) for the point-to-line measure, the
last probability
tends to $1$. Thus we obtain the lower bound
\[
\mathop{\operatorname{\underline{\lim}}}_{N\to\infty}
\frac
{Z_{x,(N)}}{e^{-h_1}Z_{x-e_1,(N)}} \ge\eta^\lambda_x\nearrow
\eta^{\theta(\mathbf{u}(h))}_x \qquad \mbox {as we let $\lambda\nearrow\theta
\bigl(\mathbf{u}(h)\bigr)$}. 
\]

For the upper bound we first bound summands away from the concentration point
of the quenched measure:
\begin{eqnarray*}
&& \sum_{m\dvtx  \llvert  m-N u(h)\rrvert \ge N\delta_0} \frac
{e^{h\cdot
((m,N-m)-x)}Z_{x,(m,N-m)}}{e^{-h_1}Z^{h}_{x-e_1,(N)}}
\\
&&\qquad \le w_{x-e_1} \sum_{m\dvtx  \llvert  m-N u(h)\rrvert \ge N\delta_0}
\frac
{e^{h\cdot
((m,N-m)-x)}Z_{x,(m,N-m)}}{Z^{h}_{x,(N)}}
\\
&&\qquad \le w_{x-e_1} Q^h_{x,(N)} \bigl\{ \bigl\llvert
X_{N-\llvert  x\rrvert _1}-N\mathbf{u} (h)\bigr\rrvert \ge N\delta_0 \bigr\}
\longrightarrow0.
\end{eqnarray*}
For the remaining fractions we develop an upper bound:
\begin{eqnarray*}
\frac{Z_{x,(m,N-m)}}{Z_{x-e_1,(m,N-m)}} &\le&\frac{Z^\lambda_{x,(m+1,N-m+1)}({t^*_{e_2}}>0)}{
Z^\lambda_{x-e_1,(m+1,N-m+1)}({t^*_{e_2}}>0)}
\\
&\le&\frac{1}{Q^\lambda_{x-e_1,(m+1,N-m+1)}\{{t^*_{e_2}}>0\}}\cdot \frac
{Z^\lambda_{x,(m+1,N-m+1)}}{Z^\lambda_{x-e_1,(m+1,N-m+1)}}
\\
&=& \frac{\eta^\lambda_x}{Q^\lambda_{x-e_1,(m+1,N-m+1)}\{
{t^*_{e_2}}>0\}
}.
\end{eqnarray*}

Combining these,
%
\begin{eqnarray*}
\frac{Z^{h}_{x,(N)}}{e^{-h_1}Z^{h}_{x-e_1,(N)}} &\le&\sum_{m\dvtx  \llvert  m-N u(h)\rrvert < N\delta_0}
Q^h_{x-e_1,(N)} \bigl\{ X_{N-\llvert  x\rrvert _1+1}=(m,N-m)\bigr\}
\\
&&\hspace*{73pt}{} \times\frac{\eta^\lambda_x}{
1 - Q^\lambda_{x-e_1,(m+1,N-m+1)}\{{t^*_{e_1}}>0\}} + o(1),
\end{eqnarray*}
%
where the $o(1)$ term tends to zero $\mathbb{P} $-a.s.
Proceed as for the lower bound, this time choosing
$ \theta(\mathbf{u}(h))<\lambda<\rho$ to show that the $Q^\lambda
$-probability
above vanishes exponentially fast. This completes the proof of Theorem
\ref{tilt-thm}.
\end{pf*}

\section{Limits of path measures}\label{Q-sec}

As in Section~\ref{secratio},
fix $\rho\in(0,\infty)$ and assume that i.i.d. $\operatorname{Gamma}(\rho)$
weights $w=\{w_x\dvtx  x\in\mathbb{Z} _+^2\}$ are given on a probability
space $(\Omega, \mathfrak{S}, \mathbb{P} )$.
Let $Z_{u,v}$ be the point-to-point
partition function defined in (\ref{Z8}), with associated
quenched polymer measure
\[
Q_{u,v} \{x_{ \centerdot}\} = \frac{1}{Z_{u,v}} \prod
_{i=0}^{\llvert
v-u\rrvert _1-1} w_{x_i}^{-1},
\qquad x_{ \centerdot}\in\Pi_{u,v}.
\]
%
Let point-to-line polymer measures be defined as before in (\ref{vQ})
or (\ref{tQ}).

For $\lambda\in(0,\rho)$, let $(\xi^\lambda, \eta^\lambda, \zeta
^\lambda, w)$ denote the gamma
system of weights constructed in Theorem~\ref{Zthm1}.
In this environment, define RWRE transitions on $\mathbb{Z} _+^2$ by
%
%
\begin{eqnarray}\label{Qrwre1}
\pi^{w, \lambda}(x, x+e_1)&=&\frac{\eta^\lambda
_{x+e_1}}{\eta^\lambda_{x+e_1}+\zeta^\lambda_{x+e_2}}\quad\mbox{and}
\nonumber\\[-8pt]\\[-8pt]
\pi^{w, \lambda}(x, x+e_2)&=&\frac{\zeta^\lambda_{x+e_2}}{\eta
^\lambda_{x+e_1}+\zeta^\lambda_{x+e_2}}.\nonumber
\end{eqnarray}
Let $P^{w,\lambda}$ be the quenched path measure of the RWRE
started at $0$.
It is characterized by the initial point and transition
\[
P^{w,\lambda}(X_0=0)=1,\qquad P^{w,\lambda}(X_{k+1}=y
\vert X_k=x)=\pi^{w, \lambda}(x, y).
\]
%
We wrote $P^{w, \lambda}$ instead of $P^{\omega, \lambda}$ because the
quenched distribution is a function of the weights $w$,
through the limits (\ref{Z82}) that appear on the right in (\ref{Qrwre1}).
In other words, the probability space has not been artificially augmented
with the variables that appear in definition (\ref{hspace}):
everything comes from the single i.i.d. collection~$w$.

Let\vspace*{1pt} $Z^\lambda_{u,v}$ denote the partition function defined by
(\ref{gZ}) in gamma system
$(\xi^\lambda, \eta^\lambda, \zeta^\lambda, w)$.
Adapt the notation from (\ref{tau}) in the
form
\[
\tau^\lambda_{x, x+z}=\cases{ \eta^\lambda_{x+e_1},
&\quad$z=e_1$,
\vspace*{3pt}\cr
\zeta^\lambda_{x+e_2}, &
\quad$z=e_2$.}
\]
%
Then we can rewrite transition (\ref{Qrwre1})
as
\begin{eqnarray*}
\pi^{w, \lambda}(x, x+z)&=&\frac{\tau^\lambda_{x,x+z}}{
\tau^\lambda_{x,x+e_1}+\tau^\lambda_{x,x+e_2}}
\\
&=& \frac{(Z^{\lambda}_{0,x+z})^{-1}}{
(Z^{\lambda}_{0,x+e_1})^{-1}+(Z^{\lambda}_{0,x+e_2})^{-1}}, \qquad z\in\{e_1,e_2\}.
\end{eqnarray*}
%
In other words, this RWRE is of the competition interface type
defined by (\ref{hrwre}) in Lemma~\ref{hlm3}. The next theorem
shows that these walks are the limits of the polymer measures on
long paths, both point-to-point and point-to-line.

%
%
\begin{theorem} The following weak limits of probability measures
on the path space $(\mathbb{Z} _+^2)^{\mathbb{Z} _+}$ happen for
$\mathbb{P} $-a.e. $w$.
\begin{longlist}[(ii)]
\item[(i)] Let $0<\lambda<\rho$, and suppose $(m,n)\to\infty$ in the
characteristic direction of
parameters $(\lambda,\rho)$ as defined in (\ref{char5}). Then $Q_{0,(m,n)}$
converges to $P^{w,\lambda}$.

\item[(ii)] Let $h\in\mathbb{R} ^2$. Then as $N\to\infty$ the tilted
point-to-line measure
$Q^h_{0,(N)}$ converges to $P^{w,\theta(\mathbf{u}(h))}$.
\end{longlist}
\end{theorem}

\begin{pf} Fix a finite path $x_{0,M}$ with $x_0=0$. Then $(m,n)\ge
x_M$ for
large enough $(m,n)$, and
%
%
\begin{eqnarray}
\label{QP} Q_{0,(m,n)} \{X_{0,M}=x_{0,M}\} &=&
\frac{Z_{x_M,(m,n)}}{Z_{0,(m,n)}} \prod_{i=0}^{M-1}
w_{x_i}^{-1} \mathop{\longrightarrow}\limits
_{(m,n)\to\infty} \prod
_{i=0}^{M-1} \frac{\tau^{\lambda}_{x_i,x_{i+1}}}{w_{x_i}}
\nonumber\\[-8pt]\\[-8pt]
&=& 
\prod_{i=0}^{M-1}
\pi^{w, \lambda}(x_i, x_{i+1}) = P^{w,\lambda}
\{X_{0,M}=x_{0,M}\}.\nonumber
\end{eqnarray}
We applied limits (\ref{Z82}) and used property $w_x=\eta
^\lambda_{x+e_1}+\zeta^\lambda_{x+e_2}$ of the
gamma system $(\xi^\lambda, \eta^\lambda, \zeta^\lambda, w)$
from Theorem
\ref{Zthm1}.
There are countably many finite paths and these determine weak convergence
on the path space. Hence $\mathbb{P} $-a.s. limits (\ref{QP}) give
claim (i).

The proof of (ii) is the same with limits (\ref{tilt51}) instead.
\end{pf}

The RWRE $P^{w, \lambda}$ has the fluctuation exponent of the $1+1$
dimensional
KPZ (Kardar--Parisi--Zhang) universality class: under the averaged
distribution, at time~$n$, the typical
fluctuation away from the characteristic velocity of $(\lambda,\rho)$
is of size~$n^{2/3}$. The reason is that the RWRE is close to a polymer,
and we can apply fluctuation results for the shift-invariant log-gamma polymer.
Below $\mathbb{E} $ denotes expectation over the weights $w$.
Recall the characteristic velocity $\mathbf{u}_{\lambda,\rho}$ from
(\ref{gchar}).

%
%
\begin{theorem}\label{Qrwrethm}
There exist constants $C_1, C_2<\infty$ such that
for $N\in\mathbb{N} $ and $b\ge C_1$,
%
%
\begin{equation}
\mathbb{E} P^{w,\lambda}\bigl\{ \llvert X_N- N\mathbf
{u}_{\lambda,\rho}\rrvert \ge bN^{2/3} \bigr\} \le C_2b^{-3}.
\label{Qrwre7}
\end{equation}
Given $\varepsilon>0$, there exists $\delta>0$ such that
%
%
\begin{equation}
\mathop{\operatorname{\overline{\lim}}}_{N\to\infty} \mathbb{E}
P^{w,\lambda}\bigl\{ \llvert X_N- N\mathbf{u}_{\lambda,\rho
}
\rrvert \le \delta N^{2/3} \bigr\} \le\varepsilon. \label{Qrwre8}
\end{equation}
\end{theorem}

\begin{pf}
For each $N$ let
$(m,n)= ( \lfloor{cN\Psi_1(\rho-\lambda)}\rfloor, \lfloor
{cN\Psi_1(\lambda)}\rfloor )$ where
$c>0$ is fixed large enough so that $m\wedge n>2N$. Define $0<\kappa
<1$ by
$\kappa^{-1}= c (\Psi_1(\rho-\lambda)+\Psi_1(\lambda))$.
Then up to errors from integer parts $(\kappa m,\kappa n)= N\mathbf{u}
_{\lambda,\rho}$.
(See Figure~\ref{figrwre}.)

%
%
\begin{figure}[t]

\includegraphics{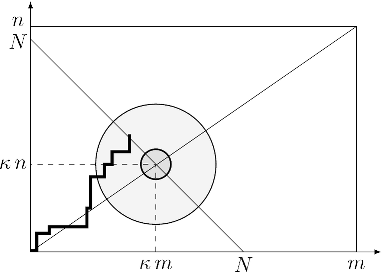}

\caption{Illustration of the proof of Theorem \protect\ref
{Qrwrethm}. The thickset
RWRE path avoids the disk of radius $\delta N^{2/3}$ (dark grey small disk)
but enters the disk of radius $b N^{2/3}$ (light grey large disk) centered
at $(\kappa m,\kappa n)= N\mathbf{u}_{\lambda,\rho}$.}
\label{figrwre}
\end{figure}

Fix $(m,n)$.
We couple the RWRE $P^{w,\lambda}$ with the polymer that obeys the quenched
distribution $Q^{\lambda\mathrm{NE}}_{0,(m,n)}$ defined by applying
construction
(\ref{Z95}) to the gamma system
$(\xi^\lambda, \eta^\lambda, \zeta^\lambda, w)$.
In other words, the boundary weights $\eta^\lambda$ and $\zeta
^\lambda$
are on the north and east, the
bulk weights come from $w$ and the distribution of the weights is
described by (\ref{gcheckxi}), with $w$ taking on the role of
$\check\xi$.
This is the stationary log-gamma polymer to which results from \cite
{sepp-12-aop}
apply.

Define the path $\check X_\centerdot\in\Pi_{0,(m,n)}$
by letting it follow the RWRE until it hits either the north or the
east boundary
of the rectangle $\{0,\ldots,m\}\times\{0,\ldots,n\}$, and then
follow the boundary
to $(m,n)$. The next calculation shows that the quenched distribution
of $\check X_\centerdot$ is $Q^{\lambda\mathrm{NE}}_{0,(m,n)}$. Let\vspace*{1pt}
$x_\centerdot\in\Pi_{0,(m,n)}$. To be concrete, let $0\le k<m$ and
suppose $x_\centerdot$ hits
the north boundary at $x_{k+n}=(k,n)$:
\begin{eqnarray*}
P^{w,\lambda}(\check X_\centerdot=x_\centerdot) &=& \prod
_{j=0}^{k+n-1} \frac{\tau^{\lambda}_{x_j,x_{j+1}}}{w_{x_j}} =\frac{1}{Z^\lambda_{0,(k,n)}}
\prod_{j=0}^{k+n-1} {w^{-1}_{x_j}}
\\
&=&\frac{1}{Z^\lambda_{0,(m,n)}} \prod_{j=0}^{k+n-1}
{w^{-1}_{x_j}} \cdot\prod_{i=k+1}^m
\bigl(\eta^\lambda_{i,n}\bigr)^{-1}
\\
&=&\frac{1}{Z^{\lambda\mathrm{NE}}_{0,(m,n)}} \prod_{j=0}^{k+n-1}
{w^{-1}_{x_j}} \cdot\prod_{i=k+1}^m
\bigl(\eta^\lambda_{i,n}\bigr)^{-1}
=Q^{\lambda\mathrm{NE}}_{0,(m,n)}\{x_\centerdot\}.
\end{eqnarray*}
The\vspace*{2pt} last equality is the definition of $Q^{\lambda\mathrm
{NE}}_{0,(m,n)}\{
x_\centerdot\}$.
The equality $Z^{\lambda\mathrm{NE}}_{0,(m,n)}=Z^\lambda_{0,(m,n)}$
comes by applying Lemma~\ref{lm-Z9} to a telescoping product of ratio weights.

With $c$ large enough, the boundary does not interfere with behavior around
$(\kappa m,\kappa n)=N\mathbf{u}_{\lambda,\rho}$. In (\ref
{Qrwre7})--(\ref{Qrwre8})
we can replace $\mathbb{E} P^{w, \lambda}\{\cdot\}$ with\break $\mathbb
{E} Q^{\lambda
\mathrm{NE}}_{0,(m,n)}\{\cdot\}$. The result follows from Theorem 2.3 of
\cite{sepp-12-aop}, after a
harmless reversal of the lattice rectangle to account for the difference
that in (\cite{sepp-12-aop}, Theorem~2.3), the boundary weights are on
the south
and west.
\end{pf}

\section{The log-gamma polymer random walk in random environment}\label{secrwre}

In the previous section we saw that the limits of log-gamma polymer measures
are polymer RWREs with transition (\ref{hrwre1}), where the
weights come from a gamma system with some
parameters $(\lambda,\rho)$. In this section we identify a
stationary, ergodic
probability distribution for the environment process of a polymer RWRE.
We expect this stationary Markov chain to be the limit of the
environment process
when its initial distribution is an appropriate gamma system (Remark
\ref{esim} below).

The process of the
environment as seen from the particle is
\[
T_{X_n}\omega=(\xi_{X_n+\mathbb{N} ^2}, \eta_{X_n+\mathbb{N}
\times\mathbb{Z} _+},\zeta
_{X_n+\mathbb{Z} _+\times\mathbb{N} }, \check\xi_{X_n+\mathbb{Z} _+^2}).
\]
The state space of this process
is the space
$\Omega_{\mathrm{NE}}$ of weight configurations $\omega=(\xi,\eta,\zeta,\check\xi)$
that satisfy NE induction, as defined in Definition~\ref{hga-def} and
(\ref{hspace}).

Let $0<\alpha,\beta<\infty$
and $\rho=\alpha+\beta+1$. Define probability distribution $\mu
^{\alpha,\beta}$
on the space $\Omega_{\mathrm{NE}}$ 
as follows: let the variables
$(\eta_{\mathbb{N} e_1}, \zeta_{\mathbb{N} e_2}, \xi_{\mathbb{N}
^2})$ be
mutually independent with marginal distributions
%
%
\begin{eqnarray}\label{ed}
\quad &&\eta_{i,0} \sim \operatorname{Gamma}(\alpha),
\zeta_{0,j} \sim \operatorname{Gamma}(\beta),
\xi_{i,j} \sim \operatorname{Gamma}(\rho),\qquad
i,j\in \mathbb{N}.
\end{eqnarray}
The remaining variables
$\{\eta_x, \zeta_x, \check\xi_{x-e_1-e_2}\dvtx  x\in\mathbb{N} ^2\}$
are then
defined by north--east induction (\ref{hNE1})--(\ref{hNE2}).

A few more notational items.
$G_\alpha$ denotes a $\operatorname{Gamma}(\alpha$) random variable and
$\mathbf{E}$ generic expectation. Let $P$ denote the distribution of the
random walk on $\mathbb{Z} _+^2$ that starts at $0$ and has step distribution
\[
p(e_1)=\frac{\alpha}{\alpha+\beta} =1-p(e_2).
\]
Let us call this the
$(\frac{\alpha}{\alpha+\beta},\frac{\beta}{\alpha+\beta})$ random
walk.
An admissible path is denoted by $x_{0,n}=(x_0,x_1,\ldots,x_n)$
with $x_0=0$ and steps $z_k=x_k-x_{k-1}\in\{e_1,e_2\}$.\vspace*{2pt}

The Burke property is not
valid for $\mu^{\alpha,\beta}$ because $\rho\ne\alpha+\beta$, so
under $\mu^{\alpha,\beta}$
the weights do not form a gamma system (Definition~\ref{ga-def}).
However, the $\check\xi$ weights
still turn out to have a tractable distribution which we record in the
next proposition.

%
%
\begin{proposition}\label{epr}
Under\vspace*{2pt} $\mu^{\alpha,\beta}$, the marginal distribution of
$\{\check\xi_x\}_{x\in\mathbb{Z} _+^2}$ is given as follows. Let $\{
h_x\}
_{x\in\mathbb{Z} _+^2}$
be arbitrary bounded Borel functions on $\mathbb{R} _+$. Then for
$n\in\mathbb{N} $,
\begin{eqnarray*}
&& E^{\mu^{\alpha,\beta}} \biggl[ \prod_{x\in\mathbb{Z} _+^2\dvtx
\llvert  x\rrvert _1\le n}
h_x(\check\xi_x) \biggr]
\\
&&\qquad = \sum
_{x_{0,n}\in\{0\}\times(\mathbb{Z} _+^2)^n}\! P(X_{0,n}=x_{0,n}) \prod
_{k=0}^n \mathbf{E}h_{x_k}(G_{\alpha+\beta})
\\
&&\quad\qquad{}\times \prod_{ \llvert  y\rrvert _1\le n\dvtx  y\notin\{x_{0,n}\}} \mathbf
{E}h_{y}(G_{\alpha
+\beta+1}).
\end{eqnarray*}
%
\end{proposition}

In other words, the distribution of the $\check\xi$ weights
is constructed as follows:
run the $(\frac{\alpha}{\alpha+\beta},\frac{\beta}{\alpha+\beta
})$ random
walk, put independent $\operatorname{Gamma}(\alpha+\beta$) variables on the path,
independent $\operatorname{Gamma}(\alpha+\beta+1$) variables off the path, and
average over the
walks.

%
%
\begin{theorem}\label{ethm}
Let the environment $\omega$ have initial distribution $\mu^{\alpha,\beta}$ on the space $\Omega_{\mathrm{NE}}$ of (\ref{hspace}),
and let
the walk $X_n$
obey transitions (\ref{hrwre1}):
\begin{longlist}[(a)]
\item[(a)] The environment process
$T_{X_n}\omega$ is a stationary ergodic Markov chain with state space
$\Omega_{\mathrm{NE}}$.

\item[(b)] The averaged distribution of walk $X_n$ is the
homogeneous $(\frac{\alpha}{\alpha+\beta},\frac{\beta}{\alpha
+\beta})$ random
walk.
\end{longlist}
\end{theorem}

Note the contrast in the behavior of the walk $X_n$.
According to Theorem~\ref{Qrwrethm},
when the environment has the distribution of a gamma system of weights,
the averaged
walk has fluctuations of order $n^{2/3}$. By part (b) above,
when the environment has the $\mu^{\alpha,\beta}$ distribution,
the averaged walk is diffusive.

%
%
\begin{remark}[(Simulations)]\label{esim}
Suppose the
environment process starts from a gamma system with parameters
$(\lambda, \rho)$, with $\rho>1$.
Simulations suggest that then
$T_{X_n}\omega$ converges to $\mu^{\alpha, \beta}$
such that $\alpha+\beta=\rho-1$ and
$(\frac{\alpha}{\alpha+\beta},\frac{\beta}{\alpha+\beta
})=\mathbf{u}
_{\lambda,\rho}$,
the characteristic direction (\ref{gchar}) of the original setting.

Under the environment distribution $\mu^{\alpha, \beta}$, the
averaged distribution of the walk $X_n$
is the diffusive $(\frac{\alpha}{\alpha+\beta},\frac{\beta
}{\alpha+\beta})$ random
walk. Simulations\vspace*{1pt} suggest that under its quenched distribution the walk
localizes,
with a positive fraction of overlap between two independent walks
in the same environment.
\end{remark}

%
%
\begin{remark} We can look at the environment as seen from the walk
with a more
general boundary, instead of simply the axes. Let $\sigma=\{y_j\}
_{j\in\mathbb{Z} }$ be a down-right path in $\mathbb{Z} ^2$ that
goes through $e_2$, $0$ and $e_1$.
That is, $y_{-1}=e_2$, $y_0=0$, $y_1=e_1$ and $y_i-y_{i-1}\in\{e_1,
-e_2\}$.
Let $\mathcal{J}=\{x\dvtx  \exists k\in\mathbb{N}\dvtx  x-(k,k)\in\sigma\}$
be the lattice
strictly to the northeast of
$\sigma$. Weights assigned to this setting are such that
$\{\xi_x\dvtx  x\in\mathcal{J}\}$ are i.i.d. $\operatorname{Gamma}(\rho)$. On the path edge
weights have
different recipes to the northwest and southeast of the origin:
\begin{eqnarray*}
&&\mbox{horizontal edge northwest of $0$: } i< 0, y_i-y_{i-1}=e_1\dvtx  \eta_{y_i}\sim\operatorname{Gamma}(\alpha+1),
\\
&&\mbox{vertical edge northwest of $0$: } i\le0, y_i-y_{i-1}=-e_2\dvtx  \zeta_{y_{i-1}}\sim\operatorname{Gamma}(\beta),
\\
&&\mbox{horizontal edge southeast of $0$: } i\ge1, y_i-y_{i-1}=e_1\dvtx  \eta_{y_i}\sim\operatorname{Gamma}(\alpha),
\\
&&\mbox{vertical edge southeast of $0$: } i> 1, y_i-y_{i-1}=-e_2\dvtx  \zeta_{y_{i-1}}\sim\operatorname{Gamma}(\beta+1).
\end{eqnarray*}
These weights are stationary as we look at the system centered at $X_n$.
The proof goes along the same lines as given below.
\end{remark}

%
%
\begin{remark}[(A degenerate limit and an invariant distribution as
seen from a last-passage competition interface)]
The results above require $\rho>1$. In the limit $\alpha\searrow0$,
$\beta\searrow0$, $\rho\searrow1$, the $\eta^{-1}, \zeta^{-1}$ weights
blow up.
We rescale so that logarithms of edge weights converge
to exponential random variables, and bulk weights vanish.
Let $\varepsilon>0$, $\rho=\varepsilon\alpha+\varepsilon\beta
+1$, and consider the weights
$( \xi^{(\varepsilon)}_{\mathbb{N} ^2}, \eta^{(\varepsilon
)}_{\mathbb{N} \times
\mathbb{Z} _+}, \zeta^{(\varepsilon)}_{\mathbb{Z} _+\times\mathbb
{N} })$
under the distribution $\mu^{\varepsilon\alpha, \varepsilon\beta
}$. The independent
weights of (\ref{ed}) now satisfy for $i,j\in\mathbb{N} $
%
%
\begin{equation}\label{nuabd}
\qquad \xi^{(\varepsilon)}_{i,j} \sim \operatorname{Gamma}(\rho),\qquad
\eta^{(\varepsilon)}_{i,0} \sim \operatorname{Gamma}(\varepsilon\alpha),\qquad
\zeta^{(\varepsilon)}_{0,j} \sim \operatorname{Gamma}(\varepsilon\beta).
\end{equation}
We can construct the weights in (\ref{nuabd})
as functions of uniform variables
as in (\ref{Z9-coup1}). Then the following limits as $\varepsilon
\searrow0$ can be
taken pointwise:
\[
-\varepsilon\log\xi^{(\varepsilon)}_{i,j} \to0, \qquad -\varepsilon
\log\eta^{(\varepsilon)}_{i,0} \to I_{i,0} \sim
\operatorname{Exp}(\alpha),
\]
and
\[
-\varepsilon\log\zeta^{(\varepsilon)}_{0,j} \to
J_{0,j} \sim \operatorname{Exp}(\beta).
\]
%
The NE induction equations (\ref{hNE1}) converge to the equations
%
%
\begin{equation}
I_x=(I_{x-e_2}-J_{x-e_1})^+\quad\mbox{and}\quad
J_x=(J_{x-e_1}-I_{x-e_2})^+. \label{nuNE1}
\end{equation}
The RWRE transition probability converges to a deterministic transition:
\[
\pi^{(\varepsilon)}_{x,x+e_1}=\frac{\eta^{(\varepsilon)}_{x+e_1}}{
\eta^{(\varepsilon)}_{x+e_1}+ \zeta^{(\varepsilon)}_{x+e_2}} \longrightarrow\mathbf{1}
\{ I_{x+e_1}<J_{x+e_2}\}\equiv\pi ^{(0)}_{x,x+e_1}
\qquad\mbox{as $\varepsilon\searrow0$.}
\]
%
The limit leads to an invariant distribution for a last-passage
system. Equations~(\ref{nuNE1}) describe inductively
the increment variables
\[
I_x=G_{0,x}-G_{0,x-e_1}\quad\mbox{and}\quad
J_x=G_{0,x}-G_{0,x-e_2}
\]
of a degenerate last-passage model with boundary weights $\{
I_{i,0},J_{0,j}\dvtx  i,j\in\mathbb{N} \}$
and zero bulk weights. This distribution on $(I_{\mathbb{N} \times
\mathbb{Z} _+},
J_{\mathbb{Z} _+\times\mathbb{N} })$
is invariant for the environment seen from the location $\varphi_n$
that starts at $\varphi_0=0$ and obeys the transition
%
%
\begin{eqnarray}
\label{transitions2}
&&\pi^{(0)}_{x,x+e_1}=\mathbf{1}
\{I_{x+e_1}<J_{x+e_2}\} \quad\mbox{and}\quad\pi^{(0)}_{x,x+e_2}=
\mathbf{1}\{I_{x+e_1}>J _{x+e_2}\}.
\end{eqnarray}
Given the environment, this
defines a deterministic path $\varphi_\centerdot$ on $\mathbb{Z} _+^2$.
We recognize in (\ref{transitions2}) the jump rule of the
competition interface (\ref{cif}).
\end{remark}

The remainder of this section is taken by the proofs.
To prove stationarity of the Markov chain it suffices to consider
the partial environment $(\eta_{\mathbb{N} e_1}, \zeta_{\mathbb{N}
e_2}, \xi_{\mathbb{N} ^2})$
because the other variables of the state are functions of these.
The notation here is that $\eta_{\mathbb{N} e_1}=\{\eta_{i e_1}\}
_{i\in\mathbb{N}
}$, and similarly for
other cases.
The next lemma proves everything in Proposition~\ref{epr} and Theorem
\ref{ethm},
except the ergodicity.

%
%
\begin{lemma}\label{elm1} Fix $n\in\mathbb{N} $ and an admissible path $x_{0,n}$
with $x_0=0$.
Fix a finite set $\mathcal{I}\subset\mathbb{Z}_+^2$, disjoint from
$(x_n+\mathbb{Z} _+^2)\cup\{x_k\}_{0\le k<n}$.
Let $\{h_k\}_{k\in\mathbb{Z}_+}$ and $\{g_u\}_{u\in\mathbb{Z}_+^2}$ be
collections of bounded\vspace*{1pt} Borel functions on $\mathbb{R} _+$.
Let $f$ be a bounded Borel function on $\mathbb{R} _+^{\mathbb{N}
+\mathbb{N} +\mathbb{N} ^2}$. Then
%
%
\begin{eqnarray}
\label{ironman} &&E^{\mu^{\alpha,\beta}} \Biggl[ P^\omega(X_{0,n}=x_{0,n})\nonumber
\\
&&\hspace*{28pt}{}\times
\prod_{k=0}^{n-1} h_{k}(
\check\xi_{x_k}) \cdot\prod_{u\in\mathcal{I}}
g_u(\check \xi_{u})
\cdot f( \eta_{x_n+\mathbb{N} e_1},
\zeta_{x_n+\mathbb{N} e_2}, \xi_{x_n+\mathbb{N}
^2}) \Biggr]
\nonumber\\[-8pt]\\[-8pt]
&&\qquad = P(X_{0,n}=x_{0,n})\nonumber
\\
&&\hspace*{31pt}{}\times \prod
_{k=0}^{n-1} \mathbf{E}\bigl[ h_k(G_{\alpha
+\beta})
\bigr] \cdot \prod_{u\in\mathcal{I}} \mathbf{E}\bigl[
g_u(G_{\alpha+\beta+1})\bigr]
\cdot E^{\mu^{\alpha,\beta}} \bigl[ f(
\eta_{\mathbb{N} e_1}, \zeta_{\mathbb{N}
e_2}, \xi_{\mathbb{N} ^2}) \bigr].\hspace*{-15pt}\nonumber
\end{eqnarray}
%
\end{lemma}

%
%
\begin{remark} Note that the independent $(\check\xi_{X_k})$ cannot
go up to
$k=n$ because $\check\xi_{X_n}=\eta_{X_n+e_1}+\zeta_{X_n+e_2}$, and these
belong in the future of the walk. Adding the statements over $x_{0,n}$ gives
the invariance of $\mu^{\alpha,\beta}$ and
the distribution of $\check\xi$. For a fixed $x_{0,n}$
we get the averaged distribution of the walk and also the statement that
when the walk looks at the $\check\xi$ weights in its past, it sees
$G_{\alpha+\beta}$-variables on its path and $G_{\alpha+\beta+1}$-variables
elsewhere.
\end{remark}

Lemma~\ref{elm1} is basically a consequence of size-biasing beta variables.
The formulation we need is in the next
lemma, whose proof we leave to the reader.

%
%
\begin{lemma} Let the gamma variables below with distinct subscripts be
independent. Then
%
%
\begin{eqnarray}\label{aa-20}
&& \mathbf{E} \biggl[ \frac{G_\alpha}{G_\alpha+G_\beta} f \biggl( G_{\alpha+\beta+1}\cdot
\frac{G_\alpha}{G_\alpha+G_\beta} \biggr)
g \biggl( G_{\alpha+\beta+1}\cdot\frac{G_\beta}{G_\alpha
+G_\beta}
\biggr) h (G_\alpha+G_\beta ) \biggr]\hspace*{-20pt}
\nonumber\\[-8pt]\\[-8pt]
&&\qquad = \frac{\alpha}{\alpha+\beta} \mathbf{E}f(G_{\alpha+1})\cdot\mathbf{E}g(G_{\beta})
\cdot\mathbf {E}h(G_{\alpha+\beta
} ).\nonumber
\end{eqnarray}
\end{lemma}


\begin{pf*}{Proof of Lemma~\ref{elm1}}
We assume that the first step of the walk is $e_1$ and calculate the
distribution.
Introduce functions $\Phi$
to represent north--east induction (\ref{hNE1})--(\ref{hNE2}),
specifically to calculate
the $\check\xi$ weights
on the vertical line $x\cdot e_1=0$ and $\zeta$ weights on the
vertical line $x\cdot e_1=1$, for $x\cdot e_2\ge1$,
%
\begin{eqnarray*}
(\check\xi_{\mathbb{N} e_2}, \zeta_{e_1+\mathbb{N} e_2}) &=&(\check\xi_{\mathbb{N} e_2},
\zeta_{e_1+ e_2}, \zeta _{e_1+e_2+\mathbb{N}
e_2})
\\
&=& \biggl( \Phi_1(\eta_{e_1+e_2}, \zeta_{e_2+\mathbb{N} e_2}, \xi
_{e_1+e_2+\mathbb{N} e_2} ), \xi_{e_1+e_2}\frac{\zeta_{e_2}}{\eta_{e_1}+\zeta_{e_2}},
\\
&&\hspace*{85pt}\Phi_2(\eta_{e_1+e_2},
\zeta_{e_2+\mathbb{N} e_2}, \xi _{e_1+e_2+\mathbb{N}
e_2} ) \biggr).
\end{eqnarray*}
%

Let $h_0, g, f_i$ be bounded Borel functions of their arguments.
The first equality below implements definitions. In the second
equality below apply (\ref{aa-20}) to the triple
$(G_\alpha, G_\beta, G_{\alpha+\beta+1}) = (\eta_{e_1}, \zeta
_{e_2}, \xi_{e_1+e_2})$,
and\vspace*{1pt} note that all other variables are independent of this triple.
Let $G_{\alpha+\beta+1}^{\mathbb{N} e_2}$ denote an i.i.d. $\operatorname{Gamma}(\alpha
+\beta+1)$ sequence. Augment
temporarily the probability space
with independent $G_{\alpha+1}$ and $G_\beta$
variables that are also independent of
all the other variables in $f_2$:
\begin{eqnarray*}
&& E^{\mu^{\alpha,\beta}} \bigl[ P^\omega(X_1=e_1)
h_0(\check\xi _0) g( \check\xi_{\mathbb{N} e_2} )
f_1( \eta_{e_1+\mathbb{N} e_1} ) f_2(\zeta_{e_1+\mathbb{N} e_2} )
f_3( \xi _{e_1+\mathbb{N} ^2}) \bigr]
\\
&&\qquad = E^{\mu^{\alpha,\beta}} \biggl[ \frac{\eta_{e_1}}{\eta
_{e_1}+\zeta_{e_2}} h_0(
\eta_{e_1}+\zeta_{e_2}) f_1( \eta_{e_1+\mathbb{N} e_1}
) f_3( \xi_{e_1+\mathbb{N} ^2})
\\
&&\hspace*{29pt}\qquad\quad{}\times g \biggl( \Phi_1 \biggl( \xi_{e_1+e_2}
\frac{\eta_{e_1}}{\eta_{e_1}+\zeta
_{e_2}}, \zeta_{e_2+\mathbb{N} e_2}, \xi_{e_1+e_2+\mathbb{N} e_2} \biggr) \biggr)
\\
&&\hspace*{29pt}\qquad\quad{}\times f_2 \biggl( \xi_{e_1+e_2}
\frac{\zeta_{e_2}}{\eta_{e_1}+\zeta
_{e_2}},
\\
&&\hspace*{92pt}
 \Phi_2 \biggl( \xi_{e_1+e_2}
\frac{\eta_{e_1}}{\eta_{e_1}+\zeta
_{e_2}}, \zeta_{e_2+\mathbb{N} e_2}, \xi_{e_1+e_2+\mathbb{N} e_2} \biggr) \biggr)
\biggr]
\\
&&\qquad =\frac{\alpha}{\alpha+\beta} \mathbf{E}\bigl[h_0( G_{\alpha+\beta})\bigr]
E^{\mu^{\alpha,\beta}}\bigl[f_1( \eta _{e_1+\mathbb{N} e_1} )\bigr]
E^{\mu^{\alpha,\beta}}\bigl[ f_3( \xi_{e_1+\mathbb{N} ^2})\bigr]
\\[-1pt]
&&\qquad\quad{}\times E^{\mu^{\alpha,\beta}} \bigl[ g \bigl( \Phi_1(
G_{\alpha+1}, \zeta_{e_2+\mathbb{N} e_2}, \xi_{e_1+e_2+\mathbb{N} e_2} ) \bigr)
\\[-1pt]
&&\hspace*{39pt}\qquad\quad{}\times f_2 \bigl( G_\beta,
\Phi_2( G_{\alpha+1}, \zeta_{e_2+\mathbb{N} e_2},
\xi_{e_1+e_2+\mathbb{N} e_2} ) \bigr) \bigr]
\\[-1pt]
&&\qquad =\frac{\alpha}{\alpha+\beta} \mathbf{E}\bigl[h_0( G_{\alpha+\beta})\bigr]
\mathbf{E}\bigl[ g\bigl( G_{\alpha
+\beta+1}^{\mathbb{N}
e_2}\bigr)\bigr]
E^{\mu^{\alpha,\beta}}\bigl[f_1( \eta_{\mathbb{N} e_1} )\bigr]
\\[-1pt]
&&\qquad\quad {}\times E^{\mu^{\alpha,\beta}}\bigl[f_2(
\zeta_{\mathbb{N} e_2} )\bigr] E^{\mu^{\alpha,\beta}}\bigl[ f_3(
\xi_{\mathbb{N} ^2})\bigr] 
\\[-1pt]
&&\qquad =\frac{\alpha}{\alpha+\beta} \mathbf{E}\bigl[h_0( G_{\alpha+\beta})\bigr]
\mathbf{E}\bigl[ g\bigl( G_{\alpha
+\beta+1}^{\mathbb{N}
e_2}\bigr)\bigr]
E^{\mu^{\alpha,\beta}} \bigl[f_1( \eta_{\mathbb{N} e_1} ) f_2(
\zeta_{\mathbb{N} e_2} ) f_3( \xi_{\mathbb{N} ^2}) \bigr].
\end{eqnarray*}
%
In the second-to-last equality, inside $f_1$ and $f_3$ we simply shift by $-e_1$.
Inside $f_2$ variable $G_\beta$ furnishes $\zeta_{e_2}$.
Here is the key point: at this stage the Burke property applies to the mappings
$(\Phi_1,\Phi_2)$
because $G_{\alpha+1}$ furnishes $\eta_{e_1+e_2}$, and thereby
the parameters of the input weights
satisfy $(\alpha+1)+\beta=\rho$. The beta size-biasing
put us back into the setting of a gamma system.
Thus $(\Phi_1,\Phi_2)$ outputs two independent sequences.
The first one denoted by $G_{\alpha+\beta+1}^{\mathbb{N} e_2}$
is i.i.d. $\operatorname{Gamma}(\alpha+\beta+1)$ and it
represents the distribution of $\check\xi_{\mathbb{N} e_2}$.
The second one is i.i.d. $\operatorname{Gamma}(\beta)$, which we take to be $\zeta
_{e_2+\mathbb{N} e_2}$. In the last equality we can combine the three
$\mu
^{\alpha,\beta}$-expectations
because the independence is in accordance with the definition of $\mu
^{\alpha,\beta}$.

Standard arguments generalize the product $f_1f_2f_3$ so that
\begin{eqnarray*}
&& E^{\mu^{\alpha,\beta}} \bigl[ P^\omega(X_1=e_1)
h_0(\check\xi _0) g( \check\xi_{\mathbb{N} e_2} ) F(
\eta_{e_1+\mathbb{N} e_1}, \zeta_{e_1+\mathbb{N} e_2}, \xi _{e_1+\mathbb{N} ^2}) \bigr]
\\
&&\qquad =p(e_1) \mathbf{E}\bigl[h_0( G_{\alpha+\beta})\bigr]
\mathbf{E}\bigl[ g\bigl( G_{\alpha
+\beta+1}^{\mathbb{N}
e_2}\bigr)\bigr]
E^{\mu^{\alpha,\beta}} \bigl[F( \eta_{\mathbb{N} e_1}, \zeta_{\mathbb{N} e_2},
\xi_{\mathbb{N} ^2}) \bigr]
\end{eqnarray*}
%
for Borel functions $h_0, g, F$ such that the expectations make sense.
Reflection across the diagonal gives the alternative formula where the
first step
is $e_2$ instead of $e_1$, $\check\xi_{\mathbb{N} e_2}$ is replaced by
$\check\xi_{\mathbb{N} e_1}$
and $G_{\alpha+\beta+1}^{\mathbb{N} e_2}$ is replaced by $G_{\alpha
+\beta
+1}^{\mathbb{N} e_1}$.\vspace*{1pt}

Referring to the goal (\ref{ironman}), let $\mathcal{I}_0=\mathcal
{I}\setminus
(x_1+\mathbb{Z} _+^2)$
and take $g( \check\xi_{\centerdot} )=\break \prod_{u\in\mathcal{I}_0}
g_u(\check\xi_u)$.
We can combine the $e_1$ and $e_2$ cases into this statement, which is
(\ref{ironman}) for $n=1$:
%
%
\begin{eqnarray}
\label{muab291} &&
E^{\mu^{\alpha,\beta}} \biggl[ P^\omega(X_1=x_1)
h_0(\check\xi _0)\cdot\prod
_{u\in\mathcal{I}_0} g_u(\check\xi_u)
\cdot F( \eta_{x_1+\mathbb{N} e_1},
\zeta_{x_1+\mathbb{N} e_2}, \xi _{x_1+\mathbb{N} ^2}) \biggr]
\nonumber\\[-9pt]\\[-9pt]
&&\qquad =p(x_1) \mathbf{E}\bigl[h_0( G_{\alpha+\beta})\bigr]
\cdot\prod_{u\in\mathcal
{I}_0} \mathbf{E}\bigl[ g_u(
G_{\alpha+\beta+1})\bigr]
\cdot E^{\mu^{\alpha,\beta}} \bigl[F(
\eta_{\mathbb{N} e_1}, \zeta_{\mathbb{N} e_2}, \xi_{\mathbb{N} ^2}) \bigr].\hspace*{-15pt}\nonumber
\end{eqnarray}

To obtain (\ref{ironman}), do induction on the length $n$ of the path.
Let $\mathcal{I}'=\mathcal{I}\cap(x_1+\mathbb{Z} _+^2)$.
In (\ref{muab291}) take
\begin{eqnarray*}
F( \eta_{\mathbb{N} e_1}, \zeta_{\mathbb{N} e_2}, \xi _{\mathbb{N} ^2})
&=& \prod
_{i=1}^{n-1} \pi_{x_i-x_1,x_{i+1}-x_1}(\omega)
\cdot\prod_{k=1}^{n-1} h_{k}(
\check\xi_{x_k-x_1})
\\[-2pt]
&&{}\times\! \prod_{u \in\mathcal{I}'-x_1}
g_{u+x_1}(\check \xi_{u})
\cdot f( \eta_{x_n-x_1+\mathbb{N} e_1},
\zeta_{x_n-x_1+\mathbb
{N} e_2}, \xi _{x_n-x_1+\mathbb{N} ^2}).
\end{eqnarray*}
Assuming (\ref{ironman}) holds for paths of length $n-1$, the
right-hand side
of (\ref{muab291}) turns into the right-hand side of (\ref{ironman}).
\end{pf*}

The ergodicity claim of Theorem~\ref{ethm} is in the next lemma.

%
%
\begin{lemma} With initial distribution $\mu^{\alpha,\beta}$, the
stationary process
$S_n=( \eta_{X_n+\mathbb{N} e_1}, \zeta_{X_n+\mathbb{N} e_2}$,
$\xi_{X_n+\mathbb{N}
^2})$ is ergodic.
\end{lemma}

\begin{pf} Denote a generic state
by $S=( \eta_{\mathbb{N} e_1}, \zeta_{\mathbb{N} e_2}, \xi
_{\mathbb{N} ^2})$.
It suffices to show that, for any function $f\in L^1(\mu^{\alpha,\beta})$,
the averages
\[
n^{-1}\sum_{k=0}^{n-1}
E^S\bigl[f(S_k)\bigr]
\]
converge\vspace*{1pt} to a constant in $L^1(\mu^{\alpha,\beta})$ (\cite{rosenblatt}, pages~91--95).
By approximation in $L^1(\mu^{\alpha,\beta})$, it suffices to prove
this for a local function $f$, that is,
a function of the variables $\mathbf{s}=(\eta_{i,0}, \zeta_{0,j},
\xi
_{i,j})_{i,j\in[M]}$
for an arbitrary but fixed $M\in\mathbb{N} $.
Let $\mathbf{s}=\varphi(S)$ denote the projection mapping, and let
the projection of the
stationary process $S_n$ be
$\mathbf{s}_n=\varphi(S_n)=(\eta_{X_n+(i,0)}, \zeta_{X_n+(0,j)},
\xi
_{X_n+(i,j)})_{i,j\in[M]}$.

Process $\mathbf{s}_n$ is also a stationary Markov chain, with state space
$\mathbb{R} _+^{2M+M^2}$ and invariant distribution $\nu=\mu^{\alpha,\beta
}\circ\varphi^{-1}$. Under
$\nu$
coordinates of $\mathbf{s}$ are independent with distributions
$\eta_{i,0} \sim \operatorname{Gamma}(\alpha$), $\zeta_{0,j} \sim \operatorname{Gamma}(\beta$) and $\xi_{i,j} \sim \operatorname{Gamma}(\rho)$.

Given state $\mathbf{s}=(\eta_{i,0}, \zeta_{0,j}, \xi_{i,j} )_{
i,j\in[M]}$,
we compute the variables
$\{\eta_x, \zeta_x\dvtx  x\in[M]^2\}$ via north--east induction
(\ref{hNE1}). The transition from state $\mathbf{s}$ to a new state
goes by two steps:
(i) randomly shift $\mathbf{s}$ by $e_1$ or $e_2$;
(ii) add fresh variables to the north or east
to replace the variables lost from south or west in the shift of the
$M\times M$ square.

Precisely speaking, from $\mathbf{s}=(\eta_{i,0}, \zeta_{0,j}, \xi
_{i,j})_{ i,j\in[M]}$
the process
jumps to either $\mathbf{t}'$ or $\mathbf{t}''$, according to the
following two cases:
\begin{longlist}[(a)]
\item[(a)] The shift is $e_1$ and
$\mathbf{t}'=(\eta_{i+1,0}, \zeta_{1,j}, \xi_{i+1,j})_{i,j\in[M]}$ where
the new independently chosen variables are $\eta_{M+1,0} \sim \operatorname{Gamma}(\alpha$) and $\xi_{M+1,j} \sim \operatorname{Gamma}(\rho)$ for $j\in[M]$.
\item[(b)] The shift is $e_2$ and
$\mathbf{t}''=(\eta_{i,1}, \zeta_{0,j+1}, \xi_{i,j+1})_{i,j\in
[M]}$ where
the new independently chosen variables are $\zeta_{0,M+1} \sim \operatorname{Gamma}(\beta$) and $\xi_{i,M+1} \sim \operatorname{Gamma}(\rho)$ for $i\in[M]$.
\end{longlist}
The probabilities of the two alternatives are
\[
\pi\bigl(\mathbf{s}, \mathbf{t}'\bigr)=\frac{\eta_{e_1}}{\eta_{e_1}+\zeta_{e_2}} \quad
\mbox{and}\quad \pi\bigl(\mathbf{s}, \mathbf{t}''
\bigr)=\frac{\zeta_{e_2}}{\eta_{e_1}+\zeta_{e_2}}.
\]
Let $\pi(\mathbf{s}, d\mathbf{t})$ denote the transition probability
of the Markov
chain $\mathbf{s}_n$: the shift followed by the random choice of new
coordinates
to complete the square \mbox{$[M]\times[M]$}.
The task is to check that $\mathbf{s}_n$ is an ergodic process.

Two general observations about checking the ergodicity of a Markov
transition $P$ with invariant distribution $\nu$. (i) Suppose
$\nu$ has a density with respect
to a background measure $\lambda$. Then it is
enough to check that, for $\nu$-a.e. $x$, $P(x,dy)$ has a density $p(x,y)$
with respect
to $\lambda(dy)$ such that $p(x,y)>0$ for $\lambda$-a.e. $y$.
For then, if $A$ is a $\nu$-a.s. invariant measurable set such that
$\nu(A^c)>0$, taking $x\in A^c$ in
\[
\mathbf{1}_A(x)= P(x,A)=\int_A p(x,y)
\lambda(dy)
\]
shows that $\lambda(A)=0$ and thereby $\nu(A)=0$.
(ii) It is enough to check the ergodicity of some power $P^m$.

%
%
\begin{figure}[b]

\includegraphics{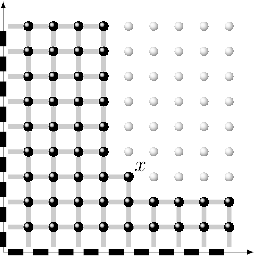}

\caption{The $\sigma$-algebra $\mathcal{H}_x$. The dark black sites
in the interior
and the thickset lines on the axes denote the $\{\xi, \eta, \zeta\}$
variables that generate $\mathcal{H}_x$. The gray lines denote $\{\eta, \zeta
\}$ variables computed via north--east induction from information
contained in $\mathcal{H}_x$. Finally the lighter gray sites denote
$\xi$
variables independent of $\mathcal{H}_x$.}\label{fig2}
\end{figure}

We show that for $m=2M+1$, $\pi^{m}(\mathbf{s}, d\mathbf{t})$ has a
Lebesgue almost everywhere positive density
on $\mathbb{R} _+^{2M+M^2}$.
Let $B$ be a Borel subset
of $\mathbb{R} _+^{2M+M^2}$. Write
%
%
\begin{equation}
T_x\mathbf{s}=\bigl(\eta_{x+(i,0)}, \zeta_{x+(0,j)},
\xi _{x+(i,j)}\dvtx  i,j\in[M]\bigr) \label{Txs}
\end{equation}
for the
shifted configuration in the $M\times M$ square:
%
%
\begin{eqnarray}
\label{muab88} \pi^{m}(\mathbf{s}, B)
&=&\sum
_{x\in\mathbb{Z} _+^2\dvtx  \llvert  x\rrvert _1=m} E^{\mu^{\alpha,\beta}} \bigl[ \mathbf
{1}_B(T_x\mathbf{s}) P^\omega_0
\{X_m=x\} \big\vert\varphi(S_0)=\mathbf{s} \bigr]
\nonumber\\[-8pt]\\[-8pt]
&=&\sum_{x\in\mathbb{Z} _+^2\dvtx  \llvert  x\rrvert _1=m} E^{\mu^{\alpha,\beta}} \bigl[
E^{\mu^{\alpha,\beta}}\bigl\{ \mathbf{1}_B(T_x\mathbf{s})
\vert\mathcal{H}_x \bigr\} P^\omega_0\{
X_m=x\} \big\vert\varphi(S_0)=\mathbf{s} \bigr].\hspace*{-20pt}\nonumber
\end{eqnarray}
On the first line above, $E^{\mu^{\alpha,\beta}}$ represents the choices
of fresh coordinates while the shifts are in the quenched probability
$P^\omega_0\{X_m=x\}$. After that
we conditioned on the $\sigma$-algebra (Figure~\ref{fig2})
%
%
\begin{equation}
\mathcal{H}_x=\sigma \bigl\{\eta_{\mathbb{N} e_1},
\zeta_{\mathbb{N}
e_2}, \xi_x, \{\xi_{i,j}\dvtx  \mbox{$i\le x
\cdot e_1-1$ or $j\le x\cdot e_2-1$} \} \bigr\}.
\label{cHx}
\end{equation}
%

$X_m=x$ implies $\llvert  x\rrvert _1=m$, and then $m=2M+1$ guarantees
that $\mathcal{H}_x$ is large enough to contain the event
$\varphi(S_0)=\mathbf{s}$.
The quenched probability
$P^\omega_0\{ X_m=x\}$ is also $\mathcal{H}_x$-measurable.
Of the variables that make up $T_x\mathbf{s}$ in (\ref{Txs}), the
$\xi
_{x+(i,j)}$'s are independent
of $\mathcal{H}_x$, but the $\eta_{x+(i,0)}$'s and $\zeta_{x+(0,j)}$'s
depend on
$\mathcal{H}_x$ through the equations
%
%
\begin{equation}
\eta_{x+(i,0)}=\xi_{x+(i,0)}\frac{\eta
_{x+(i,-1)}}{\eta_{x+(i,-1)}+\zeta_{x+(i-1,0)}}, \qquad i=1,\ldots,
M \label{muab9}
\end{equation}
and
\[
\zeta_{x+(0,j)}=\xi_{x+(0,j)}\frac{\zeta_{x+(-1,j)}}{\eta
_{x+(0,j-1)}+\zeta_{x+(-1,j)}}, \qquad j=1,
\ldots, M.
\]
%

The situations for $\{\eta_{x+(i,0)}\}$ and $\{\zeta_{x+(0,j)}\}$ are
symmetric,
so let us look at equation (\ref{muab9}) closely. $\mathcal{H}_x$ contains
variables $\{\zeta_x; \eta_{x+(i,-1)}\dvtx  i\in[M]\}$ because these
can be
computed by north--east induction from the variables listed in (\ref
{cHx}), so these
are taken as given in (\ref{muab9}). Variables $\{\xi_{x+(i,0)}\dvtx  i\in
[M]\}$ are picked
i.i.d. $\operatorname{Gamma}(\rho)$, independently of $\mathcal{H}_x$, while
variables $\{\zeta_{x+(i-1,0)}\dvtx  i=2,\ldots,M\}$ are calculated along
the way
from the equations
%
%
\begin{equation}
\qquad \zeta_{x+(i-1,0)}=\xi_{x+(i-1,0)}\frac{\zeta
_{x+(i-2,0)}}{\eta_{x+(i-1,-1)}+\zeta_{x+(i-2,0)}}, \qquad i=2,
\ldots, M. \label{muab96}
\end{equation}
Regarding $\{\zeta_x; \eta_{x+(i,-1)}\dvtx  i\in[M]\}$ as given parameters,
equations (\ref{muab9}) and (\ref{muab96}) show
that the vectors $\bar\eta=(\eta_{x+(i,0)}\dvtx  i\in[M])$ and
$\bar\xi=(\xi_{x+(i,0)}\dvtx  i\in[M])$
in $(0,\infty)^M$
are bijective functions of each other, and these functions are rational
functions with positive coefficients. (The coefficients themselves are
functions of $\{\zeta_x; \eta_{x+(i,-1)}\dvtx  i\in[M]\}$.) Thus\vspace*{2pt} the
Jacobians
of these functions cannot vanish on $(0,\infty)^M$.
Consequently, from the everywhere positive density of
$\bar\xi$ [product of $\operatorname{Gamma}(\rho)$ distributions], we get an everywhere
positive density $f_1$ for $\bar\eta$, for every given value of $\{
\zeta_x; \eta_{x+(i,-1)}\dvtx  i\in[M]\}$.

This argument can be repeated to get an everywhere
positive density $f_2$ for the vector $\bar\zeta=(\zeta_{x+(0,j)}\dvtx
i\in[M])$,
for every given value of the variables specified by the conditioning on
$\mathcal{H}_x$.

Let $f$ denote the (everywhere positive) density of the vector $(\xi
_{x+(i,j)}\dvtx  i,j\in[M])$. With this notation
we can write
\begin{eqnarray*}
&& E^{\mu^{\alpha,\beta}}\bigl[ \mathbf{1}_B(T_x\mathbf{s})
\big\vert \mathcal{H}_x\bigr] = \int_{\mathbb{R} _+^{2M+M^2}}
\mathbf{1}_B(u,v,w) f_1(u) f_2(v) f(w) \,du \,dv
\,dw,
\end{eqnarray*}
where the right-hand side is not a constant, but the densities $f_1$
and $f_2$
depend also on the variables specified by the conditioning on $\mathcal{H}_x$.
On the right the densities are multiplied due to independence that
comes from
dependence on disjoint sets of $\xi$ variables.
This formula can be substituted into (\ref{muab88}) to conclude that
$\pi^m(\mathbf{s}, \cdot)$ has an a.e. positive density on
$(0,\infty
)^{2M+M^2}$.
\end{pf}

\begin{appendix}
\section*{Appendix: Auxiliary results}\label{secapp}
\setcounter{equation}{0}
\setcounter{theorem}{0}

This appendix contains a comparison lemma
for partition functions, a large deviation bound for the log-gamma polymer
and an ergodic theorem for cocycles.

\subsection{Comparison lemma for partition functions}
Let arbitrary weights $\{V_x\}_{x\in\mathbb{Z} _+^2}$ be given,
and define partition functions as in (\ref{hZ}).
For a subset $A\subseteq\Pi_{u,v}$, define the restricted partition
function (unnormalized polymer measure) by
\[
Z_{u,v}(A)= \sum_{x_\centerdot\in A} \prod
_{i=1}^{\llvert  v-u\rrvert _1} V ^{-1}_{x_i}.
\]
%
Recall\vspace*{1pt} the definitions of the exit points (\ref{gxexit})--(\ref{gyexit}).
The restriction $A=\{{t_{e_1}}>0\}$ means that the first step of the
path is $e_1$. In other words,
$Z_{0,x}({t_{e_1}}>0)= V_{e_1}^{-1}Z_{e_1,x}$, defined for
$x\cdot e_1\ge1$.

%
%
\begin{lemma} \label{Zcomp-lm1}
For $m\ge2$ and $n\ge1$
we have this comparison of partition
functions:
%
%
\begin{equation}
\frac{Z_{0,(m-1,n)}({t_{e_1}}>0)}{Z_{0,(m,n)}({t_{e_1}}
>0)} \le \frac{Z_{(1,1), (m-1,n)}}{Z_{(1,1),(m,n)}} \le\frac{Z_{0,(m-1,n)}({t_{e_2}}>0)}{Z_{0,(m,n)}({t_{e_2}}>0)}.
\label
{Zcomp-1}
\end{equation}
\end{lemma}

\begin{pf} Consider the ratio weights for these partition functions:
\begin{eqnarray*}
\eta_x&=&\frac{Z_{0, x-e_1}({t_{e_1}}>0)}{Z_{0,x}({t_{e_1}}>0)} = \frac{Z_{e_1, x-e_1}}{Z_{e_1,x}} \quad
\mbox{and}\quad \tilde\eta_x=\frac{Z_{(1,1), x-e_1}}{Z_{(1,1),x}},
\\
\zeta_x&=&\frac{Z_{0, x-e_2}({t_{e_1}}>0)}{Z_{0,x}({t_{e_1}}>0)} = \frac{Z_{e_1, x-e_2}}{Z_{e_1,x}} \quad
\mbox{and}\quad \tilde\zeta_x=\frac{Z_{(1,1), x-e_2}}{Z_{(1,1),x}}.
\end{eqnarray*}
On the boundary of the lattice $\mathbb{N} ^2$, these ratios
satisfy
\[
\zeta_{1,j}=V_{1,j}=\tilde\zeta_{1,j}\quad
\mbox{and}\quad \eta_{i,1}=V_{i,1}\frac{\eta_{i,0}}{\eta_{i,0}+\zeta_{i-1,1}} <
V_{i,1}=\tilde\eta_{i,1}\qquad\mbox{for } i,j\ge2.
\]
NE induction (\ref{hNE1}) preserves these inequalities and gives
the first inequality of~(\ref{Zcomp-1}). The second comes analogously.
\end{pf}

\subsection{Large deviation bound for the log-gamma polymer}
Let $0<\alpha<\rho$,
and let $(\xi, \eta,\zeta)$ be a gamma system of weights
with parameters $(\alpha,\rho)$ according to Definition~\ref{ga-def}.
Let $Z_{0,v}$ be the partition function defined by
(\ref{gZ}) in this gamma system, with the corresponding point-to-point
quenched polymer measure
\[
Q_{0,v}\{x_\centerdot\}= \frac{1}{Z_{0,v}} \Biggl( \prod
_{i=1}^{t_{\mathrm{exit}}}\tau_{\{x_{i-1}, x _{i}\}
}^{-1}
\Biggr) \Biggl( \prod_{j=t_{\mathrm{exit}}+1}^{\llvert  v\rrvert _1}
\xi_{
x_j}^{-1} \Biggr), \qquad x_\centerdot\in
\Pi_{0,v}.
\]
%
Let the scaling parameter $N\ge1$ be real valued.
Let $(m,n)\in\mathbb{N}^2$ denote the endpoint of the
path. Measure the deviation from characteristic velocity by
%
%
\begin{equation}
\kappa_N=\bigl\llvert m-{N\Psi_1(\rho-\alpha)} \bigr
\rrvert \vee \bigl\llvert n- {N\Psi_1(\alpha)} \bigr\rrvert.
\label{mnN1}
\end{equation}

%
%
\begin{lemma}\label{ldp-lm1}
Let $\kappa_N$ be defined by (\ref{mnN1}). Let $\delta>0$.
Then there are constants
$0<\delta_1, c, c_1<\infty$ such that the following estimate holds.
For $(m,n)\in\mathbb{N}^2$, $N\ge1$ and $u\ge(1\vee c\kappa_N\vee
\delta N)$,
\[
\mathbb{P} \bigl[ Q_{0,(m,n)}\{{t_{e_1}}\ge u\} \ge
e^{-\delta_1 u} \bigr] \le e^{-c_1u}.
\]
%
Same bound holds for ${t_{e_2}}$.
The same constants work for
$(\alpha, \rho)$ that satisfy
$0<\alpha<\rho$ and vary in a compact set.
\end{lemma}

\begin{pf}
Let $\beta<\alpha$, and take two gamma systems:
$(\xi, \eta^\alpha, \zeta^\alpha)$ with parameters $(\alpha,\rho)$
and
$(\xi, \eta^\beta, \zeta^\beta)$ with parameters $(\beta,\rho)$.
Couple them so
that they share the $\xi$-variables, and $\eta^\beta_x\le\eta
^\alpha_x$
and $\zeta^\beta_x\ge\zeta^\alpha_x$ hold. This can be achieved by
imposing these same conditions on the variables in part (c) of
Definition~\ref{ga-def},
and then noting that the inequalities are preserved by (\ref{NE}).
Let $Z^\alpha$ and $Z^\beta$ be partition functions computed in these two
systems:
%
%
\begin{eqnarray}
\label{aux800} &&Q_{0,(m,n)}\{{t_{e_1}}\ge u\}\nonumber
\\
&&\qquad = \frac{1}{Z_{0,(m,n)}^\alpha} \sum_{x_\centerdot\in\Pi_{0,(m,n)}} \mathbf{1}
\{{t_{e_1}}\ge u\} \Biggl( \prod_{i=1}^{t_{\mathrm{exit}}}
\frac{1}{\eta^\alpha
_{i,0}} \Biggr) \Biggl( \prod_{j=t_{\mathrm{exit}}+1}^{m+n}
\xi_{ x_j}^{-1} \Biggr)
\\
&&\qquad \le\frac{Z_{0,(m,n)}^\beta}{Z_{0,(m,n)}^\alpha} \cdot\prod_{i=1}^{\lfloor{u}\rfloor}
\frac{\eta^\beta
_{i,0}}{\eta^\alpha_{i,0}}.\nonumber
\end{eqnarray}

In the bounds below, $\overline X=X-\mathbb{E} X$ denotes a centered
random variable.
Recall the mean (\ref{ElogZ}).
Let $\delta_1>0$. From (\ref{aux800})
%
%
\begin{eqnarray}
\label{app8}
&&  \mathbb{P} \bigl[ Q_{0,(m,n)}\{{t_{e_1}}\ge u\}
\ge e^{-\delta_1u} \bigr]\nonumber 
\\
&&\qquad \le \mathbb{P} \Biggl\{ \sum_{i=1}^{\lfloor{u}\rfloor}
\bigl( \overline{\log\eta ^\beta_{i,0}} - \overline{\log
\eta^\alpha_{i,0}} \bigr) \ge\delta_1 u \Biggr\}
\nonumber\\[-8pt]\\[-8pt]
&&\quad\qquad{}+ \mathbb{P} \bigl\{ \overline{ \log Z_{0,(m,n)}^\beta }-
\overline {\log Z_{0,(m,n)}^\alpha} \ge\bigl(\lfloor{u}\rfloor-m
\bigr) \bigl(\Psi_0(\alpha)-\Psi_0(\beta)\bigr)\nonumber
\\
&&\hspace*{105pt}\quad\qquad{} +n\bigl(\Psi_0(\rho-\beta)-\Psi
_0(\rho-\alpha)\bigr) -2\delta_1u \bigr\}.\nonumber
\end{eqnarray}

Standard large deviations apply to log-gamma variables, so
$\exists c_2>0$ such that
\[
\mathbb{P} \Biggl\{ \sum_{i=1}^{\lfloor{u}\rfloor}
\bigl( \overline{\log\eta ^\beta_{i,0}} - \overline{\log
\eta^\alpha_{i,0}} \bigr) \ge\delta_1 u \Biggr\}
\le e^{-c_2 u}.
\]
%
Taylor expand to second order the $\Psi_0$-differences
inside
the last probability in~(\ref{app8}). Keeping $\delta>0$ fixed,
pick $\delta_1>0$ and $\alpha-\beta>0$ small enough
and $c<\infty$ large enough. Then for another small constant $c_3>0$,
the probability
simplifies to
%
\[
\mathbb{P} \bigl\{ \overline{ \log Z_{0,(m,n)}^\beta}-\overline
{\log Z_{0,(m,n)}^\alpha} \ge c_3u \bigr\} \le
e^{-c_4u}.
\]
%
The bound comes again from i.i.d. large deviations, by virtue of
(\ref{app4}).
\end{pf}

%
%
%
%
%

\subsection{Ergodic theorem for centered cocycles}

With a bit of extra effort and with future use in mind, we prove
this ergodic theorem more generally than required for this paper. Fix a
dimension
$d\in\mathbb{N} $.
Let $\mathcal{R}\subset\mathbb{Z} ^d$ denote an arbitrary finite
set of admissible steps that contains at least one nonzero point. $0\in
\mathcal{R}$ is also acceptable. Admissible paths $(x_k)_{k=0}^n$ satisfy
$x_k-x_{k-1}\in\mathcal{R}$. Let $M=|\mathcal{R}|$ be the
cardinality of $\mathcal{R}$.

Define
\[
\mathcal{G}^+= \biggl\{ \sum_{z\in\mathcal{R}}
b_z z\dvtx  b_z\in\mathbb {Z} _+ \biggr\},
\]
and let $\mathcal{G}=\mathcal{G}^+-\mathcal{G}^+$ be the additive
subgroup of $\mathbb{Z} ^d$
generated by $\mathcal{R}$.
Let
$(\Omega,\mathfrak{S},\mathbb{P})$ be a probability space equipped
with a semigroup $(T_x)_{x\in\mathcal{G}^+}$ of commuting
measurable maps $T_x\dvtx \Omega\to\Omega$. In other words, the
assumptions are that $T_0=\operatorname{id}$ and $T_{x+y}=T_x\circ
T_y$ for $x,y\in
\mathcal{G}^+$. Generic points
of $\Omega$ are denoted by $\omega$.
Assume $\mathbb{P} $ invariant and ergodic under $(T_x)_{x\in\mathcal
{G}^+}$:  that is,
$\mathbb{P} \circ T_x^{-1}=\mathbb{P} $, and if $T_x^{-1}A=A$
$\forall x\in\mathcal{G}^+$
then $\mathbb{P} (A)\in\{0,1\}$.

%

Let $F\dvtx \Omega\times\mathcal{R}\to\mathbb{R}$ be a centered
cocycle, by which
we mean these properties:
\begin{longlist}[(ii)]
\item[(i)] $\forall z\in\mathcal{R}\dvtx  F(\omega,z)\in L^1(\mathbb{P})$
and $\mathbb{E} F(\omega,z)=0$.
\item[(ii)] The closed-loop (or cocyle) property:
if $\{x_k\}_{k=0}^n$ and $\{x'_\ell\}_{\ell=0}^m$ are two admissible
paths such
that $x_0=x'_0$ and $x_n=x'_m$, then
\[
\sum_{k=0}^{n-1} F(T_{x_k}\omega,
x_{k+1}-x_k) = \sum_{\ell
=0}^{m-1}
F\bigl(T_{x'_\ell}\omega, x'_{\ell+1}-x'_\ell
\bigr).
\]
\end{longlist}
Note that the closed-loop property forces $F(\omega,0)=0$ if $0\in
\mathcal{R}$.

Define the path integral of $F$ for $(\omega,x)\in\Omega\times
\mathcal{G}$ by
\[
f(\omega,x) = \sum_{k=0}^{n-1}F(T_{k_i}
\omega,x_{k+1}-x_k) - \sum_{\ell=0}^{m-1}F
\bigl(T_{x'_\ell}\omega,x'_{\ell+1}-x'_\ell
\bigr),
\]
where $(x_k)_{k=0}^n$ and $(x'_\ell)_{\ell=0}^m$ are any two\vspace*{1.5pt}
admissible paths from a common initial point $x_0=x'_0$ to $x_n = x$
and to $x'_m=0$.
In particular $f(\omega,0)=0$.
The closed-loop property ensures that $f$ is well defined.

Let $D_n=\{x\dvtx \exists z_1,\ldots,z_n\in\mathcal{R}\mbox{ such that
}z_1+\cdots+z_n=x\}$ denote the set of points accessible from $0$ in
exactly $n$ steps.

%
%
\begin{theorem}\label{erg-thm1}
Let $F$ be a centered cocycle. Assume there exists a function $\overline F\dvtx \Omega\times\mathcal{R}\to\mathbb{R}$ such that $F(\omega,z)\le
\overline F(\omega,z)$
for all $z\in\mathcal{R}$ and $\mathbb{P} $-almost every $\omega$,
and that satisfies
%
%
\begin{equation}
\mathop{\operatorname{\overline{\lim}}}_{\delta\searrow0} \mathop{\operatorname{
\overline{\lim}}}_{n\to\infty} \max_{\llvert  x\rrvert _1 \le n}
\frac{1}n \sum_{0\le i\le n\delta} \bigl\llvert \overline
F(T_{x+iz}\omega, z)\bigr\rrvert = 0 \qquad\forall z\in \mathcal{R}
\setminus \{0\}. \label{assL}
\end{equation}
Then for $\mathbb{P} $-almost every $\omega$
\[
\lim_{n\to\infty}\max_{x\in D_n}\frac{\llvert  f(\omega,x)\rrvert }{n}=0.
\]
%
\end{theorem}

An assumption similar to (\ref{assL}) was useful in
\cite{rass-sepp-p2p,rass-sepp-yilm} in studies of polymers.
If, for each $z\in\mathcal{R}\setminus\{0\}$,
the variables $\{ \overline F( T_{iz}\omega, z)\}_{i\in\mathbb{Z} _+}$ are
i.i.d. then by
Lemma~A.4
of \cite{rass-sepp-yilm}
a sufficient condition for (\ref{assL}) is
%
%
\begin{equation}
\exists p>d\dvt\qquad \mathbb{E} \bigl[ \bigl\llvert \overline F(\omega,z)\bigr\rrvert
^p \bigr] <\infty. \label{Lmom}
\end{equation}
Our application of Theorem~\ref{erg-thm1} is to the centered cocycle
$F^\mathbf{u}$
in (\ref{Fxi}). By Corollary~\ref{buse-cor} this satisfies the
i.i.d. condition and even has
an exponential moment. Thus hypothesis (\ref{assL}) is satisfied
by $F^\mathbf{u}$ in (\ref{Fxi}), and Theorem~\ref{erg-thm1} holds for
$F=\overline F=F^\mathbf{u}$.


As an auxiliary result toward the main theorem, we prove a limit
for averages over rectangles of any dimension. The following result is
a discrete version of
Lemma 6.1 of \cite{kosy-vara-08}.

%
%
\begin{theorem}\label{erg-thm2} Let $F$ be a centered cocycle. Let
$r\in[M]=\{1,2,\ldots, M\}$, $z_1,\ldots,z_r$ distinct points from
$\mathcal{R}$,
and $0\le a_i<b_i$ for $1\le i\le r$.
Then
%
%
\begin{equation}
\qquad\lim_{n\to\infty}\frac{1}{n^r}\sum
_{k_1=\lfloor{na_1}\rfloor
}^{\lfloor{nb_1}\rfloor-1}\cdots\sum_{k_r=\lfloor{na_r}\rfloor
}^{\lfloor{nb_r}\rfloor-1}
\frac{f(\omega,k_1 z_1+\cdots+k_r z_r)}{n} = 0 \qquad\mbox{$\mathbb{P} $-a.s.} \label{erg3}
\end{equation}
\end{theorem}


It is enough to consider the case $a_i=0$, for the general case
is obtained by successive differences and sums of such cases.
Then to simplify notation we take $b_i=1$.
We separate a part of the proof as a lemma.

%
%
\begin{lemma} 
Let $1\le j\le r\le M$ and
$g\dvtx [0,1]^r\to\mathbb{R}$ continuous. Then $\mathbb{P} $-a.s.
%
%
\begin{equation}
\qquad \lim_{n\to\infty}\frac{1}{n^r}\sum
_{k_1=0}^{n-1}\cdots\sum_{k_r=0}^{n-1}g
\bigl(n^{-1}(k_1,\ldots,k_r)
\bigr)F(T_{k_1 z_1+\cdots
+k_r z_r}\omega,z_j) = 0. \label{aslan}
\end{equation}
%
\end{lemma}

\begin{pf}
Fix $j$. Let $h(\omega)$ denote the a.s. limit of the left-hand side of
(\ref{aslan}) for
$g\equiv1$, given by the
pointwise ergodic theorem (\cite{kren}, Theorem~6.2.8). 
We show that $h$ is invariant under each shift $T_z$, $z\in\mathcal{R}$.
By the closed-loop property (now for $j\in\{1,\ldots,r\}$)
\begin{eqnarray*}
&&\frac{1}{n^r}\sum_{k_1=0}^{n-1}\cdots
\sum_{k_r=0}^{n-1} F(T_{k_1z_1
+ k_2z_2 + \cdots+ k_rz_r}\omega,
z_j)
\\
&&\quad{}+ \frac{1}{n^r}\sum_{k_1=0}^{n-1}
\cdots\sum_{k_{j-1}=0}^{n-1} \sum
_{k_{j+1}=0}^{n-1}\cdots\sum_{k_r=0}^{n-1}
F(T_{k_1z_1 + k_2z_2 + \cdots+ nz_j + \cdots+ k_rz_r}\omega, z)
\\
&&\qquad= \frac{1}{n^r}\sum_{k_1=0}^{n-1}
\cdots\sum_{k_{j-1}=0}^{n-1} \sum
_{k_{j+1}=0}^{n-1}\cdots\sum_{k_r=0}^{n-1}
F(T_{k_1z_1 + k_2z_2 + \cdots+ 0\cdot z_j + \cdots+ k_rz_r}\omega, z)
\\
&&\quad\qquad{}+ \frac{1}{n^r}\sum_{k_1=0}^{n-1}
\cdots\sum_{k_r=0}^{n-1} F\bigl(T_{k_1z_1 + k_2z_2 + \cdots+ k_rz_r}(T_{z}
\omega), z_j\bigr).
\end{eqnarray*}
The\vspace*{1pt} closed loop above is taken for fixed
$T_{k_1z_1 + \cdots+ k_{j-1}z_{j-1} +k_{j+1}z_{j+1} +\cdots+
k_rz_r}\omega$.
The two paths are $\{z_j, 2z_j,\ldots, nz_j, nz_j+z\}$
and $\{z, z+z_j, z+2z_j,\ldots, z+nz_j\}$.

The first sum converges to $h(\omega)$, the last one to
$h(T_{z}\omega)$.
By the pointwise ergodic theorem the first sum on the right converges to
$0$ because it has only $n^{r-1}$ terms. Consequently all terms converge
a.s.
The second sum on the left must also vanish in the limit because
it converges to zero in probability.
We get $h(\omega) = h(T_{z}\omega)$ $\forall z\in\mathcal{R}$ and conclude
by ergodicity
and the mean-zero property of~$F$ that $h=0$.
Then (\ref{aslan}) follows by a Riemann sum-type approximation.
\end{pf}

\begin{pf*}{Proof of Theorem~\ref{erg-thm2}}
This goes by induction on $r$.
For $r=1$, rearrange
\begin{eqnarray*}
\frac{1}{n}\sum_{k=0}^{n-1}
\frac{f(\omega, kz_1)}n &=& \frac{1}{n}\sum_{k=1}^{n-1}
\frac{1}{n} \sum_{i=0}^{k-1}F(T_{iz_1}
\omega,z_1)
= \frac{1}{n}\sum_{k=0}^{n-1}
\biggl(1-\frac{k+1}{n} \biggr)F(T_{kz_1}\omega,z_1).
\end{eqnarray*}
An application of (\ref{aslan}) with $g(y)=1-y$ gives conclusion
(\ref{erg3})
for $r=1$.

Suppose that (\ref{erg3}) holds for some $r\in\{1,\ldots,M-1\}$. Let
us show it for $r+1$.
\begin{eqnarray*}
&&\frac{1}{n^{r+1}}\sum_{k_1=0}^{n-1}\cdots
\sum_{k_r=0}^{n-1}\sum
_{k_{r+1}=0}^{n-1}\frac{f(\omega,k_1 z_1+\cdots
+k_rz_r+k_{r+1}z_{r+1})}{n}
\\
&&\qquad = \frac{1}{n^{r+2}}\sum
_{k_1=0}^{n-1}\cdots\sum_{k_r=0}^{n-1}
\sum_{k_{r+1}=0}^{n-1} \bigl[f(
\omega,k_1 z_1+\cdots +k_rz_r)
\\
&&\hspace*{140pt}{}+ f(T_{k_1 z_1+\cdots+k_rz_r}\omega,
k_{r+1}z_{r+1}) \bigr]
\\
&&\qquad= \frac{1}{n^r}\sum_{k_1=0}^{n-1}
\cdots\sum_{k_r=0}^{n-1}\frac{f(\omega,k_1 z_1+\cdots+k_rz_r)}{n}
\\
&&\quad\qquad {} + \frac{1}{n^{r+2}}\sum
_{k_1=0}^{n-1} \cdots\sum_{k_{r+1}=0}^{n-1}
f(T_{k_1 z_1+\cdots+k_rz_r}\omega, k_{r+1}z_{r+1})
\\
&&\qquad = \frac{1}{n^r}\sum_{k_1=0}^{n-1}
\cdots\sum_{k_r=0}^{n-1}\frac{f(\omega,k_1 z_1+\cdots+k_rz_r)}{n}
\\
&&\quad\qquad{} + \frac{1}{n^{r+1}}\sum_{k_1=0}^{n-1}
\cdots\sum_{k_{r+1}=0}^{n-1} \biggl(1-
\frac{k_{r+1}+1}{n} \biggr)F(T_{k_1
z_1+\cdots+k_{r+1}z_{r+1}}\omega,z_{r+1}).
\end{eqnarray*}

As $n\to\infty$ on the last line, the first sum goes to zero by the
induction hypothesis and the second sum by (\ref{aslan}) with
$g(y_1,\ldots, y_{r+1})=1-y_{r+1}$.
\end{pf*}

\begin{pf*}{Proof of Theorem~\ref{erg-thm1}}
Fix a labeling $z_1,\ldots,z_M$ of the steps in $\mathcal{R}$.
We first prove
%
%
\begin{equation}
\mathop{\operatorname{\underline{\lim}}}_{n\to\infty}\min
_{x\in
D_n}\frac{f(\omega,x)}{n} \ge 0.\label{supper}
\end{equation}
%

Let $\delta>0$ and $a_k=k\delta/(4M)$ for $k\in\mathbb{Z} _+$.
For $\mathbf{k}=(k_1,\ldots, k_M)\in\mathbb{Z} _+^M$ define sets
\[
B_{n,\mathbf{k}}= \Biggl\{ \sum_{i=1}^M
s_i z_i\dvtx  \lfloor {na_{k_i}}\rfloor\le
s_i<\lfloor{na_{k_i+1}}\rfloor\mbox{ for } i\in[M]
\Biggr\}.
\]
%
For each $x\in D_n$ we can pick
$B_{n,x}=B_{n,\mathbf{k}(x)}$ such that every point $y\in B_{n,x}$
can be reached from $x$ with an admissible path of at most $n\delta$ steps.\vspace*{2pt}
(The assumption
$x\in D_n$ implies $x=\sum_{i=1}^M b_i z_i$ with $\sum_{i=1}^M b_i=n$.
For each $i$ take $k_i$ minimal such that $\lfloor{n a_{k_i}}\rfloor
\ge b_i$.)
Our strategy is to replace $f(\omega,x)$ with an average of $f$ over~$B_{n,x}$.
Note that there is a fixed finite set $K$ of vectors $\mathbf{k}$
such that the above choices can be made from $\{ B_{n,\mathbf{k}}\dvtx
\mathbf{k}
\in K\}$
for all large enough $n$ and all $x\in D_n$. 

For every $x\in D_n$ and every $y\in B_{n, x}$ fix a path
from $x$ to $y$ such that the steps $z_1,z_2,\ldots, z_M$ are taken in
order. Recall that $F(\omega,0)=0$.
Then for any such pair $x,y$, with designated path $(x_i)_{i=0}^m$,
\begin{eqnarray*}
f(\omega,x) &=& f(\omega,y) - \sum_{i=0}^{m-1}F(T_{x_i}
\omega,x_{i+1}-x_i)\mathbf {1}\{ x_{i+1}\ne
x_i\}
\\
&\ge& f(\omega,y) - \sum_{i=0}^{m-1}
\overline F(T_{x_i}\omega,x_{i+1}-x_i)\mathbf{1}
\{x_{i+1}\ne x_i\}
\\
&\ge& f(\omega,y) - \sum_{z\in\mathcal{R}\setminus\{0\}}
\biggl\{ \max_{\llvert  u\rrvert _1\le
2nr_0} \sum_{0\le i\le n\delta}
\bigl\llvert \overline F(T_{u+iz}\omega, z)\bigr\rrvert \biggr\}.
\end{eqnarray*}
Above $r_0=\max\{\llvert  z\rrvert _1\dvtx z\in\mathcal{R}\}$.
The error term is independent of $x,y$.
Average over $y\in B_{n,x}$, and then take minimum over $x\in D_n$,
\begin{eqnarray*}
\min_{x\in D_n}\frac{f(\omega,x)}{n} &\ge& \min_{\mathbf{k}\in K}
\frac{1}{N_{n,\mathbf{k}}}\sum_{s_1=\lfloor
{na_{k_1}}\rfloor}^{\lfloor{na_{k_1+1}}\rfloor-1}\cdots
\sum_{s_M=\lfloor{na_{k_M}}\rfloor}^{\lfloor{na_{k_M+1}}\rfloor-1}\frac
{f(\omega,s_1 z_1+\cdots+s_M z_M)}{n}
\\
&&{} - \sum_{z\in\mathcal{R}\setminus\{0\}} \biggl
\{ \max_{|u|_1\le2nr_0} \frac{1}n \sum
_{0\le i\le n\delta} \bigl\llvert \overline F(T_{u+iz}\omega, z)
\bigr\rrvert \biggr\},
\end{eqnarray*}
where $N_{n,\mathbf{k}}=\prod_{i=1}^M(\lfloor{na_{k_i+1}}\rfloor
-\lfloor{na_{k_i}}\rfloor)\sim Cn^M$.
As $n\to\infty$, the first term on the
right vanishes by Theorem~\ref{erg-thm2}. After that let $\delta\to0$,
and assumption (\ref{assL}) takes care of the last term.
Bound (\ref{supper}) has been verified. 

To prove
%
%
\begin{equation}
\mathop{\operatorname{\overline{\lim}}}_{n\to\infty}\max_{x\in
D_n}
\frac{f(\omega,x)}{n} \le 0\label{slower},
\end{equation}
we repeat the argument but with more rectangles.

For $\varnothing\ne I\subset[M]$ and $\mathbf{k}=(k_i)_{i\in I
}\subset\mathbb{Z} _+^{|I|}$, define
\[
B_{n,I,\mathbf{k}}= \biggl\{ \sum_{i\in I}
s_i z_i\dvtx  \lfloor {na_{k_i}}\rfloor\le
s_i<\lfloor{na_{k_i+1}}\rfloor\mbox{ for } i\in I
\biggr\}.
\]
%
For each $x\in D_n$ pick $B_{n,x}=B_{n,I(x),\mathbf{k}(x)}$ so that
$x$ can be reached from every point $y\in B_{n,x}$ with an admissible
path of at most $n\delta$ steps. The additional flexibility of choice
of $I(x)$ accommodates points $x=\sum_{i=1}^M b_iz_i$ such that
some $b_i< \lfloor{na_1}\rfloor$ and therefore a rectangle
$B_{n,\mathbf{k}}$ that
uses all $M$ steps cannot be placed ``upstream'' from $x$. As before,
there is a fixed finite set $K$ from which all the vectors $\mathbf{k}(x)$
can be chosen, for all $x\in D_n$ and large enough $n$.

For every $x\in D_n$ and $y\in B_{n, x}$ fix a path
from $y$ to $x$ such that the steps $z_j$, $j\in I(x)$, are taken
in order.
Then for any such pair $x,y$, with designated path~$(x_i)_{i=0}^m$,
\begin{eqnarray*}
f(\omega,x) &=& f(\omega,y) + \sum_{i=0}^{m-1}F(T_{x_i}
\omega,x_{i+1}-x_i)\mathbf {1}\{ x_{i+1}\ne
x_i\}
\\
&\le& f(\omega,y) + \sum_{i=0}^{m-1}
\overline F(T_{x_i}\omega,x_{i+1}-x_i)\mathbf{1}
\{x_{i+1}\ne x_i\}
\\
&\le& f(\omega,y) + \sum_{z\in\mathcal{R}\setminus\{0\}} \biggl\{ \max
_{|u|_1\le
2nr_0} \sum_{0\le i\le n\delta} \bigl\llvert
\overline F(T_{u+iz}\omega, z)\bigr\rrvert \biggr\}.
\end{eqnarray*}
Again average over $y\in B_{n, x}$ to obtain
\begin{eqnarray*}
&& \max_{x\in D_n}\frac{f(\omega,x)}{n}
\\
&&\qquad \le\mathop{\max_{\mathbf{k}\in K}}_{\varnothing\ne I\subset
[M]}\frac{1}{N_{n,I,\mathbf{k}}}
\sum_{s_{j_1}=\lfloor
{na_{k_{j_1}}}\rfloor}^{\lfloor{na_{k_{j_1}+1}}\rfloor-1}\cdots \sum
_{s_{j_{\llvert  I\rrvert }}=\lfloor{na_{k_{j_{\llvert  I\rrvert }}}}\rfloor}^{\lfloor{na_{k_{j_{\llvert  I\rrvert }+1}}}\rfloor-1} \frac{f(\omega,s_{j_1} z_{j_1}+\cdots+s_{j_{\llvert  I\rrvert }}
z_{j_{\llvert  I\rrvert }})}{n}
\\
&&\quad\qquad{} + \sum_{z\in\mathcal{R}\setminus\{0\}} \biggl\{ \max
_{|u|_1\le
2nr_0} \frac{1}n \sum_{0\le i\le n\delta}
\bigl\llvert \overline F(T_{u+iz}\omega, z)\bigr\rrvert \biggr\},
\end{eqnarray*}
where $N_{n,I,\mathbf{k}}\sim C n^{\llvert  I\rrvert }$ and $I=\{ j_1,
\ldots,
j_{\llvert  I\rrvert }\}$.
Bound (\ref{slower}) follows as above.
\end{pf*}
\end{appendix}




\printaddresses
\end{document}